\documentclass[12pt, reqno]{amsart}
\usepackage{amssymb,amsthm,amsfonts,amsmath, amscd}

\usepackage{hyperref}
\usepackage{mathrsfs}

\newcommand\R{\mathbb R}
\newcommand\Z{\mathbb Z}
\newcommand\C{\mathbb C}
\newcommand\T{\mathbb T}
\newcommand\gt{\mathbb{GT}}

\newcommand\la{\lambda}
\newcommand\La{\Lambda}
\newcommand\epsi{\varepsilon}

\newcommand\D{\mathcal D}

\newcommand\A{E}
\newcommand\wt{\widetilde}

\newcommand\X{\mathfrak X}
\newcommand\Y{\mathfrak Y}

\newcommand\Prob{\operatorname{Prob}}
\newcommand\di{\operatorname{Dim}}
\newcommand\Ex{\operatorname{Ex}}
\newcommand\const{\operatorname{const}}

\newcommand\al{\alpha}
\newcommand\be{\beta}
\newcommand\ga{\gamma}
\newcommand\de{\delta}
\newcommand\Om{\Omega}
\newcommand\om{\omega}
\newcommand\DD{\mathbb D}
\newcommand\pd{\partial}

\newtheorem{theorem}{Theorem}[section]
\newtheorem{proposition}[theorem] {Proposition}

\newtheorem{corollary}[theorem]{Corollary}
\newtheorem{lemma}[theorem]{ Lemma}

\theoremstyle{definition}
\newtheorem{definition}[theorem]{Definition}
\newtheorem{remark}[theorem]{Remark}

\numberwithin{equation}{section}

\begin{document}

\title[Markov processes \dots]{Markov processes on the path space of the Gelfand-Tsetlin graph
and on its boundary}

\author{Alexei Borodin}
\address{California Institute of Technology; Massachusetts Institute of Technology;
Institute for Information Transmission Problems, Russian Academy of Sciences}
\email{borodin@caltech.edu}

\author{Grigori Olshanski}
\address{Institute for Information Transmission Problems, Bolshoy
Karetny 19,  Moscow 127994, Russia; Independent University of Moscow, Russia}
\email{olsh2007@gmail.com}

\begin{abstract} We construct a four-parameter family of Markov processes on
infinite Gelfand-Tsetlin schemes that preserve the class of central (Gibbs) measures.
Any process in the family induces a Feller Markov process on the infinite-dimensional
boundary of the Gelfand-Tsetlin graph or, equivalently, the space of extreme characters
of the infinite-dimensional unitary group $U(\infty)$. The process has a unique invariant
distribution which arises as the decomposing measure in a natural problem of harmonic
analysis on $U(\infty)$ posed in \cite{Ols03}. As was shown in \cite{BO05a}, this measure
can also be described as a determinantal point process with a correlation kernel expressed
through the Gauss hypergeometric function.
\end{abstract}

\maketitle

\tableofcontents

\section{Introduction}

This work is a result of interaction of two circles of ideas. The first one
deals with a certain class of random growth models in two space dimensions
\cite{War07}, \cite{Nor10}, \cite{BF08+}, \cite{BG09}, \cite{BGR09+},
\cite{BK10} \cite{Bor10+}, while the second one addresses constructing and
analyzing stochastic dynamics on spaces of point configurations with
distinguished invariant measures that are often given by, or closely related
to, determinantal point processes \cite{BO06a}, \cite{BO06b}, \cite{BO09},
\cite{Ols10},  \cite{Ols10+}.

Our main result is a construction of a Feller Markov process that preserves the
so-called $zw$-measure on the (infinite-dimensional) space $\Omega$ of extreme
characters of the infinite-dimensional unitary group $U(\infty)$. The
four-parameter family of $zw$-measures arises naturally in a problem of
harmonic analysis on $U(\infty)$ as the decomposing measures for a distinguished
family of characters \cite{Ols03}. A $zw$-measure gives rise to a determinantal
point process on the real line with two punctures, and the corresponding
correlation kernel is given in terms of the Gauss hypergeometric function
\cite{BO05a}. Such point processes degenerate, via suitable limits and/or
specializations, to essentially all known one-dimensional determinantal
processes with correlation kernels expressible through classical special
functions.

The problem of constructing a Markov process that preserves a given
determinantal point process with infinite point configurations has been
addressed in \cite{Spo87}, \cite{KT10}, \cite{Osa09+} for the sine process, in
\cite{KT09} for the Airy process, and in \cite{Ols10+} for the Whittaker
process describing the $z$-measures from the harmonic analysis on the infinite
symmetric group.

Our approach to constructing the infinite-dimensional stochastic dynamics
differs from the ones used in previous papers. We employ the fact (of
representation theoretic origin) that the probability measures on $\Omega$ are
in one-to-one correspondence with {\it central\/} or {\it Gibbs\/} measures on
infinite Gelfand-Tsetlin schemes that can also be viewed as stepped surfaces or
lozenge tilings of a half-plane. The projections of a $zw$-measure to
suitably defined slices of the infinite schemes yield {\it orthogonal
polynomial ensembles\/} with weight functions corresponding to hypergeometric
{\it Askey-Lesky orthogonal polynomials\/}.

These orthogonal polynomials are eigenfunctions for a birth and death process
on $\Z$ with quadratic jump rates; a standard argument then shows that the
$N$-dimensional Askey-Lesky orthogonal polynomial ensemble is preserved by a
Doob's $h$-transform of $N$ independent birth and death processes.

We further show that the Markov processes on the slices are {\it consistent\/}
with respect to stochastic projections of the $N$th slice to the $(N-1)$st one
(these projections are uniquely determined by the Gibbs property). This
consistency is in no way obvious, and we do not have a conceptual explanation
for it. However, it turns out to be essentially sufficient for defining the
corresponding Markov process on $\Omega$.

We do a bit more --- using a continuous time analog of the general formalism of
\cite{BF08+} (which was based on an idea from \cite{DF90}), we construct a
Markov process on Gelfand-Tsetlin schemes that preserves the class of central
(=Gibbs) measures and that induces the same Markov process on $\Omega$.

We now proceed to a more detailed description of our work.

\subsection{Gelfand-Tsetlin graph and its boundary}
Following \cite{Wey39}, for $N\ge 1$ define a {\it signature\/} of length $N$ as an $N$-tuple of
nonincreasing integers $\la=(\la_1\ge\dots\ge \la_N)$, and denote by $\gt_N$
the set of all such signatures. Elements of $\gt_N$ parameterize irreducible
representations of $U(N)$ or $GL(N,\C)$, and they are often called {\it highest
weights\/}.

For $\la\in\gt_{N}$ and $\nu\in \gt_{N+1}$, we say that $\la\prec\nu$ if
$\nu_{j+1}\le \la_j\le\nu_j$ for all meaningful values of indices. These
inequalities are well-known to be equivalent to the condition that the
restriction of the $\nu$-representation of $U(N+1)$ to $U(N)$ contains a
$\la$-component.

Set $\gt=\bigsqcup_{N\ge 1} \gt_N$, and equip $\gt$ with edges by joining $\la$
and $\nu$ iff $\la\prec\nu$ or $\nu\prec\la$. This turns $\gt$ into a graph
that we call the {\it Gelfand-Tsetlin graph\/}. A path of length
$M\in\{1,2,\dots\}\cup\{\infty\}$ in $\gt$ is a length $M$ sequence
$$
\la^{(1)}\prec\la^{(2)}\prec\dots,\qquad \la^{(j)}\in \gt_j.
$$
Equivalently, such a path can be viewed as an array of numbers
$\bigl\{\la^{(j)}_i\bigr\}$ satisfying the inequalities $\la^{(j+1)}_{i+1}\le
\la_i^{(j)}\le \la^{(j+1)}_i$; it is also called a {\it Gelfand-Tsetlin
scheme\/}. An interpretation of paths in $\gt$ in terms of lozenge tilings or
stepped surfaces can be found in the introduction to \cite{BF08+}.

The Gelfand-Tsetlin schemes of length $N$ parameterize basis vectors in the
{\it Gelfand-Tsetlin basis\/} of the irreducible representation of $U(N)$
corresponding to $\la^{(N)}$, cf. \cite{Zhe70}. Denote by $\di_N\la$ the number of such schemes
with $\la^{(N)}=\la$; this is also the dimension of the irreducible
representation of $U(N)$ corresponding to $\la$. It is essentially equal to the
Vandermonde determinant in shifted coordinates of $\la$:
$$
\di_N(\la)=\const_N\,\prod_{1\le i<j\le N} (\la_i-i-\la_j+j).
$$

A probability measure on infinite paths in $\gt$ is called {\it central\/} (or
Gibbs) if any two finite paths with the same top end are equiprobable, cf.
\cite{Ker03}. Let $P_N$ be the projection of such a measure to
$\la^{(N)}\in\gt_N$. Centrality is easily seen to be equivalent to the relation
$\mu_N=\mu_{N+1} \Lambda^{N+1}_N$, $N\ge 1$, where $\mu_N$ and $\mu_{N+1}$ are
viewed as row-vectors with coordinates $\{\mu_N(\la)\}_{\la\in\gt_N}$ and
$\{\mu_{N+1}(\nu)\}_{\nu\in\gt_{N+1}}$, and
\begin{equation}\label{i.1}
\Lambda^{N+1}_N(\nu,\la)=\frac{\di_N(\la)}{\di_N(\nu)}\,\text{\bf
1}_{\la\prec\nu}\,,\qquad \la\in\gt_N,\ \nu\in\gt_{N+1},
\end{equation}
is the stochastic matrix of {\it cotransition probabilities\/}. There is a
one-to-one correspondence between central measures on $\gt$ and characters of
$U(\infty)$ (equivalently, equivalence classes of unitary spherical
representations of the Gelfand pair $(U(\infty)\times
U(\infty),\operatorname{diag}U(\infty))$), see \cite{Ols03}.

As shown in \cite{Ols03}, see also \cite{Voi76}, \cite{VK82}, \cite{OO98},
the space of all central probability measures is
isomorphic to the space of all probability measures on the set
$\Omega\subset\R_+^{4\infty+2}$ consisting of the sextuples
$\omega=(\alpha^+,\beta^+,\alpha^-,\beta^-,\delta^+,\delta^-)\in
\R_+^{4\infty+2}$ satisfying the conditions
\begin{gather*}
 \alpha^\pm=(\alpha_1^\pm\ge\alpha_2^\pm\ge
\dots\ge0),\quad
\beta^\pm=(\beta_1^\pm\ge\beta_2^\pm\ge \dots\ge0),\quad \delta^\pm\ge 0,\\
\sum_{i=1}^\infty(\alpha_i^\pm+\beta_i^\pm)\le \delta^\pm,\qquad
\beta_1^++\beta_1^-\le 1.
\end{gather*}

The set $\Omega$ is called the {\it boundary\/} of $\gt$; its points
parameterize the {\it extreme\/} characters of $U(\infty)$.  The map from
central measures on $\gt$ to measures on $\Omega$ amounts to certain asymptotic
relations described in Subsection \ref{Central measures} below.

\subsection{$zw$-measures}

Let $z,z',w,w'$ be four complex parameters such that
\begin{equation}\label{i.2}
(z+k)(z'+k)>0\quad\text{  and  }\quad (w+k)(w'+k)>0\quad\text{  for any  }\quad
k\in\Z
\end{equation}
and
\begin{equation}\label{i.2+}
z+z'+w+w'>-1
\end{equation}
(note that \eqref{i.2} implies that $z+z'$ and $w+w'$ are real).  For $N\ge 1$,
define a probability measure on $\gt_N$ by (below $l_i=\la_i+N-i$)
\begin{equation}\label{i.3}
M_{z,z,w,w'\mid N}(\la)=\const_N\prod_{1\le i<j\le N} (l_i-l_j)^2\prod_{i=1}^N
W_{z,z',w,w'}(l_i),
\end{equation}
where
$$
W_{z,z',w,w'}(x)=\frac1{\Gamma(z+N-x)\Gamma(z'+N-x)\Gamma(w+1+x)\Gamma(w'+1+x)}.
$$
We call it the {\it $N$th level $zw$-measure\/}. It is the $N$-point orthogonal
polynomial ensemble with weight $W(\,\cdot\,)$, see e.g. \cite{Kon05} and
references therein for general information on such ensembles.

One can show that the finite level $zw$-measures are consistent: For any $N\ge
1$, $M_{z,z',w,w'\mid N}=M_{z,z',w,w'\mid N+1}\, \Lambda^{N+1}_N$. Therefore,
the collection $\bigl\{M_{z,z',w,w'\mid N}\bigr\}_{N\ge 1}$ defines a central
measure on the paths in $\gt$ and a character of $U(\infty)$. For $z'=\bar z$,
$w'=\bar w$, this character corresponds to a remarkable substitute for the
nonexisting regular representation of $U(\infty)$, see \cite{Ols03} for
details.

The corresponding measure $M_{z,z',w,w'}$ on $\Omega$ is called the {\it
spectral} $zw$-measure. If $\omega=(\alpha^\pm,\beta^\pm,\delta^\pm)\in\Omega$
is distributed according to $M_{z,z',w,w'}$ then the random point process
generated by the coordinates
$$
\left\{\tfrac 12+\alpha_i^+,\tfrac 12-\beta_i^+, -\tfrac 12+\beta^-_i,-\tfrac
12-\alpha_i^-\right\}_{i=1}^\infty
$$
is determinantal, see \cite{BO05a}, \cite{BO05b} for details.

\subsection{Doob's transforms of $N$-fold products of birth and death processes}

It is not hard to show that the first level $zw$-measure $M_{z,z',w,w'\mid 1}$
on $\gt_1=\Z$ is the symmetrizing measure for the bilateral birth and death
process that from a point $x\in\Z$ jumps to the right with intensity
$(x-u)(x-u')$ and jumps to the left with intensity $(x+v)(x+v')$, where
$(u,u',v,v')=(z,z',w,w')$. Denote by $\D=\D_{u,u',v,v'}$ the corresponding
matrix of transition rates.

More generally, we show that the $N$th level $zw$-measure \eqref{i.3} is the
symmetrizing measure for a continuous time Markov chain on $\gt_N$ with
transition rates
\begin{multline}
\D^{(N)}(\la,\nu)=\frac{\di_N(\nu)}{\di_N(\la)}\Bigl(\D(l_1,n_1) \text{\bf
1}_{\{l_i=n_i,i\ne 1\}}+\D(l_2,n_2)\text{\bf 1}_{\{l_i=n_i,i\ne 2\}}+
\dots\\+\D(l_N,n_N)\text{\bf 1}_{\{l_i=n_i,i\ne N\}}\Bigr) -d_N\cdot\text{\bf
1}_{\la=\nu}
\end{multline}
where $l_j=\la_j+N-j$, $n_j=\nu_j+N-j$, $1\le j\le N$, $d_N$ is a suitable
constant, and we take $(u,u',v,v')=(z+N-1,z'+N-1,w,w')$ in the definition of
$\D$.

Observe that $\D^{(N)}$ can be viewed as a version of Doob's $h$-transform of
$N$ copies of the Markov chain defined by $\D$ with
$h(\,\cdot\,)=\di_N(\,\cdot\,)$. Note that in our case, $\di_N(\,\cdot\,)$ is
an eigenfunction of the corresponding matrix of transition rates with a {\it
nonzero\/} eigenvalue.

For any $N\ge 1$, let $(P_N(t))_{t\ge 0}$ be the Markov semigroup corresponding
to the matrix $\D^{(N)}$ of transition rates on $\gt_N$ (we show that
$(P_N(t))_{t\ge 0}$ is uniquely defined and it possesses the Feller property).
The key fact that we prove is the consistency (or commutativity) relation
$$
P_{N+1}(t)\Lambda^{N+1}_N=\Lambda^{N+1}_NP_N(t),\qquad t\ge 0,\quad N\ge 1.
$$

Although this relation looks natural, we have no {\it a priori\/} reason to
expect it to hold, and we verify it by a brute force computational argument.

\subsection{Main result}

We prove that for any $(z,z',w,w')\in\C^4$ subject to \eqref{i.2}-\eqref{i.2+}, there exists
a unique Markov semigroup $(P(t))_{t\ge 0}$ on $\Omega$ that preserves the
spectral $zw$-measure $M_{z,z',w,w'}$, and whose trace on $\gt_N$ coincides
with Doob's transforms $(P_N(t))_{t\ge 0}$ introduced above. Moreover, the
semigroup $(P(t))_{t\ge 0}$ is Feller (it preserves $C_0(\Omega)$, the Banach
space of continuous functions vanishing at infinity; note that the space
$\Omega$ is locally compact).

By general theory, see e.g. \cite[IV.2.7]{EK86}, this means that for any
probability measure $\mu$ on $\Omega$, there exists a Markov process on
$\Omega$ corresponding to $(P(t))_{t\ge 0}$ with initial distribution $\mu$ and
c\`adl\`ag sample paths. We also show that $M_{z,z',w,w'}$ is the unique
invariant measure for this Markov process.

\subsection{Markov process on Gelfand-Tsetlin schemes}

Via the correspondence between the probability measures on $\Omega$ and central
measures on paths in $\gt$, the semigroup $(P(t))_{t\ge 0}$ defines a Markov
evolution of central measures. It is natural to ask if there exists a Markov
process on {\it all\/} probability measures on paths in $\gt$ that agrees with
the one we have when restricted to the central measures. We construct one such
process; let us describe its transition rates.

Let $\bigl\{\la_i^{(j)}\bigr\}$ be a starting Gelfand-Tsetlin scheme. Then

\noindent $\bullet$\quad Each coordinate $\la_i^{(k)}$ tries to jump to the
right by 1 with rate
$$
(\la_i^{(k)}-i-z+1)(\la_i^{(k)}-i-z'+1)
$$
and to the left by 1 with rate
$$
(\la_i^{(k)}+k-i+w)(\la_i^{(k)}+k-i+w'),
$$
independently of other coordinates.
\smallskip

\noindent $\bullet$\quad If the $\la^{(k)}_i$-clock of the right jump rings but
$\la_i^{(k)}=\la^{(k-1)}_{i-1}$, the jump is blocked. If its left clock rings
but $\la_i^{(k)}=\la^{(k-1)}_{i}$, the jump is also blocked. (If any of the two
jumps were allowed then the resulting set of coordinates would not have
corresponded to a path in $\gt$.)
\smallskip

\noindent $\bullet$\quad If the right $\la^{(k)}_i$-clock rings and there is no
blocking, we find the greatest number $l\ge k$ such that
$\la_{i}^{(j)}=\la_i^{(k)}$ for $j=k,k+1,\dots,l$, and move all the coordinates
$\{\la_{i}^{(j)}\}_{j=k}^l$ to the right by one. Given the change
$\la^{(k)}_{i}\mapsto \la^{(k)}_{i}+1$, this is the minimal modification of
 the initial Gelfand-Tsetlin scheme that preserves interlacing.
\smallskip

\noindent $\bullet$\quad If the left $\la^{(k)}_i$-clock rings and there is no
blocking, we find the greatest number $l\ge k$ such that
$\la_{i+j-k}^{(j)}=\la_i^{(k)}$ for $j=k,k+1,\dots,l$, and move all the
coordinates $\{\la_{i+j-k}^{(j)}\}_{j=k}^l$ to the left by one. Again, given
the change $\la^{(k)}_i\mapsto \la^{(k)}_i-1$, this is the minimal modification
of the set of coordinates that preserves interlacing.

Since the update rule for each coordinate $\la_i^{(k)}$ typically depends only
on a few surrounding coordinates, one can argue that we have a model of {\it
local\/} random growth. It should be compared to the models treated in
\cite{BF08+}, \cite{BK10}, where a similar block-push mechanism was considered
with constant jumps rates, and in \cite{Bor10+}, where the jump rates were also
dependent on the location and numbering of the coordinates.

The key new feature of the Markov process above is the {\it absence of the limit
shape phenomenon\/}. Often taken for granted in local growth models, it is
simply nonexistent here.

This fact becomes more apparent if we restrict ourselves to coordinates
$\{\la^{(j)}_1\}_{j\ge 1}$ only. The evolution of this set of coordinates is
also Markov, and it represents a kind of an exclusion process. Our results
imply that this process has a unique equilibrium measure. Moreover, with
respect to this measure, the asymptotic density $\lim_{j\to\infty}
\la^{(j)}_1/j$ is well-defined and {\it random\/}. It changes over time, and
its distribution is given by a solution to the classical Painlev\'e VI (second
order nonlinear) differential equation, cf. \cite{BD02}.

\subsection{Analytic continuation viewpoint}

We have so far required the parameters $(z,z',w,w')$ to satisfy \eqref{i.2} and
\eqref{i.2+}. However, all the results would hold if \eqref{i.2} is replaced by
more general conditions, see \cite{Ols03} for a precise description, with the
only difference being that the state spaces for our Markov processes would
become smaller. In particular, if we choose
$$
z=k\in\Z_{\ge 0}, \quad z'=k+a-1, \quad w=l\in\Z_{\ge 0}, \quad w'=l+b-1,
\qquad a,\, b>0,
$$
then we have to restrict ourselves to Gelfand-Tsetlin schemes with
$-l\le\la_i^{(j)}\le k$ for all $i,j\ge 1$. As the result, there are only $k+l$
nontrivial parameters remaining on the boundary, and $(P(t))_{t\ge 0}$ turns
into a finite-dimensional diffusion with an explicit second order differential
operator as its generator. The equilibrium distribution (i.e. the spectral
$zw$-measure) becomes the $(k+l)$-point Jacobi orthogonal polynomial ensemble.
See Subsection \ref{Jacobi} for details.

One can thus think of our construction as of an {\it analytic continuation\/}
in parameters $(k,l,a,b)$ of a very well understood finite-dimensional
diffusion. This point of view can be very fruitful: In \cite{Ols10+} it was
heavily exploited in the construction and analysis of the Markov process
preserving the spectral $z$-measure arising from representation theory of the
infinite symmetric group. In that case, the starting point for analytic
continuation was the Laguerre orthogonal polynomial ensemble and the
corresponding diffusion, rather than the Jacobi one that we have here.

\subsection{Pregenerator}

As our construction of the semigroup $(P(t))_{t\ge 0}$ is fairly inexplicit, it
is tempting to look for its alternative definition, for example, via a
generator.

We were able to find a countable set of `coordinates' on $\Omega$ such that the
action of the generator of $(P(t))_{t\ge 0}$ on polynomials in these
coordinates is given by an explicit formal second order differential operator,
see Subsection \ref{Formal generator} below. However, it remains a challenge
for us to derive properties of our Markov process (or its existence) from the
resulting formula.

\subsection{Further questions}
As explained in \cite{Ols03}, the slices $\gt_N$ can be embedded into the
boundary so that as $N\to\infty$, their images form an increasingly fine grid
in $\Omega$. It is known that the $N$th level $zw$-measures weakly converge to
the spectral ones under these embeddings. It would be desirable to prove a
similar statement for the Markov semigroups.

Verifying semigroup convergence would pave the way to proving that the
equilibrium Markov process on the boundary can also be described as a
time-dependent determinantal point process. The fact that the dynamical
correlation functions are determinantal on each $\gt_N$ easily follows from
known techniques, although deriving useful formulas for the correlation kernel
is a separate task. Another possible corollary of the semigroup convergence
would be that the spectral $zw$-measure is a symmetrizing (not just an
invariant) measure for $(P(t))_{t\ge 0}$.

It seems important to continue the study of the pregenerator started in
Subsection \ref{Formal generator}. For example, it would be nice to understand
if the space of polynomials in our coordinates is a core for the generator of
the Markov process, and if not then how that space should be modified.

Another way to benefit from investigating the generator would be to obtain a proof
of the continuity of trajectories for our processes; we are only able to show
that the process has c\`adl\`ag trajectories at the moment.

All these questions and more have been settled in the case of the $z$-measures
treated in \cite{Ols10+}. Unfortunately, key features of that model are not
present here (like decomposition of the process into a one-dimensional one and
a process on a compact set, or the existence of a convenient set of functions
on the state space isomorphic to the well-studied algebra of symmetric functions),
and one would clearly need new ideas.

\subsection{Organization of the paper}
  In Section \ref{Abstract construction} we present an abstract scheme of
constructing a Markov semigroup on the boundary out of a consistent family of
semigroups on the slices. In Section \ref{Specialization} we describe how the
Gelfand-Tsetlin graph fits into this abstract scheme. Section \ref{Generalities
on Markov} is a brief collection of general facts about continuous time Markov chains on
countable spaces. Section \ref{Semigroups} provides the construction of the
Markov chains on $\gt_N$'s. In Section \ref{Commutativity} we verify the
consistency of these Markov chains. Section \ref{zw-measures} contains a brief
description of the $zw$-measures. In Section \ref{Stochastic dynamics} we
develop a general formalism of building continuous time Markov chains on paths
out of a consistent family of those on the slices. In Section \ref{Stochastic
dynamics2} we apply this formalism to our specific example and discuss the
exclusion type processes. Section \ref{Appendix} is an appendix without proofs;
it contains a description of the finite-dimensional case of integral parameters
$z$ and $w$, and an explicit formula for the generator of our
Markov process in certain coordinates.

\subsection{Acknowledgements} A.~B. was partially supported by NSF grants
DMS-0707163 and DMS-1006991. G.~O. was supported by the RFBR grant 08-01-00110,
the RFBR-CNRS grant 10-01-93114, and the project SFB 701 of Bielefeld
University.

\section{Abstract construction}\label{Abstract construction}

\subsection{Markov kernels}\label{Markov
kernels} For a more detailed exposition, see e.g. \cite[Ch. IX]{Mey66}.

Let $E$ and $E'$ be measurable spaces. A {\it Markov kernel\/} $K:E\to E'$ is a
function $K(x,A)$, where $x\in E$ and $A\subset E'$ is a measurable subset,
such that $K(x,\,\cdot\,)$ is a probability measure on $E'$ and
$K(\,\cdot\,,A)$ is a measurable function on $E$.

Let $\mathcal B(E)$ and $\mathcal B(E')$ denote the Banach spaces of
real-valued bounded measurable functions with the sup-norm on $E$ and $E'$,
respectively. A Markov kernel $K:E\to E'$ induces a linear operator $\mathcal
B(E')\to\mathcal B(E)$ of norm 1 via $(Kf)(x)=\int_{E'} K(x,dy) f(y)$.

For two Markov kernels $K_1:E\to E'$ and $K_2:E'\to E''$, their composition
$K_1\circ K_2: E\to E''$ is also a Markov kernel.

Denote by $\mathcal M(\,\cdot\,)$ the Banach space of signed measures of
bounded variation with the norm given by the total variation. Let $\mathcal
M_+(\,\cdot\,)$ be the cone of finite positive measures, and let $\mathcal
M_p(\,\cdot\,)$ be the simplex of the probability measures.

A Markov kernel $K:E\to E'$ also induces a linear operator $\mathcal
M(E)\to\mathcal M(E')$ of norm 1 via $(\mu K)(dy)=\int_{E}\mu(dx)K(x,dy)$. This
operator maps $\mathcal M_+(E)$ to $\mathcal M_+(E')$ and $\mathcal M_p(E)$ to
$\mathcal M_p(E')$. Note that $\delta_x K=K(x,\,\cdot\,)$, where $\delta_x$ is
the Dirac delta-measure at $x\in E$.

The space $\mathcal M(E)$ (and hence $\mathcal M_+(E)$ and $\mathcal M_p(E)$)
is equipped with a $\sigma$-algebra of measurable sets: Any preimage of a Borel
set under the map $\mu\mapsto \mu(A)$ from $\mathcal M(E)$ to $\R$ for any
measurable $A$ is measurable.

\subsection{Feller kernels}\label{Feller kernels}

Let $E$ and $E'$ be locally compact topological spaces with countable bases.
Let us take Borel $\sigma$-algebra as the $\sigma$-algebra of measurable sets
for both of them.

Let $C(\,\cdot\,)\subset\mathcal B(\,\cdot\,)$ be the Banach space of bounded
continuous functions, and let $C_0(\,\cdot\,)\subset C(\,\cdot\,)$ be its
subspace of functions that tend to 0 at infinity.

\begin{definition}\label{1.1}
A Markov kernel $K:E\to E'$ is called {\it Feller\/} if
the induced map $\mathcal B(E')\to \mathcal B(E)$ maps $C_0(E')$ to $C_0(E)$.
\end{definition}

Note that different authors may use different (nonequivalent) definitions for
the Feller property.

The convenience of the space $C_0(\,\cdot\,)$ is based on the fact that this
space is separable (as opposed to $C(\,\cdot\,)$ which is not separable,
except in the case when the initial topological space is compact), and $\mathcal
M(\,\cdot\,)$ is its Banach dual.

\subsection{Feller semigroups}

A {\it Markov semigroup\/} is a family of Markov kernels $P(t):E\to E$, where
$t\ge 0$, $P(0)=1$ (in the obvious sense), and $P(s)P(t)=P(s+t)$. Such a
semigroup induces a semigroup of linear operators in $\mathcal B(E)$ as well as
a semigroup of linear operators in $\mathcal M_p(E)$, see above.

We say that a Markov semigroup $(P(t))_{t\ge 0}$ is {\it Feller\/} if

\noindent $\bullet$\ $E$ is a locally compact topological space with countable
base;

\noindent $\bullet$\  the corresponding operator semigroup in $\mathcal B(E)$
preserves $C_0(E)$;

\noindent $\bullet$\  the function $t\mapsto P(t)$ is strongly continuous, i.e.
$t\mapsto P(t)f$ is a continuous map from $[0,+\infty)$ to $C_0(E)$ for any
$f\in C_0(E)$ (an equivalent condition is the continuity at $t=0$).

\subsection{Feller semigroups and Markov processes}

For more details, see e.g. \cite[IV.2.7]{EK86}.

Let $E$ be a locally compact separable metric space, and let $(P(t))_{t\ge 0}$
be a Feller semigroup on $E$. Then for each $\mu\in \mathcal M_p(E)$, there
exists a Markov process corresponding to $(P(t))_{t\ge 0}$ with initial
distribution $\mu$ and c\`adl\`ag sample paths. Moreover, this process is
strongly Markov with respect to the right-continuous version of its natural
filtration.

\subsection{Boundary}\label{Boundary}

Let $E_1,E_2,\dots$ be a sequence of measurable spaces linked by Markov kernels
$$
\Lambda_N^{N+1}:E_{N+1}\to E_N,\qquad N=1,2,\dots\,.
$$
Assume that we have another measurable space $E_\infty$ and Markov kernels
$$
\Lambda^\infty_N:E_\infty\to E_N,\qquad N=1,2,\dots,
$$
such that the natural commutativity relations hold:
\begin{equation}\label{eq1.1}
\Lambda^\infty_{N+1}\circ \Lambda^{N+1}_N=\Lambda^\infty_N,\qquad
N=1,2,\dots\,.
\end{equation}

The kernels $\Lambda_N^{N+1}$ induce the chain of maps, cf. \ref{Markov
kernels},
\begin{equation}\label{eq1.2}
\dots\to\mathcal M_p(E_{N+1})\to\mathcal M_p(E_N)\to\dots\to\mathcal
M_p(E_2)\to\mathcal M_p(E_1),
\end{equation}
and we can define the projective limit $\varprojlim \mathcal M_p(E_N)$ with
respect to these maps. By definition, it consists of sequences of measures
$(\mu_N)_{N\ge 1}$, $\mu_N\in \mathcal M_p(E_N)$, that are linked by the maps
from \eqref{eq1.2}. The space $\varprojlim \mathcal M_p(E_N)$ is
measurable; the $\sigma$-algebra of measurable sets is generated by the
cylinder sets in which $\mu_N$ must lie inside a measurable subset of $\mathcal
M_p(E_N)$, and all other coordinates $(\mu_k)_{k\ne N}$, are unrestricted.

Observe that to any $\mu_\infty\in \mathcal M_p(E_\infty)$ one can assign an
element of $\varprojlim \mathcal M_p(E_N)$ by setting $\mu_N$ equal to the
image of $\mu_\infty$ under the map $\mathcal M_p(E_\infty)\to\mathcal
M_p(E_N)$ induced by $\Lambda^\infty_N$. The commutativity relations
\eqref{eq1.1} ensure that the resulting sequence $(\mu_N)_{N\ge 1}$ is
consistent with \eqref{eq1.2}.

\begin{definition}\label{1.2}
We say that $E_\infty$ is a {\it boundary\/} of the
sequence $(E_N)_{N\ge 1}$ if the map $\mathcal M_p(E_\infty)\to \varprojlim
\mathcal M_p(E_N)$ described in the previous paragraph is a bijection and also
an isomorphism of measurable spaces.
\end{definition}

\subsection{Feller boundary}\label{Feller boundary}

In the setting of the previous subsection, let us further assume that
$(E_N)_{N\ge 1}$ and $E_\infty$ are locally compact topological spaces with
countable bases, and all the links ${(\Lambda^{N+1}_N)}_{N\ge 1}$,
${(\Lambda^{\infty}_N)}_{N\ge 1}$ are Feller kernels, cf. Subsection
\ref{Feller kernels}. Then if $E_\infty$ satisfies Definition \ref{1.2}, we
shall call it the {\it Feller boundary\/} for $(E_N)_{N\ge 1}$.

According to Subsection \ref{Feller kernels}, the links
${(\Lambda^{\infty}_N)}_{N\ge 1}$ induce linear operators $C_0(E_N)\to
C_0(E_\infty)$.

\begin{lemma}\label{1.3}
The union of images of these maps over all $N\ge 1$ is
dense in the Banach space $C_0(E_\infty)$.
\end{lemma}

\begin{proof} Since $\mathcal M(E_\infty)$ is the Banach dual to $C_0(E_\infty)$, it
suffices to verify that if $\mu\in \mathcal M(E_\infty)$ kills all functions in
our union then $\mu=0$.

Assume $\mu$ is a signed measure on $E_\infty$ that kills the image of
$C_0(E_N)$, $N\ge 1$. This is equivalent to saying that $\mu K_N^\infty=0$ for
all $N\ge 1$. We can represent $\mu$ as difference of finite positive measures
$$
M=\alpha \mu'-\beta \mu'',\qquad \mu',\mu''\in \mathcal M_p(E_\infty),\quad
\alpha,\beta\ge 0.
$$
Hence, $\alpha \mu' K_N^\infty=\beta \mu'' K_N^\infty$ for all $N\ge 1$. Since
$\mu' K_N^\infty$ and $\mu'' K_N^\infty$ are in $\mathcal M_p(E_N)$, we must
have $\alpha=\beta$, and $\mu' K_N^\infty=\mu'' K_N^\infty$. Definition
\ref{1.2} implies $\mu'=\mu''$, thus $\mu=0$.
\end{proof}

\subsection{Extension of semigroups to the boundary}

In the setting of Subsection \ref{Boundary}, assume that for any $N\ge 1$, we
have a Markov semigroup $(P_N(t))_{t\ge 0}$ on $E_N$, and these semigroups are
compatible with the links:
\begin{equation}\label{eq1.3}
P_{N+1}(t)\circ\Lambda^{N+1}_{N}=\Lambda^{N+1}_{N}\circ P_{N}(t),\qquad t\ge 0,
\quad N=1,2,\dots\,.
\end{equation}

\begin{proposition}\label{1.4}
In the above assumptions, there exists a unique
Markov semigroup $P(t)$ on $E_\infty$ such that
\begin{equation}\label{eq1.4}
P(t)\circ \Lambda_N^\infty=\Lambda^\infty_N\circ P_N(t),\qquad t\ge 0,\quad
N=1,2,\dots\,.
\end{equation}

If $E_\infty$ is Feller (cf. Subsection \ref{Feller boundary}) and
$(P_N(t))_{t\ge 0}$ is a Feller semigroup for any $N\ge 1$, then $(P(t))_{t\ge
0}$ is also a Feller semigroup.
\end{proposition}

\begin{proof} Denote by $\delta_x$ the delta-measure at a point $x\in E_\infty$.
To construct the semigroup $(P(t))_{t\ge 0}$, we need to define, for any $t\ge
0$, a probability measure $P(t;x,\,\cdot\,)$ on $E_\infty$. This measure has to
satisfy
$$
P(t;x,\,\cdot\,)\Lambda^\infty_N=\delta_x(\Lambda^\infty_N\circ P_N(t)),\qquad
N\ge 1.
$$
The right-hand side defines a sequence of probability measures on $E_N$'s, and
\eqref{eq1.1}, \eqref{eq1.3} immediately imply that these measures are
compatible with maps \eqref{eq1.2}. Hence, we obtain an element of $\varprojlim
\mathcal M_p(E_N)$, which defines, by definition of the boundary, a probability
measure on $E_\infty$. The dependence of this measure on $x$ is measurable
since this is true for any of its coordinates.

Thus, we have obtained a Markov kernel $P(t)$ which satisfies
$$
\delta_x (P(t)\circ \Lambda_N^\infty)=\delta_x(\Lambda_N^\infty\circ
P_N(t)),\qquad N\ge 1,
$$
which is equivalent to \eqref{eq1.4}.

To verify the semigroup property (Chapman-Kolmogorov equation) for
$(P(t))_{t\ge 0}$ it suffices to check that
$$
(P(s)\circ P(t))\circ\Lambda^\infty_N=P(s+t)\circ\Lambda^\infty_N,\qquad s,t\ge
0,\quad N\ge 1,
$$
and this immediately follows from \eqref{eq1.4} and the corresponding relation
for $(P_N(t))_{t\ge 0}$.

The uniqueness is obvious since $P(t)$ is uniquely determined by
$(P(t)\circ\Lambda^\infty_N)_{N\ge 1}$ that are given \eqref{eq1.4}.

Finally, let us prove the Feller property assuming that the boundary is Feller
and all $(P_N(t))_{t\ge 0}$ are Feller.

We need to show that for $f\in C_0(E_\infty)$ we have $P(t)f\in C_0(E_\infty)$,
and that $P(t)f$ is continuous in $t$ in the topology of $C_0(E_\infty)$. Both
properties can be verified on a dense subset. Lemma \ref{1.3} then shows that
it suffices to consider $f$ of the form $f=\Lambda_N^\infty f_N$ with $f_N\in
C_0(E_N)$. By \eqref{eq1.4}
$$
P(t)f=P(t)(\Lambda_N^\infty f_N)=\Lambda^\infty_N(P_N(t) f_N),
$$
which is in $C_0(E_\infty)$ because $\Lambda^\infty_N$ and $P_N(t)$ are Feller.
The continuity in $t$ is obvious as $P_N(t)f_N$ is continuous in $t$, and
$\Lambda^\infty_N:C_0(E_N)\to C_0(E_\infty)$ is a contraction.
\end{proof}

It is worth noting that our definition of the semigroup $P(t)$ is
nonconstructive: We are not able to describe $P(t;x,A)$ explicitly, and we have
to appeal to the isomorphism in Definition \ref{1.2} instead. Thus, the
difficulty in making $P(t)$ explicit is hidden in the implicit nature of that
isomorphism.

\subsection{Invariant measures}\label{Invariant measures}

In the setting of Subsection \ref{Boundary}, assume that for any $N\ge 1$,
there exists $\mu_N\in \mathcal M_p(E_N)$ such that $\mu_NP_N(t)=\mu_N$ (i.e.,
$\mu_N$ is an invariant measure for $(P_N(t))_{t\ge 0}$). If we assume that
$\mu_N$'s are compatible with the links,
$$
\mu_{N+1}\Lambda^{N+1}_N=\mu_N,\qquad N\ge 1,
$$
then, via Definition \ref{1.2}, they yield a measure $\mu\in\mathcal
M_p(E_\infty)$ such that $\mu\Lambda^\infty_N=\mu_N$ for any $N\ge 1$. Note
that $\mu$ is uniquely determined by its coordinates.

One easily sees that $\mu$ is invariant with respect to $(P(t))_{t\ge 0}$.
Indeed,
$$
(\mu P(t))\Lambda^\infty_N=(\mu \Lambda^\infty_N)P_N(t)=\mu_N P_N(t)=\mu_N=\mu
\Lambda^\infty_N.
$$
Moreover, if $\mu_N$ is a unique invariant measure for $(P_N(t))_{t\ge 0}$ for
any $N\ge 1$ then the invariant measure for $(P(t))_{t\ge 0}$ is unique too as
its convolution with $\Lambda^\infty_N$ must coincide with $\mu_N$.

\section{Specialization. Gelfand-Tsetlin graph}\label{Specialization}

\subsection{Spaces and links}\label{Spaces and links}

Let $N$ be a positive integer. A {\it signature\/} $\la$ of length $N$ is an
$N$-tuple of weakly decreasing integers: $\la=(\la_1\ge\dots\ge\la_N)\in \Z^N$.
Denote by $\gt_N$ the set of all signatures of length $N$ (the notation
$\gt$ is explained below). This countable set will serve as our space $E_N$
from the previous section.

Signatures of length $N$ parameterize irreducible representations of the
unitary group $U(N)$ and are often referred to as {\it highest weights\/}, cf.
\cite{Wey39}, \cite{Zhe70}. For $\la\in\gt_N$ denote the corresponding
representation by $\pi_\la$, and denote by $\di_N\lambda$ the dimension of the
corresponding linear space. It is well known that
$$
\di_N\lambda=\frac{\prod_{1\le i<j\le N}(\la_i-i-\la_j+j)}{\prod_{i=1}^{N-1}
i!},\qquad \la\in\gt_N.
$$

Define a matrix
$\bigl[\La^{N+1}_{N}(\la,\nu)\bigr]_{\la\in\gt_{N+1},\,\nu\in\gt_{N}}$ with
rows parameterized by $\gt_{N+1}$ and columns parameterized by $\gt_{N}$ via
$$
\La^{N+1}_{N}(\la,\nu)=\begin{cases} N!\cdot\dfrac{\prod_{1\le i<j\le
N}(\nu_i-i-\nu_j+j)}{\prod_{1\le i<j\le
N+1}(\la_i-i-\la_j+j)}\,,&\text{if  } \nu\prec\la,\\
0,&\text{otherwise},
\end{cases}
$$
where the notation $\nu\prec\la$ stands for interlacing:
$$
\nu\prec\la\quad
\Longleftrightarrow\quad\la_1\ge\nu_1\ge\la_2\ge\nu_2\ge\dots\ge\nu_{N}\ge\la_{N+1}.
$$
Note that the nonzero entries of $\La^{N+1}_N$ can also be written in the form
\begin{equation}\label{eq2.0}
\La^{N+1}_N(\la,\nu)=\frac{\di_N\nu}{\di_{N+1}\la}\,.
\end{equation}

It is not hard to show that $\Lambda^{N+1}_{N}$ is a stochastic matrix:
$\sum_{\nu\in\gt_{N}}\La^{N+1}_{N}(\la,\nu)=1$ for any $\la\in\gt_{N+1}$.
Indeed, $\di_{N+1}\la$  is equal to the number of the sequences (known as {\it
Gelfand-Tsetlin schemes\/}, thus the notation $\gt$)
$$
\la^{(1)}\prec\la^{(2)}\prec\dots\prec\lambda^{(N+1)}=\lambda,
\qquad\la^{(j)}\in\gt_j,
$$
and $\Lambda^{N+1}_{N}(\la,\nu)$ is the fraction of the sequences with
$\la^{(N)}=\nu$. The stochasticity also follows from the branching rule for the
representations of unitary groups: For any $\la\in\gt_{N+1}$,
$$
\pi_\la\vert_{U(N)}\sim \bigoplus_{\nu\in\gt_N:\,\nu\prec\la}\pi_\nu.
$$

The matrices $\Lambda^{N+1}_N$ viewed as Markov kernels
$\Lambda^{N+1}_N:\gt_{N+1}\to\gt_N$ are our links, cf. Subsection
\ref{Boundary}. Set $\gt=\bigsqcup_{N\ge 1} \gt_N$. We endow $\gt$ with the
structure of a graph: Two vertices $\lambda$ and $\nu$ are joined by an edge if
and only if $\nu\prec\lambda$ or $\lambda\prec\nu$. This graph is called the
{\it Gelfand-Tsetlin graph\/}, and the matrix elements of the links are often
called {\it cotransition probabilities\/} for this graph, cf. \cite{Ker03}.

\subsection{Boundary}

let $\R_+\subset\R$ be the set of nonnegative real numbers and $\R_+^\infty$ be
the product of countably many copies of $\R_+$. Consider the space
$$
\R_+^{4\infty+2}
:=\R^\infty_+\times\R^\infty_+\times\R^\infty_+\times\R^\infty_+\times\R_+\times\R_+
$$
and equip it with the product topology. We choose $E_\infty$ to be the closed
subset $\Omega\subset\R_+^{4\infty+2}$ consisting of the sextuples
$$
\omega=(\alpha^+,\beta^+,\alpha^-,\beta^-,\delta^+,\delta^-)\in
\R_+^{4\infty+2}
$$
satisfying the conditions
$$
\gathered
 \alpha^\pm=(\alpha_1^\pm\ge\alpha_2^\pm\ge
\dots),\quad
\beta^\pm=(\beta_1^\pm\ge\beta_2^\pm\ge \dots),\quad \delta^\pm\ge 0,\\
\sum_{i=1}^\infty(\alpha_i^\pm+\beta_i^\pm)\le \delta^\pm,\qquad
\beta_1^++\beta_1^-\le 1.
\endgathered
$$
One easily sees that $\Omega$ is a locally compact metrizable topological space
with a countable base. We endow $\Omega$ with the corresponding Borel
structure which makes $\Omega$ a measurable space.

It will be convenient to use the notation
$$
\gamma^\pm=\delta^\pm-\sum_{i=1}^\infty(\alpha_i^\pm+\beta_i^\pm)\ge 0.
$$

Define the projections/links $\Lambda^\infty_N:\Omega\to \gt_N$, $N\ge 1$, by
\begin{equation}\label{eq2.1}
\Lambda^\infty_N(\omega,\la)=\di_N\la\cdot
\det\bigl[\varphi_{\la_i-i+j}\bigr]_{i,j=1}^N,\qquad \omega\in\Omega,\quad
\la\in\gt_N,
\end{equation}
where $\{\varphi_n\}_{n=-\infty}^{+\infty}$ are the Laurent coefficients of the
function ($|u|=1$)
\begin{equation}\label{eq2.2}
\Phi_\omega(u):=e^{\gamma^+(u-1)+\gamma^-(u^{-1}-1)}\prod_{i=1}^\infty
\frac{1+\beta_i^+(u-1)}{1-\alpha_i^+(u-1)}\,
\frac{1+\beta_i^-(u^{-1}-1)}{1-\alpha_i^-(u^{-1}-1)}=\sum_{n=-\infty}^{+\infty}
\varphi_nu^n.
\end{equation}

\begin{theorem}\label{2.1}
The space $E_\infty=\Omega$ is the boundary of the chain of spaces
$(E_N=\gt_N)_{N\ge 1}$ with links as above in the sense of Definition
\ref{1.2}.
\end{theorem}

\begin{proof} This result is essentially proved in \S9 of \cite{Ols03}; we
provide below some necessary additional comments.

Let us abbreviate $\Delta_N=\mathcal M_p(\gt_N)$ and $\Delta=\varprojlim
\mathcal M_p(\gt_N)=\varprojlim\Delta_N$. As in \cite{Ols03} we embed $\Delta$
into the vector space of all real-valued functions on the set of vertices of
the Gelfand--Tsetlin graph. That space is endowed with the topology of
pointwise convergence, and $\Delta$ inherits this topology. The measurable
structure of $\Delta$ is the Borel structure corresponding to this topology.

Obviously, $\Delta$ is a convex set; let $\Ex\Delta\subset\Delta$ denote the
subset of extreme points. Theorem 9.2 in \cite{Ols03} (which is based on
Choquet's theorem) says that $\Ex\Delta$ is a Borel subset of $\Delta$, and each
point of $\Delta$ is uniquely representable by a probability Borel measure
concentrated on $\Ex\Delta$. On the other hand, it is readily seen that,
conversely, any probability Borel measure on $\Delta$ (in particular, on
$\Ex\Delta$) represents a point of $\Delta$, the barycenter of that measure.
This gives us a bijection between $\mathcal M_p(\Ex\Delta)$ and $\Delta$.

The next step consists in identifying the abstract set $\Ex\Delta$ with the
concrete space $\Omega$. This is achieved with the help of Theorem 1.3 in
\cite{Ols03}. Namely, as is pointed out in the proof of Theorem 9.1 in
\cite{Ols03}, there is a natural one-to-one correspondence between the points
of $\Ex\Delta$ and the extreme characters of the infinite-dimensional unitary
group $U(\infty)$, which in turn are parameterized by the points of the space
$\Omega$, see Theorem 1.3 in \cite{Ols03}.

Then we have to verify that the embedding $\Omega\to\Delta$ induced by the
identification $\Omega=\Ex\Delta$ is given by the kernels $\La^\infty_N$. This
is shown by the computation in \cite{Voi76}.

We have thus constructed a bijective map $\mathcal M_p(\Omega)\to\Delta$, and
it remains to prove that it is a Borel isomorphism. As shown in the proof of
Theorem 8.1 of \cite{Ols03}, the map $\omega\mapsto\La^\infty_N(\omega,\la)$ is
continuous for every $N=1,2,\dots$ and every $\la\in\gt_N$. This implies that
the map $\mathcal M_p(\Omega)\to\Delta$ is Borel. To show that the inverse map
is also Borel one can apply an abstract result (Theorem 3.2 in \cite{Mack57}),
which asserts that a Borel one-to-one map of a standard Borel space onto a
subset of a countably generated Borel space is a Borel isomorphism. This result
is applicable in our situation, since the Borel structure of $\Omega$ is
standard, so that the induced Borel structure on $\mathcal M_p(\Omega)$ is
standard, too.

\end{proof}

\begin{remark}\label{2.0}
Observe that the maps on $(\gt_N)_{N\ge 1}$ consisting of
shifts of all coordinates of signatures by 1,
$$
\la=(\la_1,\dots,\la_N)\mapsto
\wt\la=(\wt\la_1=\la_1+1,\dots,\wt\la_N=\la_N+1),
$$
leave the links intact: $\La^{N+1}_N(\la,\nu)=\La^{N+1}_N(\wt\la,\wt\nu)$.
There is also a corresponding homeomorphism of $\Omega$, which amounts to the
multiplication of the function $\Phi_\omega(u)$ by $u$: For
$\omega=(\alpha^\pm,\beta^\pm,\delta^\pm)\in \Omega$ define
$\wt\omega=(\wt\alpha^\pm,\wt\beta^\pm,\wt\delta^\pm)\in\Omega$ by
$$
\gathered
\wt\alpha^\pm=\alpha^\pm,\qquad \wt\delta^\pm=\delta^\pm, \\
\wt\beta_1^+=1-\beta^-_1,\quad (\wt\beta_2^+,\wt\beta_3^+,\dots)=
(\beta_1^+,\beta_2^+,\dots),\quad (\wt\beta_1^-,\wt\beta_2^-,\dots)=
(\beta_2^-,\beta_3^-,\dots)
\endgathered
$$
(note that $\wt\beta^+_1\ge\wt\beta^+_2$ because $\beta^+_1+\beta^-_1\le1$).
Then \eqref{eq2.1} and the relation
$$
u(1+\beta_1^-(u^{-1}-1))=1+(1-\beta_1^-)(u-1)
$$
show that $\Lambda^\infty_N(\omega,\la)=\Lambda^\infty_N(\wt\omega,\wt\la)$ for
any $\la\in\gt_N$ and $N\ge 1$.

This automorphism of the Gelfand-Tsetlin graph and its boundary has a
representation theoretic origin, cf. Remark 1.5 in \cite{Ols03} and Remark 3.7
in \cite{BO05a}.
\end{remark}

\subsection{The boundary is Feller}

Following definitions of Subsection \ref{Feller boundary}, in order to show
that $E_\infty=\Omega$ is a Feller boundary of the chain $(E_N=\gt_N)_{N\ge 1}$
we need to verify two statements:

\noindent$\bullet$\quad the spaces $(E_N)_{N\ge 1}$ and $E_\infty$ are locally
compact topological spaces with countable bases;

\noindent$\bullet$\quad the links $(\Lambda^{N+1}_N)_{N\ge 1}$ and
$(\Lambda^\infty_N)_{N\ge 1}$ are Feller kernels.

The first statement is obvious from the definitions. The goal of this
subsection is to prove the second one.

\begin{proposition}\label{2.2}
For any $N\ge 1$, the linear operator $\mathcal B(\gt_N)\to \mathcal
B(\gt_{N+1})$ induced by the Markov kernel $\Lambda^{N+1}_N$ maps $C_0(\gt_N)$
to $C_0(\gt_{N+1})$.
\end{proposition}

\begin{proof} As the norm of the linear operator in question is equal to 1 and
$C_0(\,\cdot\,)$ is a closed subspace of $\mathcal B(\,\cdot\,)$, it suffices
to check that the images of all delta-functions on $\gt_N$ are in
$C_0(\gt_{N+1})$.

For a $\nu\in\gt_N$, let $\delta_\nu$ be the delta-function on $\gt_N$
concentrated at $\nu$. Then for $\la\in\gt_{N+1}$
$$
(\Lambda^{N+1}_N \delta_\nu)(\la)=\begin{cases} N!\cdot\dfrac{\prod_{1\le
i<j\le N}(\nu_i-i-\nu_j+j)}{\prod_{1\le i<j\le
N+1}(\la_i-i-\la_j+j)}\,,&\text{if  } \nu\prec\la,\\
0,&\text{otherwise}.
\end{cases}
$$
If we assume that $(\Lambda^{N+1}_N \delta_\nu)(\la)\ne 0$ then $\la\to\infty$
is equivalent to either $\la_1\to+\infty$, or $\la_{N+1}\to -\infty$, or both;
all other coordinates must remain bounded because of the interlacing condition
$\nu\prec\la$. But then it is immediate that at least one of the factors in the
denominator in $(\Lambda^{N+1}_N \delta_\nu)(\la)$ tends to infinity. Thus, for
any fixed $\nu\in\gt_N$, $(\Lambda^{N+1}_N \delta_\nu)(\la)\to 0$ as
$\la\to\infty$ as needed.
\end{proof}

\begin{proposition}\label{2.3}
For any $N\ge 1$, the linear operator $\mathcal B(\gt_N)\to \mathcal B(\Omega)$
induced by the Markov kernel $\Lambda^{\infty}_N$ maps $C_0(\gt_N)$ to
$C_0(\Omega)$.
\end{proposition}

\begin{proof} As in the proof of Proposition \ref{2.2}, it suffices to prove that for
any $\nu\in\gt_N$,
$(\Lambda^\infty_N\delta_\nu)(\omega)=\Lambda^\infty_N(\omega,\lambda)$ belongs
to $C_0(\Omega)$ as a function in $\omega\in\Omega$. The proof of continuity of
$\Lambda^\infty_N(\omega,\lambda)$ in $\omega$ is contained in the proof of
Theorem 8.1 of \cite{Ols03}. It remains to show that
$\Lambda^\infty_N(\omega,\lambda)\to 0$ as $\omega\to\infty$. Note that
$\omega\to\infty$ is equivalent to $\delta^++\delta^-\to+\infty$.

Observe that the coefficients $\varphi_n=\varphi_n(\omega)$ from \eqref{eq2.2}
can be written as
\begin{equation}\label{eq2.3}
\varphi_n(\omega)=\frac 1{2\pi i}\oint_{|u|=1} \Phi_\omega(u)\frac
{du}{u^{n+1}},\qquad n\in\Z.
\end{equation}
Note that $|\Phi_\omega(u)|\le 1$ on the unit circle $|u|=1$, because the
modulus of each of the factors $(1-\alpha_i^\pm (u^{\pm 1}-1))^{-1}$,
$(1+\beta_i^\pm (u^{\pm 1}-1))$, and $e^{\gamma^\pm(u^{\pm 1}-1)}$ is $\le1$.

We are going to prove that for any fixed $n\in\Z$, $\varphi_n\to 0$ as
$\omega\to\infty$; by \eqref{eq2.1} this would imply the needed claim.

Let us assume the converse, i.e. assume that there exist $c>0$, $n_0\in\Z$, and
a sequence $\{\omega(k)\}_{k\ge 1}\subset\Omega$ with
$\lim_{k\to\infty}\omega(k)=\infty$, such that $\varphi_{n_0}(\omega(k))>c$.
Let us denote by $(\alpha^\pm(k),\beta^\pm(k),\delta^\pm(k))$ the coordinates
of $\omega(k)$. We will now collect information about $\{\omega(k)\}$ that will
eventually lead to a contradiction.

\medskip

\noindent{\it Step 1.\/} We must have $\sup_{k\ge 1} \alpha^\pm_1(k)<\infty$.
Indeed, if there is a subsequence $\{k_m\}_{m\ge 1}$ such that
$\alpha_1^\pm(k_m)\to\infty$, then along this subsequence $(1-\alpha_1^\pm
(u^{\pm 1}-1))^{-1}$ tends to zero uniformly on any compact subset of
$\{u:|u|=1\}\setminus\{u=1\}$, which implies that the right-hand side of
\eqref{eq2.3} tends to zero.

Let us fix $A>0$ such that $\sup_{k}\alpha_1^{\pm}(k)\le A$.

\medskip

\noindent{\it Step 2.\/} Assume $\omega$ ranges over the subset of elements of
$\Omega$ with $\alpha_1^\pm\le A$ and $\beta_1^\pm\le \frac 12$. Then for any
$\epsilon>0$,
$$
\lim_{\delta^++\delta^-\to\infty}\Phi_\omega(u)=0\ \text{ uniformly on }\
\{u\in\C: |u|=1,\, \Re u\le 1-\epsilon\}.
$$

Indeed, for $u$ on the unit circle with $\Re u\le 1-\epsilon$ we have
elementary estimates
\begin{multline}
|1+\beta(u-1)|^2=
(1-\beta)^2+\beta^2+2\beta(1-\beta)\Re u\\=1-2\beta(1-\beta)(1-\Re u)\le
1-2\beta(1-\beta)\epsilon\le 1-\beta\epsilon\le e^{-\beta\epsilon},
\end{multline}

\begin{multline}
|1-\alpha(u-1)|^{-2}=(1+2\alpha(1+\alpha)(1-\Re u))^{-1}\\\le
(1+2\alpha(1+\alpha)\epsilon)^{-1}\le (1+2\alpha\epsilon)^{-1}\le
e^{-\const\alpha\epsilon},
\end{multline}

$$
|e^{\gamma^+(u-1)+\gamma^-(u^{-1}-1)}|^2=e^{-2(\gamma^++\gamma^-)(1-\Re u)}\le
e^{-2(\gamma^++\gamma^-)\epsilon},
$$
with a suitable constant $\const>0$ (that depends on $A$). Thus, if
$$
\delta^++\delta^-=\gamma^++\gamma^-+
\sum_{i=1}^\infty(\alpha^+_i+\beta^+_i+\alpha^-_i+\beta^-_i)\to\infty
$$
then at least one of the right-hand sides in these estimates yields an
infinitesimally small contribution, and $\Phi_\omega(u)$ must be small. Thus,
under the above assumptions on $\omega$, we see that $\omega\to\infty$ implies
$\varphi_n(\omega)\to0$ uniformly on $n\in\Z$.

\medskip

\noindent {\it Step 3.\/} Now we get rid of the restriction
$\beta_1^{\pm}\le\frac12$. Set
$$
B^\pm(k)=\#\{i\ge 1\mid \beta_i^\pm(k)>\tfrac 12\}.
$$
Since for any $k\ge 1$ we have $\beta_1^+(k)+\beta_1^-(k)\le 1$, at least one
of the numbers $B^\pm(k)$ is equal to 0. The statement of Step 2 shows that for
any subsequence $\{\omega_{k_m}\}$ of our sequence $\{\omega(k)\}$, we must
have $B^+(k_m)+B^-(k_m)\to\infty$. Hence, possibly passing to a subsequence and
switching $+$ and $-$, we may assume that $B^+(k)\to\infty$ as $k\to\infty$.

Using the identity (cf. Remark \ref{2.0})
$$
1+\beta(u-1)=u(1+(1-\beta)(u^{-1}-1))
$$
$B^+(k)$ times on $\Phi_{\omega(k)}(u)$, we see that
$\varphi_{n_0}(\omega(k))=\varphi_{n_0-B^+(k)}(\tilde\omega(k))$, where
$\tilde\omega(k)$ is obtained from $\omega(k)$ as follows: Each
$\beta^+$-coordinate of $\omega(k)$ that is $>1/2$ is transformed into a
$\beta^-$ coordinate of $\tilde \omega(k)$ equal to $1$ minus the original
$\beta^+$-coordinate; all other coordinates are the same (equivalently, the
function $\Phi_{\omega(k)}(u)$ is multiplied by $u^{-B(k)}$). Let
$(\tilde\alpha^\pm(k),\tilde\beta^\pm(k),\tilde\gamma^\pm(k),\tilde\delta^\pm(k))$
be the coordinates of $\tilde \omega(k)$.

\medskip

\noindent {\it Step 4.\/} Since no $\beta$-coordinates of $\tilde\omega(k)$ are
greater than $1/2$, the argument of Step 2 implies that if
$\tilde\delta^+(k)+\tilde\delta^-(k)\to\infty$ then
$\varphi_{n_0}(\omega(k))=\varphi_{n_0-B^+(k)}(\tilde\omega(k))\to 0$ as
$k\to\infty$, which contradicts our assumption. Hence,
$\tilde\delta^+(k)+\tilde\delta^-(k)$ is bounded.

Let us deform the integration contour in \eqref{eq2.3} to $|u|=R$ with
$A/(1+A)<R<1$. Using the estimates (for $|u|=R$, $0\le \alpha\le A$, $0\le
\beta\le 1/2$)
$$
\gathered
|1+\beta(u^{\pm 1}-1)|\le 1+\beta|u^{\pm 1}-1|\le e^{\const_1\beta},\\
|1-\alpha(u^{\pm 1}-1)|^{-1}\le |1-\alpha(R^{\pm 1}-1)|^{-1}\le e^{\const_2\alpha},\\
|e^{\gamma(u^{\pm 1}-1)}|\le e^{\const_3\gamma}
\endgathered
$$
with suitable $\const_j>0$, $j=1,2,3$, we see that
$|\Phi_{\tilde\omega(k)}(u)|\le
e^{\const_4({\tilde\delta^+(k)+\tilde\delta^-(k)})}$, for a $\const_4>0$, which
remains bounded. On the other hand, the factor $u^{-n_0-1+B^+(k)}$ in the
integral representation \eqref{eq2.3} for
$\varphi_{n_0-B^+(k)}(\tilde\omega(k))$ tends to 0 uniformly in $u$, $|u|=R<1$.
Hence, $\varphi_{n_0}(\omega(k))=\varphi_{n_0-B^+(k)}(\tilde\omega(k))\to 0$ as
$k\to\infty$, and the proof of Proposition \ref{2.3} is complete.
\end{proof}

\section{Generalities on Markov chains on countable spaces}\label{Generalities
on Markov}

\subsection{Regularity}

Let $\A$ be a countable set, and let $(P(t))_{t\ge 0}$ be a Markov semigroup on
$\A$. Each $P(t)$ may be viewed as a matrix with rows and columns marked by
elements of $\A$; its entries will be denoted by $P(t;a,b)$, $a,b\in\A$. By
definition, $P(t;a,b)$ is the probability that the process will be in the state
$b$ at the time moment $t$ conditioned that it is in the state $a$ at time $0$.
Thus, all matrix elements of $P(t)$ are nonnegative, and their sum is equal to
1 along any row - the matrix $P(t)$ is stochastic. The transition matrices
$P(t)$ also satisfy the Chapman-Kolmogorov equation $P(s)P(t)=P(s+t)$.

Assume that there exists an $\A\times \A$ matrix $Q$ such that
\begin{equation}\label{eq3.0}
P(t;a,b)=\text{\bf 1}_{a=b}+Q(a,b)t+o(t), \qquad t\downarrow0.
\end{equation}
This relation implies that $Q(a,b)\ge 0$ for $a\ne b$ and $Q(a,a)\le 0$.
Further, we will always assume that
$$
\sum_{b\ne a}Q(a,b)=-Q(a,a)\quad\text{for any}\quad a\in\A.
$$
This is the infinitesimal analog of the condition $\sum_{b\in\A}P(t;a,b)=1$.

It is well known that the Chapman-Kolmogorov equation implies that $P(t)$
satisfies {\it Kolmogorov's backward equation\/}
\begin{equation}\label{eq3.1}
\frac{d}{dt}\, P(t)=Q P(t),\qquad t>0,
\end{equation}
with the initial condition
\begin{equation}\label{eq3.2}
P(0)\equiv \operatorname{Id}.
\end{equation}
Under certain additional conditions, $P(t)$ will also satisfy {\it Kolmogorov's
forward equation\/}
\begin{equation}\label{eq3.3}
\frac{d}{dt}\, P(t)= P(t) Q,\qquad t>0.
\end{equation}

One says that $Q$ is the {\it matrix of transition rates\/} for $(P(t))_{t\ge
0}$.

One often wants to define a Markov semigroup by giving the transition rates.
However, it may happen that this does not specify the semigroup uniquely (then
the backward equation has many solutions). Uniqueness always holds if  $\A$ is
finite or, more generally, if $\A$ is infinite but the diagonal entries
$Q(a,a)$ are bounded. However, these simple conditions do not suit our
purposes, and we need to go a little deeper into the general theory.

Let us write $Q$ in the form $Q=-q+\wt Q$, where $-q$ is the diagonal part of
$Q$ and $\wt Q$ is the off-diagonal part of $Q$. In other words,
$$
q(a,b)=-Q(a,a)\text{\bf 1}_{ab},\qquad \wt Q(a,b)=\begin{cases} Q(a,b),&a\ne
b,\\0,&a=b.\end{cases}
$$

Define $P^{[n]}(t)$ recursively by
$$
P^{[0]}(t)=e^{-t q},\qquad P^{[n]}(t)=\int_0^t e^{-\tau q}\wt Q
P^{[n-1]}(t-\tau)\,d\tau,\quad n\ge 1,
$$
and set
$$
\overline{P}(t)=\sum_{n=0}^\infty P^{[n]}(t),\qquad t\ge 0.
$$

\begin{theorem}[\cite{Fel40}]\label{3.1}

(i) The matrix $\overline{P}(t)$ is substochastic {\rm (}i.e., its elements are
nonnegative and $\sum_b P(t;a,b)\le 1${\rm).} Its elements are continuous in
$t\in[0,+\infty)$ and differentiable in $t\in(0,+\infty)$, and it provides a
solution of Kolmogorov's backward and forward equations \eqref{eq3.1},
\eqref{eq3.3} with the initial condition \eqref{eq3.2}.

(ii) $\overline{P}(t)$ also satisfies the Chapman-Kolmogorov equation.

(iii) $\overline{P}(t)$ is the minimal solution of \eqref{eq3.1} {\rm (}or
\eqref{eq3.3}{\rm )} in the sense that for any other solution $P(t)$ of
\eqref{eq3.1} {\rm (}or \eqref{eq3.3}{\rm )} with the initial condition
\eqref{eq3.2} in the class of substochastic matrices, one has $P(t;a,b)\ge
\overline{P}(t;a,b)$ for any $a,b\in\A$.
\end{theorem}

\begin{corollary}\label{3.2}
If the minimal solution $\overline{P}(t)$ is stochastic {\rm (}the sums of
matrix elements along the rows are all equal to 1\,{\rm )} then it is the
unique solution of \eqref{eq3.1} {\rm (}or \eqref{eq3.3}{\rm )} with the
initial condition \eqref{eq3.2} in the class of substochastic matrices.
\end{corollary}

If the minimal solution $\overline{P}(t)$ is stochastic one says that the matrix
of transition rates $Q$ is {\it regular\/}, cf. Proposition \ref{3.3}.

Observe that the construction of $\overline{P}(t)$ is very natural: the
summands $P^{[n]}(t;a,b)$ are the probabilities to go from $a$ to $b$ in $n$
jumps. The condition of $\overline{P}(t)$ being stochastic exactly means that
we cannot make infinitely many jumps in a finite amount of time.

A much more detailed account of Markov chains on countable sets can be found
e.g. in \cite{And91}.

Later on we will need the following sufficient condition for $\overline{P}(t)$
to be stochastic.

For any finite $X$, $X\subset \A$, $a\in X$, denote by $T_{a,X}$ the time of
the first exit from $X$ under the condition that the process is in $a$ at time
$0$. Formally, we can modify $\A$ and $Q$ by contracting all the states
$b\in\A\setminus X$ into one absorbing state $\wt{b}$ with $
Q_{\wt{b},c}\equiv0$ for any $ c\in X\cup \{\wt{b}\}$. We obtain a process with
a finite number of states for which the solution $\wt P(t)$ of the backward
equation is unique. Then $T_{a,X}$ is a random variable with values in
$(0,+\infty]$ defined by
$$
\Prob \{T_{a,X}\le t\}= \wt P(t;a,\wt b).
$$

\begin{proposition}\label{3.3}
Assume that for any $a\in\A$ and any $t>0$, $\varepsilon>0$, there exists a
finite set $X(\varepsilon)\subset \A$ such that\/ $\Prob
\{T_{a,X(\varepsilon)}\le t\}\le \varepsilon.$ Then the minimal solution
$\overline{P}(t)$ provided by Theorem \ref{3.1} is stochastic.
\end{proposition}

\begin{proof} Consider the modified process on the finite state space
$X(\epsi)\cup\{\wt{b}\}$ described above. Since its transition matrix $\wt
P(t)$ is stochastic,
$$
\sum_{b\in X(\epsi)} \wt P(t;a,b)=1-\wt P(t;a,\wt b)\ge 1-\epsi.
$$
The construction of the minimal solution as the sum of $P^{[n]}$'s, see above,
immediately implies that $P(t;a,b)\ge \wt P(t;a,b)$. Thus, $\sum_{b}
P(t;a,b)\ge 1-\epsi$ for any $\epsi>0$.
\end{proof}

\subsection{Collapsibility}

In what follows we will also need a result on {\it collapsibility\/} or {\it
lumpability\/} of Markov chains on discrete spaces. Let us describe it.

Let $E=\bigsqcup_{i\in I} E_i$ be a partition of the countable set $E$ on
disjoint subsets. Assume we are given a matrix $Q_E$ of transition rates on $E$
and a matrix $Q_I$ of transition rates on $I$ such that
\begin{equation}\label{eq3.31}
\sum_{b\in E_j} Q_E(a,b)=Q_I(i,j) \qquad \text{for any}\quad a\in E_i,\  i,j\in
I.
\end{equation}

Denote
$$
q_E(a)=-Q_E(a,a),\qquad q_I(i)=-Q_I(i,i).
$$

For any $i\in I$, let $Q_i$ be a matrix of transition rates on $E_i$ defined by
\begin{align*}
Q_i(a,b)=Q_E(a,b)\qquad \text{if}\quad a\ne b,\ &a,b\in E_i, \\
q_i(a)=-Q_i(a,a)=\sum_{b\in E_i,b\ne a}Q_i(a,b),\qquad &a\in E_i.
\end{align*}

Observe that $q_E(a)=q_i(a)+q_I(i)$ for any $a\in E_i$.

Denote by $\overline{P}_E(t)$ and $\overline{P}_I(t)$ the minimal solutions of
Kolmogorov's equations for $Q_E$ and $Q_I$, respectively.

\begin{proposition}\label{3.31}
Assume that for any $i\in I$, $Q_i$ is regular.
Then for any $t\ge 0$
\begin{equation}\label{eq3.30}
\sum_{b\in E_j}\overline{P}_E(t;a,b)=\overline{P}_I(t;i,j)
\end{equation}
for any $i,j\in I$ and any $a\in E_i$. In particular, if $Q_I$ is regular then
so is $Q_E$ and vice versa.
\end{proposition}

\begin{proof} Let us use notations $q$, $\wt Q$, and $P^{[n]}$ for the diagonal
and off-diagonal parts of the matrices of transition rates, and for the $n$th
terms in the series representations of minimal solutions, respectively.

The hypothesis means that the identity
$$
\sum_{b\in E_i}\sum_{n=0}^\infty P_i^{[n]}(t;a,b)=1,\qquad a\in E_i,\quad i\in
I,
$$
holds, where
$$
P_i^{[n]}(t)=\int\limits_{0\le t_1\le\dots\le t_n\le t}e^{-t_1 q_i}\wt Q_i
e^{-(t_2-t_1)q_i}\wt Q_i \cdots e^{-(t_n-t_{n-1})q_i}\wt
Q_ie^{-(t-t_n)q_i}dt_1\cdots dt_n
$$
and $P_i^{(0)}(t)=e^{-tq_i}$. Using the fact that
$q_i(\,\cdot\,)=q_E(\,\cdot\,)-q_I(i)$ on $E_i$, rewrite this identity as

\begin{multline}\label{eq3.32}
\sum_{b\in E_i}\sum_{n=0}^\infty \,\int\limits_{0\le t_1\le\dots\le t_n\le
t}\left(e^{-t_1 q_E}\text{\bf 1}_{E_i}\,\wt Q_E \,\text{\bf
1}_{E_i}e^{-(t_2-t_1)q_E}\text{\bf 1}_{E_i} \cdots\right.\\
\left.\cdots\text{\bf 1}_{E_i}\,\wt Q_E\,\text{\bf
1}_{E_i}e^{-(t-t_n)q_E}\right)(a,b)dt_1\cdots dt_n=e^{-tq_I(i)}.
\end{multline}

The probabilistic meaning of this formula is that the time that the minimal
solution $\overline{P}_E(t)$ started at $a\in E_i$ spends in $E_i$ is
exponentially distributed with rate $q_I(i)$, independent of $a$.

The minimal solution $\overline{P}_I(t)$ has the form
\begin{multline}
\overline{P}_I(t;i,j)=\sum_{n=0}^\infty \,\int\limits_{0\le s_1\le\dots\le
s_n\le t} \sum_{k_1,\dots,k_{n-1}\in I} e^{-s_1q_I(i)}\wt Q_I(i,k_1)
e^{-(s_2-s_1)q_I(k_1)}\cdots
\\\cdots \wt Q_I(k_{n-1},j)e^{-(t-s_n)q_I(j)}ds_1\dots ds_n.
\end{multline}

By \eqref{eq3.31},
\begin{equation}\label{eq3.33}
\wt Q_I(k,l)=\sum_{d\in E_l} \wt Q_E(c,d) \quad \text{for any} \quad k,l\in
I,\; k\ne l,\; c\in E_k.
\end{equation}

Substituting the right-hand side of \eqref{eq3.33} for each $\wt
Q_I(\,\cdot\,,\,\cdot\,)$ and the left-hand side of \eqref{eq3.32} for each
$e^{-sq_I(\,\cdot\,)}$, in the $n$th term we obtain the part of the series for
$\overline{P}_E(t;a,b)$ with $a\in E_i$ that takes into account trajectories
whose projections to $I$ make exactly $n$ jumps, and in addition to that there
is a summation over $b\in E_j$. Clearly, the summation over $n$ reproduces the
complete series for $\overline{P}_E(t;a,b)$ thus proving \eqref{eq3.30}.

The equivalence of stochasticity of $\overline{P}_E(t)$ and that of
$\overline{P}_I(t)$ immediately follows from summation of \eqref{eq3.30} over
$j\in I$.
\end{proof}

\subsection{Infinitesimal generator}

The last part of the general theory that we need involves generators of Markov
semigroups.

Assume that we have a regular matrix of transition rates $Q$. Let $(P(t))_{t\ge
0}$ be the corresponding Markov semigroup and assume in addition that it is
Feller.

The {\it generator\/} $A$ of the semigroup $(P(t))_{t\ge 0}$ is a linear
operator in $C_0(E)$ defined by
\begin{equation}\label{eq3.3a}
Af=\lim_{t\to +0}\frac{P(t)f-f}{t}\,.
\end{equation}
The set of $f\in C_0(E)$ for which this limit exists (in the norm topology of
$C_0(E)$) is called the {\it domain\/} of the generator $A$ and denoted by
$D(A)$. It is well known that the operator $A$ with $D(A)$ as above is closed
and dissipative.

It turns out that the domain $D(A)$ can be characterized by an apparently
weaker condition, which is easier to verify in practice:

\begin{proposition}\label{3.3b}
If $f\in C_0(E)$ is such that the limit in the
right-hand side of \eqref{eq3.3a} exists pointwise and the limit function belongs to
$C_0(E)$, then $f\in D(A)$, so that the limit actually holds in the norm
topology.
\end{proposition}

\begin{proof} The idea is that the set of couples of vectors $(f,g)\in
C_0(E)\times C_0(E)$, such that $g$ is the pointwise limit of the right-hand
side of \eqref{eq3.3a}, serves as the graph of a dissipative operator $\wt A$
extending $A$, whence $\wt A=A$. A detailed argument can be found in
\cite[\S4.8]{Ito06}. In fact, \cite{Ito06} considers the case of a compact
state space $E$. However, the proof goes through word-for-word; the only
property one needs is that for any $f\in C_0(E)$, $f$ attains its minimum if it
has negative values.
\end{proof}

The following statement is probably well known but we were not able to locate
it in the literature.

\begin{proposition}\label{3.3a}
Assume that for any $a\in E$ the set of $b$ such
that $Q(a,b)\ne 0$ is finite. Then
\begin{equation}\label{eq3.3b}
D(A)=\{f\in C_0(E)\mid Qf \in C_0(E)\},
\end{equation}
and for $f\in D(A)$, $Af=Qf$.
\end{proposition}

\begin{proof} First of all, due to the assumption on the matrix $Q$, $Qf$ is well
defined for any function $f$ on $E$. We will show that for any $f\in C_0(E)$ and
$a\in E$
\begin{equation}\label{eq3.3c}
\lim_{t\to +0}t^{-1}\sum_{b\in E} (P(t;a,b)-\text{\bf 1}_{a=b})f(b)=\sum_{b\in
E}Q(a,b)f(b).
\end{equation}
Then the claim of the proposition will follow from Proposition \ref{3.3b}.

Set
$$
X=\{a\}\cup\{b'\in E\mid Q(a,b')>0\}.
$$
By our hypothesis, this set is finite. We will show that
\begin{equation}\label{eq3.3d}
\sum_{b\in E\setminus X} P(t;a,b) =O(t^2),\qquad t\to 0.
\end{equation}
This would imply that we can keep only finitely many terms in \eqref{eq3.3c},
and then \eqref{eq3.3c} would follow from \eqref{eq3.0}.

Observe that the left-hand side of \eqref{eq3.3d} is the probability of the
event that the trajectory started at $a$ is outside of $X$ after time $t$.  In
order to exit $X$ the trajectory started at $a$ needs to make at least two
jumps. Assume that the first two jumps are $a\to a'\to a''$ with $a'\in X$.
Since $X$ is finite, the rates of leaving $a'$ (equal to $-Q(a',a')$) are
bounded from above, and the probability of leaving $X$ after time $t$ can be
estimated by
$$
-Q(a,a)\max_{a'\in X} (-Q(a',a'))\cdot t^2+o(t^3)=O(t^2),\qquad t\to +0,
$$
as required.
\end{proof}

\begin{corollary}\label{3.3c}
Under the hypothesis of Proposition \ref{3.3a} assume additionally that for any
$b\in E$ the set of $a\in E$ with $Q(a,b)\ne0$ is finite. Then any finitely
supported function $f$ on $E$ belongs to $D(A)$.
\end{corollary}

\begin{proof} Indeed, this follows immediately from Proposition \ref{3.3a}, since $Qf$
is finitely supported and hence belongs to $C_0(E)$.
\end{proof}

\section{Semigroups on $\gt_N$}\label{Semigroups}

The goal of this section is to define Markov semigroups $(P_N(t))_{t\ge 0}$ on
$E_N=\gt_N$ and prove that they are Feller.

\subsection{Case $N=1$. Birth and death process on $\Z$}\label{Case N=1}

Let $(u,u')$ and $(v,v')$ be two pairs of complex numbers such that
$(u+k)(u'+k)>0$ and $(v+k)(v'+k)>0$ for any $k\in\Z$. The condition on $(u,u')$
means that either $u'=\bar u\in\C\setminus\R$ or there exists $k\in\Z$ such
that $k<u,u'<k+1$; the condition on $(v,v')$ is similar. Note that $u+u'\in\R$
and $v+v'\in\R$. Assume additionally that $u+u'+v+v'>-1$.

Define a matrix of transition rates $\bigl[\D(x,y)\bigr]_{x,y\in\Z}$ with rows
and columns parameterized by elements of $E_1=\gt_1=\Z$ by
\begin{equation}\label{eq3.4}
\D(x,y)=\begin{cases} (x-u)(x-u'),
& \text{if  }y=x+1,\\
(x+v)(x+v'),&\text{if }y=x-1,\\
-(x-u)(x-u')-(x+v)(x+v'),&\text{if }y=x,\\
0,&\text{otherwise}.
\end{cases}
\end{equation}

In the corresponding Markov chain the particle would only be allowed to jump by
one unit at a time; such processes on $\Z_{\ge0}$ are usually referred to as
{\it birth and death processes\/}, while our Markov chain is an example of
so-called {\it bilateral birth and death processes\/} which were also
considered in the literature, see e.g. \cite[Section 17]{Fel57}, \cite{Pru63},
\cite{Yan90}.

Note that $\D(x,x\pm1)>0$ for all $x\in\Z$, because of the conditions imposed
on the parameters.

\begin{theorem}\label{3.4}

The matrix of transition rates $\D$ is regular. Moreover, the corresponding
Markov semigroup is Feller.
\end{theorem}

In what follows we denote this semigroup by $(P_1(t))_{t\ge 0}$.

The proof of Theorem \ref{3.4} is based on certain results from \cite{Fel59};
let us recall them first.

Consider a birth and death process on $\Z_{\ge 0}$ with transition rates given
by
$$
Q(x,y)=\begin{cases} \beta_x,
& \text{if  }y=x+1,\\
\delta_x,&\text{if }y=x-1,\\
-\beta_x-\delta_x,&\text{if }y=x,\\
0, &\text{otherwise}.
\end{cases}
$$
Here $\{\beta_x\}_{x\ge 0}$, $\{\delta_x\}_{x\ge 1}$ are positive numbers, and
we also set $\delta_0=0$.

The {\it natural scale\/} of the process is given by
\begin{equation}\label{eq3.5}
x_0=0, \qquad x_k=\sum_{l=0}^{k-1}
\frac{\delta_1\cdots\delta_l}{\beta_0\dots\beta_l},\qquad k=1,2,\dots\,, \qquad
x_\infty=\lim_{k\to\infty}x_k.
\end{equation}
Note that $x_\infty$ may be infinite. Denote by $\mathcal A$ the operator on
the space of functions on $\mathbb A=\{x_0,x_1,\dots\}$ defined by
$$
(\mathcal Af)(x_i)=-(\delta_i+\beta_i)f(x_i)+\delta_i
f(x_{i-1})+\beta_if(x_{i+1}),\qquad i=0,1,\dots\,.
$$

Fix $n>0$. Let $F_i(t)$ be the probability that the process started at $i$
reaches $n$ before time $t$. Let $G_n(t)$ be the probability that the process
started at $n$ reaches 0 before time $t$ and before the process escapes to infinity.

\begin{theorem}[\cite{Fel59}]\label{3.5}

(i) For any $a>0$ there exists exactly one function $u$ on $\mathbb A$ such
that $\mathcal A u=au$, $u(x_0)=1$. The function $u$ is strictly increasing:
$u(x_0)<u(x_1)<u(x_2)<\dots$ and satisfies
\begin{equation}\label{eq3.6}
u(x_n)=1+a\sum_{k=0}^{n-1} u(x_k)(x_n-x_k)\mu_k
\end{equation}
with
\begin{equation}\label{eq3.7}
\mu_k=\frac{\beta_0\cdots\beta_{k-1}}{\delta_1\cdots\delta_k},\quad
k=1,2,\dots,\qquad\mu_0=1.
\end{equation}
Furthermore,
\begin{equation}\label{eq3.8}
\frac{u(x_i)}{u(x_n)}=\int_0^{\infty}e^{-at}dF_i(t), \qquad 0\le i<n.
\end{equation}

(ii) With $u(\,\cdot\,)$ as above, set
$$
v(x_n)=u(x_n)\sum_{j=n}^\infty \frac{x_{j+1}-x_j}{u(x_j)u(x_{j+1})}\,,\qquad
n=0,1,\dots\,.
$$
This is a strictly decreasing function, and
\begin{equation}\label{eq3.9}
\frac{v(x_n)}{v(x_0)}=\int_0^\infty e^{-at}dG_n(t),\qquad n=1,2,\dots\,.
\end{equation}
Furthermore, $\lim_{n\to\infty} v(x_n)=0$ if $x_\infty=\infty$ and $\sum_n
x_n\mu_n$ diverges.
\end{theorem}

The following statement is contained in Feller's paper as well, but not
explicitly; for that reason we formulate it separately.

\begin{corollary}\label{3.6}
If $x_\infty=\infty$ then
$\lim_{n\to\infty}u(x_n)=\infty$.
\end{corollary}

\begin{proof} Let us estimate the sum in the right-hand side of \eqref{eq3.6}:
\begin{multline}
\sum_{k=0}^{n-1}u(x_k)(x_n-x_k)\mu_k\ge
\sum_{k=0}^{n-1}(x_n-x_k)\mu_k=\sum_{k=0}^{n-1}\sum_{l=k+1}^n(x_l-x_{l-1})\mu_k\\=
\sum_{l=1}^n\sum_{k=0}^{l-1}(x_l-x_{l-1})\mu_k=\sum_{l=0}^{n-1}(x_{l+1}-x_l)
\sum_{k=0}^l\mu_k \ge\sum_{l=0}^{n-1}(x_{l+1}-x_l)=x_n.
\end{multline}
Hence, $u(x_n)\ge 1+ax_n$, and the statement follows.
\end{proof}

Let us now apply Feller's results to our situation.

\begin{proof}[Proof of Theorem \ref{3.4}]
By Proposition \ref{3.3}, in order to show that the minimal solution is
stochastic it suffices to prove that the probability that the first passage
time from 0 to $n$ is below a fixed number, converges to zero as $n\to+\infty$.
Indeed, as shifts $x\to x+\const$ and sign change $x\mapsto -x$ keep our class
of processes intact, similar convergence would automatically hold for passage
times to the left, and also for passage times from any initial position. Denote
the first passage time from 0 to $n$ by $T_n$.

A simple coupling argument shows that $T_n$ stochastically dominates the first
passage time from 0 to $n$ for the birth and death process on $\Z_{\ge 0}$ with
the same transition rates (see \eqref{eq3.4}), except that the jump from $0$ to
$-1$ is forbidden. Let us denote this new first passage time by $\wt T_n$.
Thus,
$$
\Prob\{T_n\le t\}\le \Prob\{\wt T_n\le t\} \quad\text{for any $n\ge 0$ and
$t>0$}.
$$
For the application of Theorem \ref{3.4} we then set
$$
\beta_x=(x-u)(x-u'),\quad x\ge 0,\qquad \delta_x=(x+v)(x+v'),\quad x\ge
1,\qquad \delta_0=0.
$$
As
$$
\frac{\delta_1\cdots\delta_l}{\beta_0\cdots\beta_l}=\const
\frac{\Gamma(v+l+1)\Gamma(v'+l+1)}{\Gamma(-u+l+1)\Gamma(-u'+l+1)}\sim
\const\cdot\, l^{u+u'+v+v'},\quad l\to\infty,
$$
our original assumption $u+u'+v+v'>-1$ implies
\begin{equation}\label{eq3.10}
x_k\sim\const\cdot\, k^{u+u'+v+v'+1},\qquad k\to\infty,
\end{equation}
cf. \eqref{eq3.5}, and $x_\infty=\lim_{k\to \infty}x_k=\infty$. Therefore,
Corollary \ref{3.6} yields
$$
\lim_{n\to\infty}u(x_n)=\infty.
$$

On the other hand, from \eqref{eq3.8} with $i=0$ and any $a>0$ we obtain
$$
\frac 1{u(x_n)}=\int_0^\infty e^{-a\tau}dF_0(\tau)\ge \int_0^t e^{-a\tau}
dF_0(\tau)\ge e^{-at}\int_0^t dF_0(\tau)= e^{-at}\Prob \{\wt T_n\le t\}
$$
whence
$$
\Prob\{\wt T_n\le t\}\le \frac {e^{at}}{u(x_n)}\to 0,\qquad n\to +\infty.
$$
Since $\wt T_n$ is dominated by $T_n$, we have shown that our Markov chain does
not run away to infinity in finite time, and hence it is uniquely specified by the
transition rates. Let $(P_1(t))_{t\ge 0}$ be the corresponding semigroup.

We now need to prove that $(P_1(t))_{t\ge 0}$ is Feller. This is equivalent to
showing that $\lim_{n\to\pm\infty}P_1(t;n,i)=0$ for any $i\in\Z$ and $t>0$.

Shift and sign change invariance (see the beginning of the proof) imply that it
suffices to consider $i=0$ and $n\to+\infty$. Observe that $P_1(t;n,0)$ cannot
be greater than the probability that the first passage time from $n$ to $0$ is
not more than $t$. Let us denote this first passage time by $S_n$; we have
$P_1(t;n,0)\le \Prob\{S_n\le t\}$.

This first passage time is the same for our birth and death process on $\Z$ and
for its modification on $\Z_{\ge 0}$ that was used in the first part of the
proof. On the other hand, for the process on $\Z_{\ge 0}$ the Laplace transform
of $S_n$ is given by \eqref{eq3.9}.

By \eqref{eq3.7} we have, as $k\to\infty$,
$$
\mu_k=\frac{\beta_0\cdots\beta_{k-1}}{\delta_1\cdots\delta_k}=
\const\frac{\Gamma(-u+k)\Gamma(-u'+k)}{\Gamma(v+k+1)\Gamma(v'+k+1)}\sim \const
\cdot\, k^{-2-u-u'-v-v'}.
$$
Hence, cf. \eqref{eq3.10}
$$
x_k\mu_k\sim \const\cdot\, k^{-1}, \qquad k\to\infty,
$$
with a nonzero constant, and $\sum_n x_n\mu_n$ diverges. Theorem \ref{3.5}(ii)
then gives
$$
\lim_{n\to\infty}v(x_n)=0
$$
and using \eqref{eq3.9} and estimating the Laplace transform as above we obtain
$$
\Prob\{S_n\le t\}\le \frac{e^{at}v(x_n)}{v(x_0)}\to 0,\qquad n\to\infty.
$$
As $P_1(t;n,0)\le \Prob\{S_n\le t\}$, the proof of Theorem \ref{3.4} is
complete.
\end{proof}

\subsection{The case of general $N$}\label{The case of general N}

Let $N>1$ be a positive integer, and let $(u,u')$ and $(v,v')$ be as in
Subsection \ref{Case N=1}.

Define a matrix $\bigl[\D^{(N)}(\la,\nu)\bigr]_{\la,\nu\in\gt_N}$ of transition
rates with rows and columns parameterized by points of $E_N=\gt_N$ via
\begin{multline}\label{eq3.11}
\D^{(N)}(\la,\nu)=\frac{\di_N(\nu)}{\di_N(\la)}\Bigl(\D(l_1,n_1)
\text{\bf 1}_{\{l_i=n_i,i\ne 1\}}+\D(l_2,n_2)\text{\bf 1}_{\{l_i=n_i,i\ne 2\}}+
\dots\\+\D(l_N,n_N)\text{\bf 1}_{\{l_i=n_i,i\ne N\}}\Bigr) -d_N\cdot\text{\bf
1}_{\la=\nu}
\end{multline}
with $l_j=\la_j+N-j$, $n_j=\nu_j+N-j$, $1\le j\le N$, matrix
$\D(\,\cdot\,,\,\cdot\,)$ as in \eqref{eq3.4}, and
\begin{equation}\label{eq3.12}
d_N=\frac{N(N-1)(N-2)}{3}-(u+u'+v+v')\,\frac{N(N-1)}2.
\end{equation}

In other words, an off-diagonal element $\D^{(N)}(\la,\nu)$ can only be nonzero
if there exists exactly one index $i$ such that $\nu_i-\la_i=\pm 1$ while for
all other indices $j$ we have $\la_j=\nu_j$. Under this condition
$$
\D^{(N)}(\la,\nu)=\begin{cases} (l_i-u)(l_i-u')\prod\limits_{j\ne i}\dfrac
{l_i+1-l_j}{l_i-l_j},
& \text{if  }\nu_i-\la_i=1,\\
(l_i+v)(l_i+v')\prod\limits_{j\ne i}\dfrac {l_i-1-l_j}{l_i-l_j},&\text{if
}\nu_i-\la_i=-1.
\end{cases}
$$
With this explicit description, the diagonal entries of $\D^{(N)}$ have to be
defined by
\begin{equation}\label{eq3.13}
\D^{(N)}(\la,\la)=-\sum_{\nu\in\gt_N:\,\nu\ne \la} \D^{(N)}(\la,\nu),\qquad
\la\in\gt_N.
\end{equation}
The fact that \eqref{eq3.13} holds for $\D^{(N)}$ defined by \eqref{eq3.11}
will be proved in Step 1 of the proof of the following theorem.

\begin{theorem}\label{3.7}
The matrix of transition rates
$\D^{(N)}$ is regular. The corresponding semigroup $(P_N(t))_{t\ge 0}$ has the
form
\begin{equation}\label{eq3.14}
P_N(t;\la,\nu)=e^{-d_Nt}\, \frac{\di_N(\nu)}{\di_N(\la)} \det\bigl[
P_1(t;\la_i+N-i,\nu_j+N-j)\bigr]_{i,j=1}^N, \qquad \la,\nu\in\gt_N,
\end{equation}
with $(P_1(t))_{t\ge 0}$ as in Subsection \ref{Case N=1}. Moreover, this
semigroup is Feller.
\end{theorem}

\begin{proof} The proof of Theorem \ref{3.7} will consist of several steps.

\medskip

\noindent {\it Step 1.\/} Let us show that with definition \eqref{eq3.11},
relation \eqref{eq3.13} holds. It is convenient to encode signatures of length
$N$ by $N$-tuples of strictly decreasing integers via
$$
\la=(\la_1\ge\la_2\ge\dots\ge\la_N)\longleftrightarrow
(l_1>l_2>\dots>l_N),\qquad l_j=\la_j+N-j, \quad 1\le j\le N.
$$
This establishes a bijection between $\gt_N$ and the set
$$
\X_N=\{(x_1,\dots,x_N)\in \Z^N\mid x_1>x_2>\dots>x_N\}.
$$
In $\X_N$, matrix $\D^{(N)}$ from \eqref{eq3.11} takes the form
\begin{multline}
\D^{(N)}(X,Y)=\frac {V_N(Y)}{V_N(X)}\bigl(\D(x_1,y_1)\text{\bf
1}_{\{x_i=y_i,i\ne 1\}}+\D(x_2,y_2)\text{\bf 1}_{\{x_i=y_i,i\ne
2\}}+\dots\\+\D(x_N,y_N)\text{\bf 1}_{\{x_i=y_i,i\ne N\}}\bigr)-d_N\text{\bf
1}_{X=Y}
\end{multline}
with $X=(x_1,\dots,x_N)\in\X_N$, $Y=(y_1,\dots,y_N)\in\X_N$, and
$$
V_N(z_1,\dots,z_N)=\prod_{1\le i<j\le N} (z_i-z_j).
$$

In this notation, \eqref{eq3.13} is equivalent to
\begin{equation}\label{eq3.15}
(\D_1+\dots+\D_N)V_N(X)=d_NV_N(X),\qquad X\in\Z^N,
\end{equation}
where $\D_{j}$ denotes a linear operator on $\Z^N$ with
$$
\D_j(X,Y)=\D(x_j,y_j) \text{\bf 1}_{\{x_i=y_i,i\ne j\}}.
$$
Indeed, both sides of \eqref{eq3.15} are skew-symmetric, and restricting to
$\X_N$ yields \eqref{eq3.13}.

Let $\Delta$ and $\nabla$ be the standard forward and backward difference
operators on $\Z$:
$$
\Delta f(x)=f(x+1)-f(x),\qquad \nabla f(x)=f(x)-f(x-1)
$$
for any function $f:\Z\to\C$. Note that $\Delta\nabla=\Delta-\nabla$.

One easily checks that the operator $\D$ with matrix \eqref{eq3.4} has the form
\begin{equation}\label{eq3.15a}
\D=\sigma\Delta\nabla+\tau\Delta
\end{equation}
with
$$
\sigma=(x+v)(x+v'),\quad \tau=sx+(uu'-vv'),\quad s=-(u+u'+v+v').
$$
Hence, for any $m=0,1,2,\dots$
\begin{equation}\label{eq3.16}
\D x^m=\bigl(m(m-1)+sm\bigr)\cdot x^m+\text{lower degree terms},
\end{equation}
in paticular, $\D$ preserves the degree of a polynomial. This implies that the
left-hand side of \eqref{eq3.15} is a skew-symmetric polynomial of degree at
most $N(N-1)/2$. It must be divisible by the Vandermonde determinant $V_N(X)$,
and it remains to verify the constant prefactor. Following the highest in
lexicographic order term $x_1^{N-1}x_2^{N-2}\cdots x_N^0$ we see that upon the
action of $(\D_1+\dots+\D_N)$ it collects the coefficient
$$
\sum_{j=0}^{N-1} \bigl(j(j-1)+sj\bigr),
$$
which sums to \eqref{eq3.12}. Thus, \eqref{eq3.13} is proved.

\medskip

\noindent {\it Step 2.\/} Let us now prove that
\begin{equation}\label{eq3.17}
\sum_{\nu\in\gt_N} P_N(t;\la,\nu)=1,\qquad \la\in\gt_N,\quad t\ge 0,
\end{equation}
with $P_N$ as in \eqref{eq3.14}. In the space $\X_N$, \eqref{eq3.14} reads
\begin{equation}\label{eq3.18}
P_N(t;X,Y)=e^{-d_Nt}\, \frac{V_N(Y)}{V_N(X)} \det\bigl[
P_1(t;x_i,y_j)\bigr]_{i,j=1}^N, \qquad X,Y\in\X_N.
\end{equation}

Since the action of $\D$ in the space of polynomials $\R[x]$ is consistent with
filtration by degree, see \eqref{eq3.16}, the action of the corresponding
semigroup $(P_1(t))_{t\ge 0}$ in $\R[x]$ is well-defined, and \eqref{eq3.16}
implies
$$
\sum_{y\in\Z} P_1(t;x,y)y^m=e^{(m(m-1)+sm)t}\,x^m+\text{lower degree terms}.
$$
We obtain
\begin{multline}
\sum_{Y\in\X_N}P_N(t;X,Y)=\frac
1{N!}\sum_{Y\in\Z^N}P_N(t;X,Y)\\=\frac{e^{-d_Nt}}{V(X)}\sum_{\sigma\in S_N}
\operatorname{sgn}\sigma\sum_{Y\in\Z^N} P_1(t;x_{\sigma(1)},y_1)\cdots
P_1(t;x_{\sigma(N)},y_N)
y_1^{N-1}y_2^{N-2}\cdots y_{N-1}\\
=\frac{e^{-d_Nt}}{V(X)}\sum_{\sigma\in S_N}\operatorname{sgn}\sigma\,
e^{t\sum_{j=0}^{N-1} (j(j-1)+sj)}x_{\sigma(1)}^{N-1}
x_{\sigma(2)}^{N-2}\cdots x_{\sigma(N-1)}=1,
\end{multline}
where $S_N$ denotes the group of permutations of $\{1,\dots,N\}$. Note that the
first equality (change of the summation domain) holds because the expression
for $P_N(t;X,Y)$ is symmetric in $(y_j)$, and it vanishes if $y_i=y_j$ for
$i\ne j$.

\medskip

\noindent {\it Step 3.\/} Consider $N$ independent copies of the bilateral
birth and death process of Subsection \ref{Case N=1}, and denote by
$\pi_n(t;X,Y)$, $X,Y\in\X_N$, the probability that these processes started at
$x_1,\dots,x_N$ end up at $y_1,\dots,y_N$ after time $t$ having made a total of
$n$ jumps all together, and their trajectories had no common points at any time
moment between $0$ and $t$. We want to show that
\begin{equation}\label{eq3.19}
P^{[n]}_N(t;X,Y)=e^{-d_N t} \frac{V(Y)}{V(X)}\,\pi_n(t;X,Y),
\end{equation}
where $P^{[n]}_N$ is defined as in Section \ref{Generalities on Markov} using
$\D^{(N)}$ as the matrix of transition rates.

Indeed, computing $\pi_n$'s boils down to recurrence relations
$$
\gathered
\pi_0(t;X,Y)=e^{t\D_{ind}^{(N)}(X,Y)}\text{\bf 1}_{X=Y},\\
\pi_n(t;X,Y)=\int_0^t e^{\tau\D_{ind}^{(N)}(X,X)}\sum_{Z\in\X_N,\,Z\ne X}
\D_{ind}^{(N)}(X,Z) \pi_{n-1}(t-\tau;Z,Y)d\tau,\quad n\ge 1,
\endgathered
$$
where $\D_{ind}^{(N)}=\D_1+\dots+\D_N$ is the matrix of transition rates for
the $N$ independent birth and death processes.

For $n=0$, \eqref{eq3.19} follows from \eqref{eq3.11}. Assuming \eqref{eq3.19}
holds for $n-1$, we rewrite the recurrence relation for $\pi_n$'s as
\begin{multline} \pi_n(t;X,Y)=\int_0^t e^{\tau(\D^{(N)}(X,X)+d_N)}\\
\times \sum_{Z\in\X_N,\,Z\ne X} \frac{V(X)}{V(Z)}\,\D^{(N)}(X,Z) \cdot
e^{d_N(t-\tau)} \frac{V(Z)}{V(Y)}\, P_N^{[n-1]}(t-\tau;Z,Y)\,d\tau.
\end{multline}
Comparing with the recurrence relation for $P^{[n]}$, cf. Section
\ref{Generalities on Markov}, yields \eqref{eq3.19}.

\medskip

\noindent {\it Step 4.\/} Following Section \ref{Generalities on Markov} and
using \eqref{eq3.19}, we see that the minimal solution for the backward
equation with $\D^{(N)}$ as the matrix of transition rates, has the form
$$
\overline{P_N}(t;X,Y)=\sum_{n=0}^\infty P^{[n]}_N(t;X,Y)=e^{-d_N
t}\frac{V(Y)}{V(X)} \sum_{n=0}^\infty \pi_n(t;X,Y).
$$
The last sum is clearly equal to the probability that $N$ independent copies of
the bilateral birth and death process of Subsection \ref{Case N=1} started at
$x_1,\dots,x_N$ end up at $y_1,\dots,y_N$ after time $t$ without intermediate
coincidences and without any restriction on the number of jumps. Note that we
are using the fact that the birth and death process does not make infinitely
many jumps in finite time (minimal solution is stochastic), cf. Theorem
\ref{3.4}.

Such a probability of having nonintersecting paths is given by a celebrated
formula of Karlin-McGregor \cite{KM59}:
$$
\sum_{n=0}^\infty \pi_n(t;X,Y)=\det\bigl[ P_1(t;x_i,y_j)\bigr]_{i,j=1}^N,
\qquad X,Y,\in\X_N.
$$
Hence, the minimal solution $\overline{P_N}(t;X,Y)$ coincides with the
right-hand side of \eqref{eq3.18}, and by Step 2 it is stochastic. We have thus
shown that the matrix $\D^{(N)}$ of transition rates on $\gt_N$ is regular, and
the semigroup has the form \eqref{eq3.14} (or \eqref{eq3.18}).

\medskip

\noindent {\it Step 5.\/} To conclude the proof of Theorem \ref{3.7} it remains
to show that the Markov semigroup $(P_N(t))_{t\ge 0}$ is Feller. This is
equivalent to proving that
\begin{equation}\label{eq3.20}
\lim_{\la\to\infty}P_N(t;\la,\nu)=0,\qquad t\ge 0,\quad\nu\in\gt_N.
\end{equation}
But this immediately follows from \eqref{eq3.14} because we already know that
\eqref{eq3.20} holds for $N=1$ (Theorem \ref{3.4}), and $\di_N(\la)$ is always
at least 1.

\end{proof}

\section{Commutativity}\label{Commutativity}

The goal of this section is to address the question of compatibility of the
semigroups of Section \ref{Semigroups} and links of Section
\ref{Specialization}, cf. \eqref{eq1.3}.

\subsection{Parameterization}

As we shall see, in order for the commutativity relations \eqref{eq1.3} to be
satisfied, the parameters $(u,u',v,v')$ used to define semigroups
$(P_N(t))_{t\ge 0}$ need to depend on $N$. For that reason, introduce two new
pairs of parameters $(z,z')$ and $(w,w')$ that satisfy the same conditions as
$(u,u',v,v')$ before:
\begin{gather}\label{eq4.1}
(z+k)(z'+k)>0,\quad (w+k)(w'+k)>0\quad \forall k\in\Z;\qquad z+z'+w+w'>-1.
\end{gather}

Furthermore, for $N\ge 1$ define
\begin{equation}\label{eq4.2}
u_N=z+N-1,\quad u_N'=z'+N-1,\quad v=w,\quad v'=w',
\end{equation}
and let $(P_N(t))_{t\ge 0}$ be the Feller semigroup of the previous section
with parameters $(u,u',v,v')=(u_N,u'_N,v_N,v'_N)$.

We are aiming to prove the following statement.

\begin{theorem}\label{4.1}
With links ${\{\Lambda^{N+1}_N\}}_{N\ge 1}$ as in Subsection \ref{Spaces and
links} and semigroups $(P_N(t))_{t\ge 0}$ as above, the compatibility
relations \eqref{eq1.3} hold.
\end{theorem}

\subsection{Infinitesimal commutativity}

We first prove a version of \eqref{eq1.3} that involves matrices of transition
rates.

\begin{proposition}\label{4.2}
For any $N\ge 1$, $u,u',v,v'\in\C$, and
$\la\in\gt_{N+1},\nu\in\gt_{N}$, we have
\begin{equation}\label{eq4.3}
\sum_{\kappa\in \gt_{N+1}}
\widetilde\D^{(N+1)}(\la,\kappa)\Lambda^{N+1}_{N}(\kappa,\nu)=
\sum_{\rho\in\gt_{N}}\Lambda^{N+1}_{N}(\la,\rho)\D^{(N)}(\rho,\nu)
\end{equation}
or, in matrix notation,
$\widetilde\D^{(N+1)}\Lambda^{N+1}_{N}=\Lambda^{N+1}_{N}\D^{(N)}$,
where $\D^{(N)}$ is the operator defined by \eqref{eq3.11}, and in
$\widetilde\D^{(N+1)}$ we replace $N$ by $N+1$ and the parameters $(u,u')$ by
$(\tilde u,\tilde u')=(u+1,u'+1)$.
\end{proposition}

\begin{proof} We start with the following simple lemma.

\begin{lemma}\label{4.3}
Let $\bigl[A(\la,\nu)\bigr]_{\la\in\gt_{N+1},\,\nu\in\gt_{N}}$ be a matrix with
rows parameterized by $\gt_{N+1}$ and columns parameterized by $\gt_{N}$, and
such that each row of $A$ has finitely many nonzero entries. If for any
symmetric polynomial $F$ in $N$ variables and any $\la\in\gt_{N+1}$ we have
\begin{equation}\label{eq4.4}
\sum_{\nu\in\gt_{N}}A(\la,\nu)F(\nu_1+N-1,\nu_2+N-2,\dots,\nu_{N})=0,
\end{equation}
then $A(\la,\nu)\equiv 0$.
\end{lemma}

\begin{proof} Assume $A(\hat\la,\hat\nu)\ne 0$ for some $\hat\la$ and $\hat\nu$.
Let $\nu^{(1)},\dots,\nu^{(l)}\in\gt_{N}$ be all signatures different from
$\hat\nu$ and such that $A(\hat\la,\nu^{(j)})\ne 0$.

Set $x=(\hat\nu_1+N-1,\dots,\hat\nu_{N})\in \Z^{N}$ and
$$
y^{(j)}=\bigl(\nu^{(j)}_1+N-1,\dots,\nu^{(j)}_{N}\bigr)\in\Z^{N},\qquad
j=1,\dots,l.
$$
Observe that the orbits of the vectors $x, y^{(1)},\dots,y^{(l)}$ under the
group of permutations of the coordinates do not intersect. It follows that
there exists a polynomial $f$ in $N$ variables, which takes value $1$ on the
orbit of $x$ and vanishes on the orbits of the vectors $y^{(1)},\dots,y^{(l)}$.
Then for the symmetrized polynomial $F(z_1,\dots,z_{N})=\sum_{\sigma\in S_{N}}
f(z_{\sigma(1)},\dots,z_{\sigma(N)})$ the left-hand side of \eqref{eq4.4} is
equal to $N!A(\hat\la,\hat\nu)\ne 0$. Contradiction.
\end{proof}

Let us now introduce symmetric polynomials on which we will evaluate (in the
sense of Lemma \ref{4.3}) both sides of \eqref{eq4.3}. For a partition
(=signature with nonnegative coordinates) $\mu\in\gt_N$ and $c\in\C$ set
\begin{align*}
F_{\mu,c}(x_1,\dots,x_n)&=\frac
1{(N)_\mu}\,\frac{\det\bigl[(x_i+c)^{\downarrow(\mu_j+N-j)}\bigr]_{i,j=1}^N}{\prod_{1\le
i<j\le N}(x_i-x_j)}\,, \\
G_{\mu,c}(x_1,\dots,x_{n+1})&=\frac
1{(N+1)_\mu}\,\frac{\det\bigl[(x_i+c)^{\downarrow(\mu_j+N+1-j)}\bigr]_{i,j=1}^{N+1}}{\prod_{1\le
i<j\le N+1}(x_i-x_j)}\,, \
\end{align*}
where we assume $\mu_{N+1}=0$ and use the notation ($a\in\C$, $k\in\Z_{\ge 0}$)
\begin{gather*}
a^{\downarrow k}=a(a-1)\cdots(a-k+1),\quad
(a)_k=a(a+1)\cdots(a+k-1),
\quad a^{\downarrow 0}=(a)_0=1,\\
(a)_\mu=\prod_{j=1}^N (a-j+1)_{\mu_j}.
\end{gather*}
Clearly, $F_{\mu,c}$ and $G_{\mu,c}$ are symmetric polynomials in $N$ and $N+1$
variables, respectively. Moreover, for any fixed $c\in\C$, the polynomials
$\{F_{\mu,c}\}$ with $\mu$ ranging over all nonnegative signatures in $\gt_N$
form a linear basis in the space of all symmetric polynomials in $N$ variables.
Indeed, this follows from the fact that the highest degree homogeneous
component of $F_{\mu,c}$ coincides with the {\it Schur polynomial\/}
$s_\mu(x_1,\dots,x_N)$, and those are well known to form a basis, see e.g.
\cite{Macd95}.

Hence, to prove Proposition \ref{4.2} it suffices to verify that the two sides
of \eqref{eq4.3} give the same results when applied to $F_{\mu,c}$ for a fixed
$c$ and $\mu$ varying over nonnegative signatures of length $N$.

\begin{lemma}\label{4.4}
For any $\lambda\in\gt_{N+1}$, any nonnegative signature
$\mu\in\gt_N$, and $c\in\C$, we have
$$
\sum_{\nu\in\gt_N} \Lambda^{N+1}_N(\la,\nu)
F_{\mu,c}(\nu_1+N-1,\dots,\nu_N)=G_{\mu,c}(\la_1+N,\la_2+N-1,\dots,\la_{N+1}).
$$
\end{lemma}

\begin{proof} The argument is similar to that for relation (10.30) in
\cite{OO97}. Denote
$$
(x_1,\dots,x_{N+1})=(\la_1+N,\dots,\la_{N+1}),\quad
(y_1,\dots,y_N)=(\nu_1+N-1,\dots,\nu_N).
$$
Then $\nu\prec\la$ means $x_{i+1}\le y_i< x_i$ for all $i=1,\dots,N$. Taking
into account the definition of $\Lambda^{N+1}_N$, one sees that the relation in
question is equivalent to the following one
\begin{multline}\label{eq4.5}
\det\bigl[(x_i+c)^{\downarrow(\mu_j+N+1-j)}\bigr]_{i,j=1}^{N+1}
\\=\frac{(N+1)_\mu N!}{(N)_\mu}\sum_{\substack{y_1,\dots,y_N\in\Z\\
x_{i+1}\le y_i< x_i\text{  for all  } i}}
\det\bigl[(y_i+c)^{\downarrow(\mu_j+N-j)}\bigr]_{i,j=1}^N.
\end{multline}

The last column in the $(N+1)\times (N+1)$ matrix in the left-hand side of
\eqref{eq4.5} consists of 1's. Subtracting from the $i$th row the $(i+1)$st one
for each $i=1,\dots,N$, we see that the left-hand side is equal to the $N\times
N$ determinant
$$
\det\Bigl[(x_i+c+1)^{\downarrow(\mu_j+N+1-j)}-
(x_{i+1}+c+1)^{\downarrow(\mu_j+N+1-j)} \Bigr]_{i,j=1}^N.
$$
On the other hand, the summation in the right-hand side of \eqref{eq4.5} can be
performed in each row separately using the relation
$$
\sum_{y=a}^{b-1}(y+c)^{\downarrow m}=\frac{(b+c)^{\downarrow
(m+1)}-(a+c)^{\downarrow (m+1)}}{m+1}\,.
$$
Collecting constant prefactors completes the proof of Lemma \ref{4.4}:
\begin{multline} \frac{(N+1)_\mu
N!}{(N)_\mu\prod_{j=1}^N(\mu_j+N-j+1)}\\
=N!\prod_{j=1}^N\frac{(\mu_j+N+1-j)!(N-j)!}
{(N+1-j)!(\mu_j+N-j)!(\mu_j+N-j+1)}=1.
\end{multline}
\end{proof}

To conclude the proof of Proposition \ref{4.2} we want to prove that, for a
suitable fixed constant $c\in\C$, $\D^{(N)} F_{\mu,c}$ decompose on
$\{F_{\nu,c}\}$ in exactly the same way as $\wt\D^{(N+1)}G_{\mu,c}$ decompose
on $\{G_{\nu,c}\}$.

It is actually convenient to take $c=v$, where $v$ is one of the four
parameters $(u,u',v,v')$. With this specialization we prove

\begin{lemma}\label{4.5}
For any $\la\in\gt_N$ and any nonnegative signature
$\mu\in\gt_N$, with the notation $m_j=\mu_j+N-j$, $j=1,\dots,N$, we have
\begin{multline}
\sum_{\nu\in\gt_N} \D^{(N)}(\la,\nu)
F_{\mu,v}(\nu_1+N-1,\dots,\nu_N)\\= \left(\sum_{j=1}^N m_j(m_j-1)+s\sum_{j=1}^N
m_j-d_N\right)F_{\mu,v}(\la_1+N-1,\dots,\la_N)\\
+\sum_{j=1}^N \Bigl((m_j-1)(v'-v+m_j-1)+s(m_j-v-1)+uu'-vv'\Bigr)\text{\bf
1}_{\mu_j-1\ge
\mu_{j+1}}\\
\times F_{\mu-e_j,v}(\la_1+N-1,\dots,\la_N),
\end{multline}
where $d_N$ is as in \eqref{eq3.12}, $e_j=(0,\dots,0,1,0,\dots,0)$ with $1$ at
the $j$th place, and we assume $\mu_{N+1}=0$.
\end{lemma}

\begin{proof}
 We first compute, cf. \eqref{eq3.15a},
\begin{multline}
\mathcal D(x+v)^{\downarrow m}=(x+v)(x+v')\Delta\nabla
(x+v)^{\downarrow m}+(sx+uu'-vv')\Delta (y+v)^{\downarrow
m}\\=m(m-1)(x+v')(x+v)^{\downarrow(m-1)}+m(sx+uu'-vv')(x+v)^{\downarrow(m-1)}.
\end{multline}
This is the place where the choice of $c=v$ matters; for different values of
$c$ the expression for $\mathcal D(x+c)^{\downarrow m}$ would have been more
complicated.

Substituting
$$
\gathered x+v'=(x+v-m+1)+(v'-v+m-1),\\
sx+uu'-vv'=s(x+v-m+1)+(s(m-v-1)+uu'-vv'),
\endgathered
$$
we obtain
\begin{multline}
\mathcal D(x+v)^{\downarrow m}=\bigl(m(m-1)+sm\bigr)(x+v)^{\downarrow
m}\\+ \bigl((m-1)(v'-v+m-1)+s(m-v-1)+uu'-vv'\bigr)m(x+v)^{\downarrow(m-1)}.
\end{multline}
The statement now follows from \eqref{eq3.11} and the definition of
$F_{\mu,c}$.

\end{proof}

Let us complete the proof of Proposition \ref{4.2}.

Apply both sides of \eqref{eq4.3} to $F_{\mu,v}$ in the sense of Lemma
\ref{4.3}. Using Lemma \ref{4.4} we see that the left-hand side of
\eqref{eq4.3} turns into
$$
\sum_{\kappa\in \gt_{N+1}}
\widetilde\D^{(N+1)}(\la,\kappa)G_{\mu,v}(\kappa_1+N,\dots,\kappa_{n+1}),
$$
and repeating the arguments of Lemma \ref{4.5} we see that this is equal to
\begin{multline}\label{eq4.6}
\left(\sum_{j=1}^N \tilde m_j(\tilde m_j-1)+\tilde s\sum_{j=1}^N
\tilde m_j-\tilde d_{N+1}\right)G_{\mu,v}(\la_1+N-1,\dots,\la_N)\\
+\sum_{j=1}^N \Bigl((\tilde m_j-1)(v'- v+\tilde m_j-1)+\tilde s(\tilde
m_j-\tilde v-1)+\tilde u\tilde u'-vv' \Bigr)\text{\bf 1}_{\mu_j-1\ge
\mu_{j+1}}\\
\times G_{\mu-\delta_j,v}(\la_1+N-1,\dots,\la_N),
\end{multline}
where $\tilde m_j=\mu_j+N+1-1=m_j+1$, and tildes over the other constants mean
that in their definitions we replace $(u,u')$ by $(\tilde u,\tilde
u')=(u+1,u'+1)$.

On the other hand, by Lemmas \ref{4.4} and \ref{4.5} the right-hand side of
\eqref{eq4.3} equals

\begin{multline}\label{eq4.7}
\left(\sum_{j=1}^N m_j(m_j-1)+s\sum_{j=1}^N
m_j-d_N\right)G_{\mu,v}(\la_1+N-1,\dots,\la_N)\\
+\sum_{j=1}^N
\Bigl((m_j-1)(v'-v+m_j-1)+s(m_j-v-1)+uu'-vv'\Bigr)\\
\times \text{\bf 1}_{\mu_j-1\ge
\mu_{j+1}}G_{\mu-\delta_j,v}(\la_1+N-1,\dots,\la_N).
\end{multline}

It is a straightforward computation to see that all the coefficients in
\eqref{eq4.6} and \eqref{eq4.7} coincide. The proof of Proposition \ref{4.2} is
complete.

\end{proof}

\subsection{From matrices of transition rates to semigroups}

In order to complete the proof of Theorem \ref{4.1} we need the following
lemma.

\begin{lemma}\label{4.6}
If $f$ is a finitely supported function on $\gt_N$ then $\Lambda^{N+1}_Nf$ is
in the domain of the generator $A_{N+1}$ of the semigroup $(P_{N+1}(t))_{t\ge
0}$ (see Section \ref{Generalities on Markov} for the definition of the
generator and its domain).
\end{lemma}

Let us postpone the proof of Lemma \ref{4.6} until the end of this subsection
and proceed with the proof of Theorem \ref{4.1}.

In order to prove \eqref{eq1.3} it suffices to prove that the two sides are
equal when applied to a function $f$ on $\gt_N$ with finite support (as such
are dense in $C_0(\gt_N))$:
\begin{equation}\label{eq4.8}
P_{N+1}(t)\Lambda^{N+1}_{N}f=\Lambda^{N+1}_NP_N(t)f,\qquad t\ge 0,\quad
N=1,2,\dots\,.
\end{equation}
Let us denote the left and right-hand sides of \eqref{eq4.8} by $F_{left}(t)$
and $F_{right}(t)$. We will show that they solve the same Cauchy problem in the
Banach space $C_0(\gt_{N+1})$. Then \eqref{eq4.8} will follow from an abstract
uniqueness theorem for solutions of the Cauchy problem for vector functions
with values in a Banach space,
$$
\frac d{dt} F(t)=A F(t),\quad t>0,\qquad F(0)=\text{fixed vector},
$$
which holds under the assumptions that (1) $A$ is a closed dissipative
operator, (2) $F(t)$ is continuous for $t\ge 0$ and strongly differentiable for
$t>0$, and (3) $F(t)\in D(A)$ for $t\ge0$; see e.g. \cite[IX.1.3]{Kat80}.

In our situation, $A=A_{N+1}$ and the fixed vector is $\La^{N+1}_Nf$.
Obviously, both $F_{left}(t)$ and $F_{right}(t)$ are continuous for $t\ge0$ and
they have the same initial value $\La^{N+1}_Nf$ at $t=0$.

Let us check the differential equation for $F_{left}(t)$. By Lemma \ref{4.6} we
have $\Lambda^{N+1}_N f\in D(A_{N+1})$. Hence, $F_{left}(t)\in D(A_{N+1})$
(semigroups preserve the domains of the generators) and it satisfies
$$
\frac d{dt} F_{left}(t)= A_{N+1} F_{left}(t),\quad t>0.
$$

Let us turn to $F_{right}$. By Corollary \ref{3.3c},  $f$ belongs to $D(A_N)$.
It follows that the function $t\mapsto P_{N}(t)f$ is strongly differentiable
and
$$
\frac d{dt} P_N(t) f=A_N P_N(t) f=\D^{(N)}P_N(t), \qquad t>0.
$$
Hence, $F_{right}(t)$ is also strongly differentiable and for $t>0$
$$
\frac d{dt} F_{right}(t)=\Lambda^{N+1}_N \frac
d{dt}P_N(t)f=\Lambda^{N+1}_N\D^{(N)}P_N(t)f.
$$
By definition, the last expression should be understood as
$\Lambda^{N+1}_N(\D^{(N)}(P_N(t)f))$. However, since all rows of the matrices
$\Lambda^{N+1}_N$ and $\D^{(N)}$ have finitely many nonzero entries, we may
write
$$
\Lambda^{N+1}_N(\D^{(N)}(P_N(t)f))=(\Lambda^{N+1}_N\D^{(N)})P_N(t)f.
$$
By virtue of Proposition \ref{4.2}, this equals
$$
\wt\D^{(N+1)}\Lambda^{N+1}_NP_{N}(t)f=\wt\D^{(N+1)}F_{right}(t),
$$
so that
$$
\frac d{dt} F_{right}(t)=\wt\D^{(N+1)}F_{right}(t).
$$

Next, as $\frac d{dt}F_{right}(t)$ is in $C_0(\gt_{N+1})$, so is $\wt\D^{(N+1)}
F_{right}(t)$. By Proposition \ref{3.3a}, we may replace $\wt\D^{(N+1)}$ by
$A_{N+1}$, which gives the desired differential equation
$$
\frac d{dt} F_{right}(t)= A_{N+1} F_{right}(t),\quad t>0,
$$
and we conclude that $F_{left}=F_{right}$.

Thus, we have proved Theorem \ref{4.1} modulo Lemma \ref{4.6}.

\begin{proof}[Proof of Lemma \ref{4.6}]
Let $f$ be a finitely supported function on $\gt_N$, $g=\Lambda^{N+1}_N f$.
Proposition \ref{2.2} says that $g\in C_0(\gt_{N+1})$, and by Proposition
\ref{3.3a} it suffices to check that $\D^{(N+1)}g\in C_0(\gt_{N+1})$. We have
$$
(\D^{(N+1)}g)(\la)=\sum_{\epsi:\,\la+\epsi\in\gt_{N+1}}
\D^{(N+1)}(\la,\la+\epsilon) \bigl(g(\la+\epsilon)-g(\la)\bigr),
$$
where $\la\in\gt_{N+1}$, $\epsi$ ranges over $\{\pm e_j\}_{j=1,\dots,N+1}$,
with $(e_j)$ being the standard basis in $\R^{N+1}$, and
$\D^{(N+1)}(\la,\la+\epsilon)$ are off-diagonal entries of the matrix
$\D^{(N+1)}$.

Without loss of generality we may assume that $f$ is the delta-function at some
$\nu\in\gt_N$. We obtain
$$
g(\la)=\dfrac{N!\,\prod_{1\le i<j\le N}(\nu_i-i-\nu_j+j)}{\prod_{1\le i<j\le
N+1}(\la_i-i-\la_j+j)}\,\cdot\text{\bf 1}_{\nu\prec\la},
$$
and
\begin{equation}\label{eq4.11}
(\D^{(N+1)}g)(\la)=\sum_{i=1}^{N+1}\sum_{\epsi_i
=\pm1}\frac{\di_{N+1}(\la+\epsi_ie_i)}{\di_{N+1}(\la)}\,
 \D(l_i,l_i+\epsi_i)
\bigl(g(\la+\epsi_ie_i)-g(\la)\bigr)
\end{equation}
where $l_j=\la_j+N+1-j$, $j=1,\dots,N+1$, and we assume
$\di_{N+1}(\la+\epsi_ie_i)=0$ in case $\la+\epsi_ie_i\notin\gt_{N+1}$ (this is
supported by the explicit formula for $\di_{N+1}(\,\cdot\,)$).

Observe that for
$$
\wt g(\la)=\dfrac{N!\,\prod_{1\le i<j\le N}(\nu_i-i-\nu_j+j)}{\prod_{1\le
i<j\le N+1}(\la_i-i-\la_j+j)}=\frac{\const_1}{\di_{N+1}(\la)}
$$
(we removed the factor $\text{\bf 1}_{\nu\prec\la}$ from $g(\la)$ above), we
have

\begin{multline}\label{eq4.12}
(\D^{(N+1)}\wt g)(\la)=\frac{\const_2}{\di_{N+1}^2(\la)}\\
\times\sum_{i=1}^{N+1}\sum_{\epsi_i=\pm1} \D(l_i,l_i+\epsi_i)
\bigl(\di_{N+1}(\la+\epsi_ie_i)-\di_{N+1}(\la)\Bigr)
=\frac{\const_3}{\di_{N+1}(\la)}\,,
\end{multline}
where we used \eqref{eq3.15}.

Next, observe that the function
$$
(\D^{(N+1)}\wt g)(\la)\text{\bf
1}_{\nu\prec\la}=\frac{\const_3}{\di_{N+1}(\la)}\text{\bf 1}_{\nu\prec\la}
$$
belongs to $\C_0(\gt_{N+1})$. Indeed, if $\la$ goes to infinity inside the
subset $\{\la: \nu\prec\la\}$ then $\la_i-\la_j\to+\infty$ for at least one
couple $i<j$ of indices, which entails $\di_{N+1}\la\to+\infty$.

The discrepancy between $(\D^{(N+1)}\wt g)(\la)\text{\bf 1}_{\nu\prec\la}$ and
$(\D^{(N+1)} g)(\la)$ (or rather between the summations in \eqref{eq4.11} and
\eqref{eq4.12}) comes from values of $i$ and $\epsi_i$ such that either
$\nu\prec\la$ but $\nu\not\prec(\la+\epsi_ie_i)$, or $\nu\prec\la+\epsi_ie_i$
but $\nu\not\prec\la$. In both cases, for that value of $i$, the quantities
$\la_i$, $l_i$, and $\D(l_i,l_i+\epsi_i)$ must remain bounded as $\nu$ is
fixed.

Note that $\la\to\infty$ inside the subset
$$
\{\la\in\gt_{N+1}: \text{\rm $\nu\prec\la$ or $\nu\prec\la+\epsi_i e_i$ for some
$i$}\},
$$
then either $\la_1\to +\infty$ or $\la_{N+1}\to-\infty$, or both, while all
other $\la_j$ remain bounded from both sides. But then a direct inspection of
the summands in \eqref{eq4.11} and \eqref{eq4.12} that contribute to the
discrepancy shows that they converge to zero as $\la\to\infty$. Hence,
$(\D^{(N+1)} g)(\la)\in C_0(\gt_{N+1})$.
\end{proof}

\section{Invariant measures}\label{zw-measures}

In three previous sections we defined a chain of countable sets
$\{E_N=\gt_N\}_{N\ge 1}$, constructed links $\Lambda^{N+1}_N$ between then, and
identified the boundary $E_\infty=\Omega$. Furthermore, for any quadruple of
complex parameters $(z,z',w,w')$ satisfying \eqref{eq4.1} we constructed Feller
semigroups $(P_N(t))_{t\ge 0}$ on $\gt_N$ and showed that they are compatible
with the links; by Proposition \ref{1.4} this yields a Feller semigroup
$(P(t))_{t\ge 0}$ on $\Omega$.

The goal of this section is to exhibit an invariant measure for $(P(t))_{t\ge
0}$.

\subsection{$zw$-measures}

Let $z,z',w,w'$ be complex parameters satisfying \eqref{eq4.1}. As was
pointed out in Subsection \ref{Case N=1}, this is equivalent to saying that
each pair $(z,z')$ and $(w,w')$ belongs to one (or both) of the sets
$$
\gathered
\{(\zeta,\zeta')\in(\C\setminus\Z)^2\mid \zeta'=\bar{\zeta}\}\quad\text{and}\quad\\
\{(\zeta,\zeta')\in(\R\setminus\Z)^2\mid m<\zeta,\zeta'<m+1 \text{ for some } m\in\Z\},
\endgathered
$$
and also $z+z'+w+w'>-1$.

For $\la\in\gt_N$ set
$$
M_{z,z',w,w'\mid N}(\la)= (\const_N)^{-1}\cdot M'_{z,z',w,w'\mid N}(\la)
$$
where
\begin{multline}
M'_{z,z',w,w'\mid N}(\la)= \prod_{i=1}^N
\bigg(\frac1{\Gamma(z-\la_i+i)\Gamma(z'-\la_i+i)}\\
\times\frac1{\Gamma(w+N+1+\la_i-i)\Gamma(w'+N+1+\la_i-i)}\bigg)\cdot
(\di_N(\la))^2,
\end{multline}
and
$$
\const_N=\sum_{\la\in\gt_N}M'_{z,z',w,w'\mid N}(\la)
$$
is the normalizing constant depending on $z,z',w,w',N$.

\begin{theorem}[\cite{Ols03}]\label{5.1}
Under our assumptions on the parameters, for any $N\ge 1$, $M_{z,z',w,w'\mid N}$
is a probability measure, we call it the $N$th $zw$-measure. Moreover, these
measures are consistent with the links,
$$
M_{z,z',w,w'\mid N}=M_{z,z',w,w'\mid N+1}\,\Lambda^{N+1}_N, \qquad N\ge 1,
$$
with $\Lambda^{N+1}_N$ as in Subsection \ref{Spaces and links}.
\end{theorem}

Theorem \ref{5.1} implies that the system $(M_{z,z',w,w'\mid N})_{N\ge 1}$
defines a probability measure $M_{z,z',w,w'}$ on the boundary $\Omega$ that we
call the {\it spectral $zw$-measure\/}, cf. Theorem \ref{2.1}, and a character
of the infinite-dimensional unitary group $U(\infty)$, cf.\cite{Ols03}. For
$z'=\bar{z}$ and $w'=\bar{w}$ one can find a geometric construction of the
corresponding representations of $U(\infty)$ in \cite{Ols03}. There is also a
fairly simple ``coordinate-free'' description of general $zw$-measures that we
now give, cf. \cite{BO05c}.

Let $\T$ be the unit circle in $\C$ and $\T^N$ be the product of $N$ copies of
$\T$ (the $N$--dimensional torus). For any $\la\in\gt_N$, the character
$\chi^\la$ of the corresponding irreducible representation $\pi_\la$ of $U(N)$
can be viewed  as a symmetric function on $\T^N$, where coordinates are
interpreted as eigenvalues of unitary matrices. Explicitly, the character is
given by the (rational) Schur function
$$
\chi^\la(u_1,\dots,u_N)=s_\la(u_1,\dots,u_N) =
\frac{\det\bigl[u_i^{\la_j+N-j}\bigr]_{1\le i,j\le N}}
{\det\bigl[u_i^{N-j}\bigr]_{1\le i,j\le N}}\,.
$$
Consider the Hilbert space $H_N$ of symmetric functions on $\T^N$, square
integrable with respect to the measure
$$
\frac1{N!}\, \prod_{1\le i<j\le N}|u_i-u_j|^2 \prod_{i=1}^N du_i\,,
$$
which is the push--forward of the normalized Haar measure on $U(N)$ under the
correspondence $U\mapsto(u_1,\dots,u_N)$. Here $du_i$ is the normalized
invariant measure on the $i$th copy of $\T$.

Given two complex numbers $z,w$, we define a symmetric function on $\T^N$ by
$$
f_{z,w\mid N}(u)=\prod_{i=1}^N (1+u_i)^z(1+\bar u_i)^w.
$$
If $\Re(z+w)>-\frac12$ then $f_{z,w\mid N}$ belongs to the space $H_N$. Let
$(z',w')$ be another couple of complex numbers with $\Re(z'+w')>-\frac12$. We
set
$$
M_{z,z',w,w'\mid N}(\la)=\frac{(f_{z,w\mid N},
\chi_\la)(\chi_\la,f_{\overline{w'},\, \overline{z'}\mid N})} {(f_{z,w\mid
N},f_{\overline{w'}, \,\overline{z'}\mid N})}\,,  \qquad \la\in\gt_N,
$$
where $(\,\cdot\,,\,\cdot\,)$ is the inner product in $H_N$. It turns out that
this definition leads us to the explicit formula given above.

The spectral $zw$-measures were the subject of an extensive investigation in
\cite{BO05a} the upshot of which is the statement that with $\omega\in\Omega$
distributed according to $M_{z,z',w,w'}$, its coordinates
$$
\left\{\tfrac 12+\alpha_i^+,\tfrac 12-\beta_i^+, -\tfrac 12+\beta^-_i,-\tfrac
12-\alpha_i^-\right\}_{i=1}^\infty
$$
(where possible zero values of $\alpha^\pm_i$ and $\beta^\pm_i$ should be
removed) form a {\it determinantal point process\/} on $\R\setminus\{\pm \frac
12\}$ with an explicit correlation kernel. See \cite{BO05a}, \cite{BO05b} for
details.

\subsection{Invariance}

The main statement of this section is

\begin{theorem}\label{5.2}
For any quadruple $(z,z',w,w')$ of parameters satisfying \eqref{eq4.1}, the
spectral $zw$-measure $M_{z,z',w,w'}$ is the unique invariant probability
measure with respect to the semigroup $(P(t))_{t\ge 0}$.
\end{theorem}

\begin{proof} Let us prove the invariance first. By Subsection \ref{Invariant measures}, it suffices to verify
that for each $N\ge 1$, the $N$th level $zw$-measure is invariant with respect
to $(P_N(t))_{t\ge 0}$. We will check this fact on the level of matrices of
transition rates:
$$
\sum_{\la\in\gt_N} M_{z,z',w,w'\mid N}(\la)\,\D^{(N)}(\la,\nu)=0,\qquad N\ge
1,\quad \nu\in\gt_N.
$$
Since it is easy to check that $\D^{(N)}$ is reversible with respect to
$M_{z,z',w,w'\mid N}(\la)$,
$$
M_{z,z',w,w'\mid N}(\la)\,\D^{(N)}(\la,\nu)=\D^{(N)}(\la,\nu)\,M_{z,z',w,w'\mid
N}(\nu),\qquad \la,\nu\in\gt_N,
$$
an argument in Section 3 of \cite{Kel83} shows that the invariance on the level
of transition rates implies the invariance with respect to the corresponding
semigroup.

As in the proof of Theorem \ref{3.7}, it is convenient to employ the bijection
$\la\leftrightarrow (\la_j+N-j)_{1\le j\le N}$ between $\gt_N$ and $\X_N$, see
Subsection \ref{The case of general N} for the notation and also recall that we
are using parameterization \eqref{eq4.2}. Under the bijection of $\gt_N$ and
$\X_N$, the desired identity takes the form (removing irrelevant prefactors)

\begin{multline}\label{eq5.1}
\sum_{X\in\X_N} \left(\prod_{i=1}^N W(x_i)\right) V_N(X)
\Bigl(\bigl(\D(x_1,y_1)\text{\bf 1}_{\{x_i=y_i,i\ne
1\}}+\dots\\
\dots+\D(x_N,y_N)\text{\bf 1}_{\{x_i=y_i,i\ne N\}}\bigr)-d_N\text{\bf
1}_{X=Y}\Bigr)=0,
\end{multline}
where
$$
W(x)=\frac1{\Gamma(z+N-x)\Gamma(z'+N-x)\Gamma(w+1+x)\Gamma(w'+1+x)}\,,\quad
x\in\Z.
$$

Let $p_0=1,p_1,p_2,\dots$, $\deg p_j=j-1$, be monic orthogonal polynomials on
$\Z$ corresponding to the weight function $W(x)$.  As
$$
W(x)=O(|x|^{-z-z'-w-w'-2N}), \qquad x\to\infty,
$$
the assumption $z+z'+w+w'>-1$ implies that $W(x)$ has at least $2N-3$ finite
moments, and polynomials $p_j$ with $j=0,1\dots,N-1$ are well defined.

Polynomials $\{p_j\}$ can be written explicitly in terms of the hypergeometric
function ${}_3F_2$ evaluated at 1. They were discovered by R.~Askey
\cite{Ask87}, and independently by P.~Lesky \cite{Les97}, \cite{Les98}; see
also the recent book \cite[\S5.3, Theorem 5.2, Case IIIc]{KLS10}. We call them
the {\it Askey-Lesky polynomials\/}.

The Askey-Lesky polynomials are eigenfunctions of the operator $\D$ on $\Z$,
see \cite[\S7]{BO05a}:
$$
\sum_{y\in\Z}\D(x,y)p_j(y)=\gamma_j p_j(x) \qquad \forall x\in\Z, \quad
j=0,1,2,\dots,
$$
where
$$
\gamma_j=j((j-1)-(u_N+u'_N+v_N+v'_N)).
$$
Multiplying both sides by $W(x)$ and using the fact that $W(x)\D(x,y)$ is
symmetric with respect to transposition $x\leftrightarrow y$ we obtain
\begin{equation}\label{eq5.2}
\sum_{x\in\Z} p_j(x)W(x)\D(x,y)=\gamma_j p_j(y)\qquad \forall y\in\Z, \quad
j=0,1,2,\dots\,.
\end{equation}
Let us rewrite the Vandermonde determinant in the left-hand side of
\eqref{eq5.1} as
$$
V_N(x)=\pm\det\bigl[ p_{i-1}(x_j)\bigr]_{i,j=1}^N.
$$

Applying operators $\D_1,\dots,\D_N$ to individual columns in this determinant
multiplied by $W(x_1)\cdots W(x_N)$ according to \eqref{eq5.2}, and recalling
the definition of $d_N$, we obtain \eqref{eq5.1}.

Let us now prove uniqueness. As explained in Subsection \ref{Invariant
measures}, it suffices to show that the $N$th level $zw$-measure is the unique
invariant probability measure for $(P_N(t))_{t\ge 0}$ for any $N\ge 1$. But
uniqueness of invariant measures holds in general for irreducible Markov chains
on countable sets, see e.g. Theorem 1.6 in \cite{And91}.

\end{proof}

\section{Stochastic dynamics on paths. General formalism}\label{Stochastic
dynamics}

\subsection{Overview}\label{Overview}

Let us return to the general setting of Section \ref{Abstract construction} and
assume that all $E_N$'s are discrete. For $N=1,2,\dots$ set
\begin{equation}\label{eq6.1}
\A^{(N)}=\Bigl\{(x_1,\dots,x_n)\in\A_1\times\cdots\times\A_N\mid
\prod_{k=1}^{N-1}\Lambda_{k}^{k+1}(x_{k+1},x_{k})\ne 0\Bigr\}.
\end{equation}
There are natural projections $\Pi^{N+1}_N:\A^{(N+1)}\to\A^{(N)}$ consisting in
forgetting the last coordinate; let $\A^{(\infty)}=\varprojlim \A^{(N)}$, where
the projective limit is taken with respect to these projections. Obviously,
$\A^{(\infty)}$ is a closed subset of the infinite product space
$\prod_{N=1}^\infty E_N$. Thus, elements of $E^{(\infty)}$ are some infinite
sequences. Let $\Pi^\infty_N:\A^{(\infty)}\to\A^{(N)}$ be the map that extracts
the first $N$ members of such a sequence.

\begin{definition}\label{6.1}
We say that a probability measure $\mu^{(N)}$ on $\A^{(N)}$ is {\it central\/}
if there exists a probability measure $\mu_N$ on $\A_N$ such that
\begin{equation}\label{eq6.2}
\mu^{(N)}(x_1,\dots,x_N)=\mu_N(x_N)\Lambda^{N}_{N-1}(x_N,x_{N-1})
\cdots\Lambda^2_1(x_2,x_1)
\end{equation}
for any $(x_1,\dots,x_N)\in\A^{(N)}$. Relation \eqref{eq6.2} establishes a
bijection between probability measures on $\A_N$ and central probability
measures on $\A^{(N)}$.

We say that $\mu^{(\infty)}\in \mathcal M_p(E^{(\infty)})$ is central if all
its pushforwards under projections $\Pi^\infty_N$ are central. Relation
\eqref{eq6.2} also establishes a bijection between central measures on
$E^{(\infty)}$ and elements of $\varprojlim\mathcal M_p(\A_N)$ of Subsection
\ref{Boundary}.

Finally, we say that a Markov semigroup $(P^{(N)}(t))_{t\ge 0}$ on $\A^{(N)}$
is {\it central\/} if the associate linear operators in $\mathcal M(\A^{(N)})$
map central measures to central measures.
\end{definition}

Clearly, a central Markov semigroup $(P^{(N)}(t))_{t\ge 0}$ defines a Markov
semigroup on $\A_N$ --- in order to obtain $\mu_N P_N(t)$ for $\mu_N\in\mathcal
M_p(\A_N)$ one needs to define $\mu^{(N)}$ via \eqref{eq6.2}, evaluate
$\mu^{(N)}P^{(N)}(t)$, and read off a measure on $\A_N$ using Definition
\ref{6.1}.

\begin{proposition}\label{6.2}
Let $(P^{(N)}(t))_{t\ge 0}$, $N\ge 1$, be a sequence of central Markov
semigroups on $\A^{(N)}$'s that are compatible with the system of projections:
$$
P^{(N+1)}(t)\circ\Pi^{N+1}_N=\Pi^{N+1}_N\circ P^{(N)}(t),\qquad t\ge 0,\quad
N\ge 1.
$$
Then the corresponding Markov semigroups $(P_N(t))_{t\ge 0}$ on $\A_N$, $N\ge
1$, are compatible with projections $\Lambda^{N+1}_N$ as in \eqref{eq1.3}.
\end{proposition}

\begin{proof} Follows from the fact that if $\mu^{(N+1)}$ and $\mu_{N+1}$ are
related as in Definition \ref{6.1} then $\mu^{(N+1)}\Pi^{N+1}_N$ and
$\mu_{N+1}\Lambda^{N+1}_N$ are also related in the same way.
\end{proof}

The goal of this section and the next one is to construct central Markov
semigroups $(P^{(N)}(t))_{t\ge 0}$ that would yield, as in Proposition
\ref{6.2}, semigroups $(P_N(t))_{t\ge 0}$ on $E_N=\gt_N$ that we dealt with in
the previous sections. One reason for such a construction is the fact that for the
Gelfand-Tsetlin graph, the isomorphism between central measures on
$\gt^{(\infty)}$ and probability measures on the boundary $\Omega$, cf.
Definition \ref{1.2}, is somewhat explicit, see Section \ref{Stochastic
dynamics2} below. Thus, $(P^{(N)}(t))_{t\ge 0}$ can be thought of as providing a
more ``hands-on'' description of the corresponding semigroup $(P(t))_{t\ge 0}$
on $\Omega$.

\subsection{Construction of bivariate Markov chains}\label{Construction}

Let $\A$ and $\A^*$ be countable sets, and let $Q$ and $Q^*$ be matrices of
transition rates on these sets. Let
$\Lambda=[\Lambda(x^*,x)]_{x^*\in\A^*,x\in\A}$ be an additional stochastic
matrix which we view as a stochastic link between $\A^*$ and $\A$.

We will assume that for each of the three matrices $Q$, $Q^*$, and $\Lambda$,
each row contains only finitely many nonzero entries. In addition, we assume
the relation
\begin{equation}\label{eq6.3}
\sum_{x\in\A}\Lambda(x^*,x)Q(x,y)=\sum_{y^*\in\A^*}
Q^*(x^*,y^*)\Lambda(y^*,y),\qquad x^*\in \A^*,\ y\in\A,
\end{equation}
or $\Lambda Q=Q^*\Lambda$ in matrix notation.

Observe that in case $\Lambda(x^*,y)=0$, the diagonal entries $Q(x,x)$ and
$Q(x^*,x^*)$ give no contribution to \eqref{eq6.3}, and the commutativity
relation can be rewritten as
\begin{equation}\label{eq6.4}
\sum_{x\in\A,x\ne y}\Lambda(x^*,x)Q(x,y)=\sum_{y^*\in\A^*,y^*\ne x^*}
Q^*(x^*,y^*)\Lambda(y^*,y),\qquad x^*\in \A^*,\ y\in\A.
\end{equation}
We will denote the above expression by $\Delta(x^*,y)$; it is only defined if
$\Lambda(x^*,y)=0$.

In what follows we also use the notation
$$
q_x=-Q(x,x),\quad x\in \A;\qquad q^*_{x^*}=-Q^*(x^*,x^*),\quad x^*\in \A^*.
$$

Consider the bivariate state space
$$
\A^{(2)}=\{(x^*,x)\in\A^*\times\A\mid \Lambda(x^*,x)\ne 0\}.
$$

We want to construct a Markov chain on $\A^{(2)}$ that would satisfy two
conditions:

\noindent $\bullet$\quad The projection of this Markov chain to $\A$ gives the
Markov chain defined by $Q$;

\noindent $\bullet$\quad It preserves the class of measures on $\A^{(2)}$
satisfying $\Prob(x\vert x^*)=\Lambda(x^*,x)$;

\noindent $\bullet$\quad In this class of measures, the projection of this
Markov chain to $\A^*$ gives the Markov chain defined by $Q^*$.

To this end, define a matrix $Q^{(2)}$ of transition rates on $\A^{(2)}$ with
off-diagonal entries given by
$$
Q^{(2)}\bigl((x^*,x),(y^*,y)\bigr)=\begin{cases} Q(x,y),&x^*=y^*,\\
Q^*(x^*,y^*)\,\dfrac{\Lambda(y^*,x)}{\Lambda(x^*,x)}\,,& x=y,\\
Q(x,y)\,\dfrac{Q^*(x^*,y^*)\Lambda(y^*,y)}{\Delta(x^*,y)}\,,&\Lambda(x^*,y)=0,
\Delta(x^*,y)\ne 0,
\\0,&\text{otherwise}.
\end{cases}
$$

Note that $\Lambda(x^*,y)=0$ implies $x^*\ne y^*$ and $x\ne y$ (provided that
$(x^*,x),(y^*,y)$ are in $\A^{(2)}$) so all the cases in the above definition
are mutually exclusive.

The diagonal entries $Q^{(2)}\bigl((x^*,x),(x^*,x)\bigr)$ with
$(x^*,x)\in\A^{(2)}$ are defined by
$$
-Q^{(2)}\bigl((x^*,x),(x^*,x)\bigr)=q^{(2)}_{(x^*,x)}:=\sum_{(y^*,y)\ne
(x^*,x)} Q^{(2)}\bigl((x^*,x),(y^*,y)\bigr).
$$
Clearly, any row of $Q^{(2)}$ also has only finitely many nonzero entries. One
immediately verifies that for any $(x^*,x)\in\A^{(2)}$ and $y\in\A$ with $x\ne
y$,
\begin{equation}\label{eq6.5}
\sum_{y^*:(y^*,y)\in\A^{(2)}}Q^{(2)}\bigl((x^*,x),(y^*,y)\bigr)=Q(x,y).
\end{equation}
Indeed, one needs to consider two cases $\Lambda(x^*,y)=0$ and $\ne 0$, and in
both cases the statement follows from the definitions. As the row sums of
$Q^{(2)}$ and $Q$ are all zero, we obtain \eqref{eq6.5} for $x=y$ as well.

For any $x\in \A$, let us also introduce a matrix of transition rates $Q_x$ on
the fiber $\A_x=\{x^*\in \A^*\mid \Lambda(x^*,x)\ne 0\}$ via
$$
Q_x(x^*,y^*)=Q^{(2)}((x^*,x),(y^*,x))=Q^*(x^*,y^*)\,\dfrac{\Lambda(y^*,x)}{\Lambda(x^*,x)},\qquad
x^*\ne y^*,
$$
and
$$
Q_x(x^*,x^*)=-\sum_{y^*\in \A_x,\,y^*\ne x^*} Q_x(x^*,y^*).
$$

The following statement is similar to Lemma 2.1 of \cite{BF08+} proved in the
discrete time setting. As we will see, the proof of the continuous time
statement is significantly more difficult.

\begin{proposition}\label{6.3}
Assume that the matrices of transition rates $Q$, $Q^*$, and $Q_x$ for any
$x\in\A$ are regular. Then $Q^{(2)}$ is also regular, and denoting by $P(t)$,
$P^*(t)$, and $P^{(2)}(t)$ the transition matrices corresponding to $Q,Q^*$,
and $Q^{(2)}$, we have
\begin{align}
\sum_{y^*:(y^*,y)\in\A^{(2)}} P^{(2)}\bigl(t;(x^*,x),(y^*,y)\bigr)=&P(t;x,y),
\label{eq6.6}\\
\sum_{x:(x^*,x)\in\A^{(2)}}\Lambda(x^*,x)P^{(2)}
\bigl(t;(x^*,x),(y^*,y)\bigr)=&P^*(t;x^*,y^*)\Lambda(y^*,y),\label{eq6.7}
\end{align}
where in the first relation $(x^*,x)\in\A^{(2)}$, $y\in\A$ are arbitrary, while
in the second relation $x^*\in\A^*,$ $(y^*,y)\in\A^{(2)}$ are arbitrary.
\end{proposition}

\begin{proof} The regularity of $Q^{(2)}$ and collapsibility relation
\eqref{eq6.6}
follow from Proposition \ref{3.31} with \eqref{eq3.31} specializing to
\eqref{eq6.5}.

Proving \eqref{eq6.7} is more difficult, and we will follow the following path.
First, we will show that both sides of \eqref{eq6.7} satisfy the same
differential equation (essentially the Kolmogorov backward equation for
$P^*(t;x^*,y^*)$) with a certain initial condition. Then we will see that the
right-hand side of \eqref{eq6.7} represents the minimal of all nonnegative
solutions of this equation. Since for a fixed $x^*$, both sides of
\eqref{eq6.7} represent probability measures on
 $\A^{(2)}$, the equality will immediately follow.

For the first step, let us show that the left-hand side $f_t(x^*,y^*,y)$ of
\eqref{eq6.7} satisfies
\begin{equation}\label{eq6.8}
\frac{d}{dt}f_t(x^*,y^*,y) = \sum_{z^*\in \A^*} Q^*(x^*,z^*) f_t(z^*,y^*,y)
\end{equation}
with the initial condition
\begin{equation}\label{eq6.9}
\lim_{t\to+0}f_t(x^*,y^*,y)=\mathbf{1}_{x^*=y^*}\Lambda(y^*,y).
\end{equation}
The initial condition satisfied by $P^{(2)}(t)$ implies \eqref{eq6.9}, so let
us prove \eqref{eq6.8}.

Using the Kolmogorov backward equation for $P^{(2)}(t)$, we obtain
\begin{multline}\label{eq6.10}
\frac{d}{dt} f_t(x^*,y^*,y)=\sum_{x:(x^*,x)\in\A^{(2)}}\Lambda(x^*,x)
\Biggl(-q^{(2)}_{(x^*,x)}P^{(2)}\bigl(t;(x^*,x),(y^*,y)\bigr)\\+
\sum_{(z^*,z)\ne
(x^*,x)}Q^{(2)}((x^*,x),(z^*,z)\bigr)P^{(2)}\bigl(t;(z^*,z),(y^*,y)\bigr)
\Biggr).
\end{multline}

For the first term in the right-hand side, we use
$$
q^{(2)}_{(x^*,x)}=q_x+\sum_{w^*:w^*\ne
x^*}Q^*(x^*,w^*)\,\frac{\Lambda(w^*,x)}{\Lambda(x^*,x)},
$$
which follows directly from the definition of $Q^{(2)}$. Thus, we can rewrite
the first term in the right-hand side of \eqref{eq6.10} as
\begin{multline}\label{eq6.11}
-\sum_{x\in\A} q_x\Lambda(x^*,x)P^{(2)}\bigl(t;(x^*,x),(y^*,y)\bigr)\\-
\sum_{x:\,(x^*,x)\in\A^{(2)}}\sum_{w^*:\,w^*\ne
x^*}\Lambda(w^*,x)Q^*(x^*,w^*)P^{(2)}\bigl(t;(x^*,x),(y^*,y)\bigr).
\end{multline}

For the second term of the right-hand side of \eqref{eq6.10}, according to the
definition of $Q^{(2)}$, let us split the sum over $(z^*,z)$ into three
disjoint parts: (1) $x^*=z^*$, $x\ne z$; (2) $x^*\ne z^*$, $x=z$; (3)
$\Lambda(x^*,z)=0$ (hence, $x^*\ne z^*$, $x\ne z$).

Part (1) gives
$$
(1)=\sum_{x:(x^*,x)\in\A^{(2)}} \Lambda(x^*,x) \sum_{z:\, z\ne x,\,
(x^*,z)\in\A^{(2)}} Q(x,z)P^{(2)}\bigl(t;(x^*,z),(y^*,y)\bigr).
$$
Interchanging the summations over $x$ and $z$, we can employ the commutativity
relation \eqref{eq6.3}. This gives
\begin{multline}\label{eq6.12}
(1)=\sum_{z:\,(x^*,z)\in\A^{(2)}}\sum_{v^*\in\A^*}\Lambda(v^*,z)Q^*(x^*,v^*)
P^{(2)}\bigl(t;(x^*,z),(y^*,y)\bigr)\\+
\sum_{x\in\A}q_x\Lambda(x^*,x)P^{(2)}\bigl(t;(x^*,x),(y^*,y)\bigr).
\end{multline}

Observe that the last term cancels out with the first term in \eqref{eq6.11},
while the sum of the first term of \eqref{eq6.12} and the second term of
\eqref{eq6.11}, with identification $z=x$, $v^*=w^*$ of the summation
variables, yields (only terms with $v^*=x^*$ survive)
\begin{equation}\label{eq6.13}
-q^*_{x^*}\sum_{x\in\A}\Lambda(x^*,x)P^{(2)}\bigl(t;(x^*,x),(y^*,y)\bigr).
\end{equation}

Further, part (2) of the second term of \eqref{eq6.10} reads
\begin{equation}\label{eq6.14}
(2)=\sum_{x:(x^*,x)\in \A^{(2)}}\sum_{z^*:z^*\ne x^*}
\Lambda(z^*,x)Q^*(x^*,z^*) P^{(2)}\bigl(t;(z^*,x),(y^*,y)\bigr).
\end{equation}

Finally, part (3) gives
\begin{multline}\label{eq6.15}
(3)=\sum_{x:(x^*,x)\in
\A^{(2)}}\sum_{(z^*,z)\in\A^{(2)}:\Lambda(x^*,z)=0}\Lambda(x^*,x)Q(x,z)
\,\frac{Q^*(x^*,z^*)\Lambda(z^*,z)}{\Delta(x^*,z)}\\
\qquad\qquad\qquad\qquad\qquad\qquad\qquad\qquad\qquad\qquad\times
P^{(2)}\bigl(t;(z^*,z),(y^*,y)\bigr)
\\=\sum_{z:(x^*,z)\notin\A^{(2)}}\sum_{z^*:z^*\ne x^*}
\Lambda(z^*,z)Q^*(x^*,z^*) P^{(2)}\bigl(t;(z^*,z),(y^*,y)\bigr),
\end{multline}
where we used the definition of $\Delta$, see \eqref{eq6.4}, to perform the
summation over $x\ne z$. One readily sees that adding \eqref{eq6.13},
\eqref{eq6.14}, \eqref{eq6.15} yields the right-hand side of \eqref{eq6.8}.

Assume now that we have a nonnegative solution $f_t(x^*,y^*,y)$ of
\eqref{eq6.8} satisfying the initial condition \eqref{eq6.9}. Multiplying both
sides of \eqref{eq6.8} by $\exp(q^*_{x^*}t)$ we obtain
$$
\bigl(\exp(q^*_{x^*}t)f_t(x^*,y^*,y) \bigr)'=\exp(q^*_{x^*}t)\sum_{z^*\ne x^*}
Q^*(x^*,z^*) f_t(z^*,y^*,y).
$$
Integrating both sides over $t$ and using \eqref{eq6.9} gives

\begin{multline}\label{eq6.16}
f_t(x^*,y^*,y)=\mathbf{1}_{x^*=y^*}\Lambda(y^*,y)\exp(-q^*_{x^*}t)\\+\int_0^t
\exp(-q^*_{x^*}s)\sum_{z^*\ne x^*} Q^*(x^*,z^*) f_{t-s}(z^*,y^*,y)ds.
\end{multline}

Set $F^{(0)}_t(x^*,y^*)=\mathbf{1}_{x^*=y^*}\exp(-q^*_{x^*}t)$, and for
$n=1,2,\dots$ define
$$
F^{(n)}_t(x^*,y^*)=F^{(0)}_t(x^*,y^*)+\int_0^t \exp(-q^*_{x^*}s)\sum_{z^*\ne
x^*} Q^*(x^*,z^*) F^{(n-1)}_{t-s}(z^*,y^*)ds.
$$

Clearly, \eqref{eq6.16} implies $f_t(x^*,y^*,y)\ge
F^{(0)}_t(x^*,y^*)\Lambda(y^*,y)$, and substituting such estimates into
\eqref{eq6.16} recursively we see that
$$
f_t(x^*,y^*,y)\ge F^{(n)}_t(x^*,y^*)\Lambda(y^*,y),\qquad n=0,1,2,\dots
$$

On the other hand, we know that
$$
\lim_{n\to\infty}F^{(n)}_t(x^*,y^*)=P^*(t;x^*,y^*),
$$
see Section \ref{Generalities on Markov}, \cite{Fel40}, \cite{And91}. Hence,
any nonnegative solution of \eqref{eq6.8}, \eqref{eq6.9} is bounded by
$P^*(t;x^*,y^*)\Lambda(y^*,y)$ from below, and the proof of Proposition
\ref{6.3} is complete.
\end{proof}

The following statement is the analog of Proposition 2.2 in \cite{BF08+}.

\begin{corollary}\label{6.4}
Let $\mu^*(x^*)$ be a probability measure on $\A^*$. For $t\ge0$, let
$(x^*(t),x(t))$ be an $\A^{(2)}$-valued random variable with
$$
\operatorname{Prob}\bigl\{(x^*(t),x(t))=(x^*,x)\bigr\}=\sum_{(y^*,y)\in
\A^{(2)}} \mu^*(y^*)\Lambda(y^*,y) P^{(2)}\bigl(t;(y^*,y),(x^*,x)\bigr).
$$
Then for any time moments $0\le t_0\le t_1\le\dots\le t_k\le t_{k+1}\le\dots
\le t_{k+l}$, the joint distribution of
$$
\bigl(x^*(t_0),x^*(t_1),\dots,x^*(t_k),x(t_k),x(t_{k+1}),\dots,x(t_{k+l})\bigr)
$$
coincides with the stochastic evolution of $\mu^*$ under transition matrices
$$
\bigl(P^*(t_0),P^*(t_1-t_0),\dots,P^*(t_k-t_{k-1}),\Lambda,P(t_{k+1}-t_k),\dots,
P(t_{k+l}-t_{k+l-1})\bigr)
$$
\end{corollary}

\begin{proof} In the joint distribution
\begin{multline}
\mu^*(y^*)\Lambda(y^*,y)
P^*\bigl(t_0;(y^*,y),(x^*_0,x_0)\bigr)P^*\bigl({t_1-t_0};(x^*_0,x_0),(x^*_1,x_1)\bigr)
\cdots\\ \cdots
P^*\bigl({t_{k+l}-t_{k+l-1}};(x^*_{k+l-1},x_{k+l-1}),(x^*_{k+l},x_{k+l})\bigr)
\end{multline}
one uses \eqref{eq6.7} to sum over $y,x_0,\dots,x_{k-1}$ and \eqref{eq6.6} to
sum over $x_{k+1}^*,\dots,x_{k+l}^*$.
\end{proof}

\subsection{Construction of multivariate Markov chains}

Let $\A_1,\dots,\A_N$ be countable sets, $Q_1,\dots,Q_N$ be matrices of
transition rates on these sets, and $\Lambda^2_1,\dots,\Lambda^N_{N-1}$ be
stochastic links:
$$
\Lambda_{k-1}^k:\A_k\times\A_{k-1}\to [0,1],\qquad
\sum_{y\in\A_{k-1}}\Lambda_{k-1}^k(x,y)=1,\quad x\in \A_k,\qquad k=2,\dots,N.
$$
It is also convenient to introduce a formal symbol $\Lambda^{1}_0$ with
$\Lambda^1_0(\,\cdot\,,\,\cdot\,)\equiv 1$. It can be viewed as a stochastic
link between $\A_1$ and a singleton $\A_0$.

We assume that for each of the matrices $Q_j$, $\Lambda^j_{j-1}$, each row
contains only finitely many nonzero entries, and that the following
commutativity relations are satisfied:
$$
\sum_{u\in\A_{k-1}}\Lambda^k_{k-1}(x,u)Q_{k-1}(u,y)
=\sum_{v\in\A_k}Q_k(x,v)\Lambda^k_{k-1}(v,y), \qquad k=2,\dots,N,
$$
or $\Lambda^k_{k-1}Q_{k-1}=Q_k\Lambda^k_{k-1}$ in matrix notation. If
$\Lambda^k_{k-1}(x,y)=0$, the terms with $u=y$ and $v=x$ give no contribution
to the sums and thus can be excluded. In that case we define ($x\in\A_{k}$,
$y\in\A_{k-1}$)
\begin{equation}\label{eq6.17}
\Delta^k_{k-1}(x,y):= \sum_{u:u\ne
y}\Lambda^k_{k-1}(x,u)Q_{k-1}(u,y)=\sum_{v:v\ne x}Q_k(x,v)\Lambda^k_{k-1}(v,y),
\end{equation}
and also
$$
\hat Q_k(x,v,y)=\begin{cases}
\dfrac{Q_k(x,v)\Lambda^k_{k-1}(v,y)}{\Delta^k_{k-1}(x,y)}\,,& \text{if  }\
\Delta^k_{k-1}(x,y)\ne 0,\\ 0,& \text{if  }\  \Delta^k_{k-1}(x,y)=0.
\end{cases}
$$
In case $\Delta^{k}_{k-1}(x,y)\ne 0$, $Q_k(x,v,y)$ is a probability
distribution in $v\in\A_k$ that depends on $x$ and $y$.

In the application of this formalism that we consider in the next section,
there is always exactly one $v$ that contributes nontrivially to the right-hand
side of \eqref{eq6.17}, which means that the distribution $\hat Q_k(x,v,y)$ is
supported by one point.

We define the state space $\A^{(N)}$ for the multivariate Markov chain by
\eqref{eq6.1} and then define the off-diagonal entries of the matrix $Q^{(N)}$
of transition rates on $\A^{(N)}$ as (we use the notation
$X_N=(x_1,\dots,x_N)$, $Y_N=(y_1,\dots,y_N)$)
$$
Q^{(N)}(X_N,Y_N) =\begin{cases}
Q_k(x_k,y_k)\,\dfrac{\Lambda^k_{k-1}(y_k,x_{k-1})}{\Lambda^k_{k-1}(x_k,x_{k-1})}\,,\\
Q_k(x_k,y_k)\,\dfrac{\Lambda^k_{k-1}(y_k,x_{k-1})}{\Lambda^k_{k-1}(x_k,x_{k-1})}\,
\hat Q_{k+1}(x_{k+1},y_{k+1},y_k)\cdots \hat Q_l(x_l,y_l,y_{l-1}),
\end{cases}
$$
where for the first line we must have $x_j=y_j$ for all $j\ne k$ and some
$k=1,\dots,N$, while for the second line we must have $x_j=y_j$ iff $j<k$ or
$j>l$ for some $1\le k<l\le N$, and $\Lambda(x_j,y_{j-1})=0$ for $k+1\le j\le
l$. If neither of the two sets of conditions is satisfied, we set
$Q^{(N)}(X_N,Y_N)$ to 0.

The diagonal entries $Q^{(N)}\bigl(X_N,X_N\bigr)$ are defined by
$$
Q^{(N)}\bigl(X_N,X_N\bigr)=-\sum_{Y_N\ne X_N} Q^{(N)}\bigl(X_N,Y_N\bigr).
$$

The definition of $Q^{(N)}$ can be interpreted as follows: Each of the
coordinates $x_k$, $k=1,\dots,N$, is attempting to jump to $y_k\in\A_k$ with
certain rates. Only $y_k$'s with $Q(x_k,y_k)\ne 0$ are eligible. Three
situations are possible:
\medskip

(1) The change of $x_k$ to $y_k$ does not move $X_N$ out of the state space,
that is $\Lambda^{k+1}_k(x_{k+1},y_k)\Lambda^k_{k-1}(y_k,x_{k-1})\ne0$. Such
jumps have rates
$Q_k(x_k,y_k)\,\frac{\Lambda^k_{k-1}(y_k,x_{k-1})}{\Lambda^k_{k-1}(x_k,x_{k-1})}$.
Note that for $k=1$ the last factor is always 1.
\medskip

(2) The change of $x_k$ to $y_k$ is in conflict with $x_{k-1}$, that is
$\Lambda^{k}_{k-1}(y_k,x_{k-1})=0$. Such jumps are blocked.
\medskip

(3) The change of $x_k$ to $y_k$ is in conflict with $x_{k+1}$, that is
$\Lambda^{k+1}_k(x_{k+1},y_k)=0$. Then $x_{k+1}$ has to be changed too, say to
$y_{k+1}$. We must have $\Lambda^{k+1}_k(y_{k+1},y_k)\ne 0$; relation
\eqref{eq6.17} guarantees the existence of at least one such $y_{k+1}$. If the
double jump $(x_k,x_{k+1})\to(y_k,y_{k+1})$ keeps $X_N$ in the state space, it
is allowed, and its rate is
$Q_k(x_k,y_k)\,\frac{\Lambda^k_{k-1}(y_k,x_{k-1})}{\Lambda^k_{k-1}(x_k,x_{k-1})}\,
\hat Q_{k+1}(x_{k+1},y_{k+1},y_k)$. Otherwise, $x_{k+2}$ has to be changed as
well, and so on.

\medskip

 To say it differently, unless $\Lambda^{k}_{k-1}(y_k,x_{k-1})=0$, the move
$x_k\to y_k$ always happens with rate
$Q_k(x_k,y_k)\,\frac{\Lambda^k_{k-1}(y_k,x_{k-1})}{\Lambda^k_{k-1}(x_k,x_{k-1})}$,
and it may cause a sequence of displacements of $x_{k+1},x_{k+2},\dots$, where
each next $x_j$ uses the distribution $\hat Q_{j}(x_{j},\,\cdot\,,y_{j-1})$ to
choose its new position. Displacements end once $X_N$ is back in $\A^{(N)}$.
This description implies the following formula for the diagonal entries of
$Q^{(N)}$:
\begin{equation}\label{eq6.18}
Q^{(N)}(X_N,X_N)=-\sum_{k=1}^N \sum_{y_k\in\A_k:y_k\ne x_k}
Q_k(x_k,y_k)\,\frac{\Lambda^k_{k-1}(y_k,x_{k-1})}{\Lambda^k_{k-1}(x_k,x_{k-1})}\,.
\end{equation}

The definition of $Q^{(N)}$ is explained by the following statement.

\begin{proposition}\label{6.5}
Consider the matrix $\Lambda$ with rows marked by elements of $\A_N$, columns
marked by $\A^{(N-1)}$, and entries given by
\begin{equation}\label{eq6.19}
\Lambda(x_N,(x_1,\dots,x_{N-1}))=\Lambda^N_{N-1}(x_N,x_{N-1})\cdots
\Lambda^2_1(x_2,x_1).
\end{equation}
Then the commutativity relation $\Lambda Q^{(N-1)}=Q_N\Lambda$ holds.
\end{proposition}

\begin{proof}
We have
\begin{equation}\label{eq6.20}
\Lambda Q^{(N-1)}(x_N,Y_{N-1})=\sum_{X_{N-1}\in\A^{(N-1)}}
\Lambda(x_N,X_{N-1})Q^{(N-1)}(X_{N-1},Y_{N-1}).
\end{equation}
By \eqref{eq6.18}, the contribution of $X_{N-1}=Y_{N-1}$ to the right-hand side
has the form
\begin{equation}\label{eq6.21}
-\Lambda(x_n,Y_{N-1})\sum_{k=1}^{N-1} \sum_{z_k\in\A_k:z_k\ne y_k}
Q_k(y_k,z_k)\,\frac{\Lambda^k_{k-1}(z_k,y_{k-1})}{\Lambda^k_{k-1}(z_k,y_{k-1})}\,.
\end{equation}
For $X_{N-1}\ne Y_{N-1}$, the contribution of matrix elements of
$Q^{(N-1)}(X_{N-1},Y_{N-1})$ that correspond to jumps
$(x_k,x_{k+1},\dots,x_l)\to (y_k,y_{k+1},\dots,y_l)$, $1\le k\le l\le N$, with
all other $x_j=y_j$, has the form
\begin{multline}\label{eq6.22}
\sum
\Lambda^N_{N-1}(x_N,y_{N-1})\Lambda^{N-1}_{N-2}(y_{N-1},y_{N-2})\cdots
\Lambda^{l+2}_{l+1}(y_{l+2},y_{l+1})
\\
\times\Lambda^{l+1}_{l}(y_{l+1},x_{l})\Lambda^{l}_{l-1}(x_{l},x_{l-1})\cdots\Lambda^{k+1}_k(x_{k+1},x_k)
\\
\times
\Lambda^{k}_{k-1}(x_{k},y_{k-1})\Lambda^{k-1}_{k-2}(y_{k-1},y_{k-2})\cdots\Lambda^2_1(y_2,y_1)
\\ \times
Q_k(x_k,y_k)\,\frac{\Lambda^k_{k-1}(y_k,y_{k-1})}{\Lambda^k_{k-1}(x_k,y_{k-1})}\,
\hat Q_{k+1}(x_{k+1},y_{k+1},y_k)\cdots \hat Q_l(x_l,y_l,y_{l-1}),
\end{multline}
where the summation is over $x_k,\dots,x_l$ satisfying $x_i\ne y_i$ for all
$k\le i\le l$ and
\begin{equation}\label{eq6.23}
\Lambda^k_{k-1}(x_k,y_{k-1})\ne 0,\qquad \Lambda^i_{i-1}(x_i,y_{i-1})=0,\quad
k<i\le l.
\end{equation}
Denote this expression by $A(k,l)$.

Observe that in \eqref{eq6.22}, the factors $\Lambda^k_{k-1}(x_k,y_{k-1})$
cancel out. Let us denote by $B(k,l)$ the sum of same expressions
\eqref{eq6.22} with canceled $\Lambda^k_{k-1}(x_k,y_{k-1})$, and with
conditions \eqref{eq6.23} replaced by
$$
\Lambda^k_{k-1}(x_k,y_{k-1})= 0,\qquad \Lambda^i_{i-1}(x_i,y_{i-1})=0,\quad
k<i\le l.
$$
Thus, the sum $A(k,l)+B(k,l)$ has no restrictions on $x_k$ other that $x_k\ne
y_k$.

Using the definitions of $\Delta^{k+1}_k$ and $\hat Q_{k+1}$ we see that
\begin{multline}
\sum_{x_k:x_k\ne y_k} \Lambda^{k+1}_k(x_{k+1},x_k)Q_k(x_k,y_k)\hat
Q_{k+1}(x_{k+1},y_{k+1},y_k)\\=Q_{k+1}(x_{k+1},y_{k+1})\Lambda^{k+1}_k(y_{k+1},y_k).
\end{multline}
Hence, $A(k,l)+B(k,l)=B(k+1,l)$. Noting that $B(1,l)=0$, we obtain, for any
$l=1,\dots,N-1$,
\begin{multline}
A(1,l)+A(2,l)+\dots+A(l,l)=A(l,l)+B(l,l)\\
=\frac{\Lambda(x_N,Y_{N-1})}{\Lambda^{l+1}_l(y_{l+1},y_l)}\sum_{x_l:x_l\ne y_l}
\Lambda^{l+1}_{l}(y_{l+1},x_l)Q_l(x_l,y_l)\\=\frac{\Lambda(x_N,Y_{N-1})}
{\Lambda^{l+1}_l(y_{l+1},y_l)} \sum_{x_l\in\A_l}
\Lambda^{l+1}_{l}(y_{l+1},x_l)Q_l(x_l,y_l)-\Lambda(x_N,Y_{N-1})Q_l(y_l,y_l)\\
={\Lambda(x_N,Y_{N-1})} \left(\sum_{z_{l+1}\in\A_{l+1}}
Q_{l+1}(y_{l+1},z_{l+1})\frac{\Lambda^{l+1}_l(z_{l+1},y_l)}{\Lambda^{l+1}_l(y_{l+1},y_l)}-Q_l(y_l,y_l)\right)\\
={\Lambda(x_N,Y_{N-1})} \Biggl(\sum_{z_{l+1}\ne y_{l+1}}
Q_{l+1}(y_{l+1},z_{l+1})\frac{\Lambda^{l+1}_l(z_{l+1},y_l)}{\Lambda^{l+1}_l(y_{l+1},y_l)}\\
+Q_{l+1}(y_{l+1},y_{l+1})-Q_l(y_l,y_l)\Biggr),
\end{multline}
where we used the commutativity relation
$\Lambda^{l+1}_lQ_l=Q_{l+1}\Lambda^{l+1}_l$ along the way. Hence, using
\eqref{eq6.20} we obtain
\begin{multline}
\Lambda Q^{(N-1)}(x_N,Y_{N-1})=\sum_{1\le k\le l\le
N-1}A(k,l)+\Lambda(x_N,Y_{N-1})Q^{(N-1)}(Y_{N-1},Y_{N-1}) \\= \sum_{z_N\ne
y_N}Q_N(x_N,z_N)\Lambda(z_N,Y_{N-1})+Q_N(x_N,x_N)\Lambda(x_N,Y_{N-1})=
Q_N\Lambda(x_N,Y_{N-1}).\\
\end{multline}
\end{proof}

For any $N\ge 2$ and $x_{N-1}\in \A_{N-1}$ let us define a matrix $Q_{x_{N-1}}$
of transition rates on the fiber
$$
\A_{x_{N-1}}=\{x_N\in \A_N\mid \Lambda^{N}_{N-1}(x_N,x_{N-1})\ne 0\}
$$
via
\begin{equation}\label{eq6.23a}
\gathered Q_{x_{N-1}}(x_N,y_N)=
Q_N(x_N,y_N)\,\frac{\Lambda^N_{N-1}(y_N,x_{N-1})}{\Lambda^N_{N-1}(x_N,x_{N-1})},
\qquad y_N\ne x_N,\\
Q_{x_{N-1}}(x_N,x_N)=-\sum_{y_N\in \A_{X_{N-1}},\,y_N\ne
x_N}Q_{x_{N-1}}(x_N,y_N).
\endgathered
\end{equation}

The next statement is analogous to Proposition 2.5 in \cite{BF08+}.

\begin{proposition}\label{6.6}
Assume that the matrices of transition rates $Q_1,\dots,Q_N$ and $Q_{x_1},\dots
Q_{x_{N-1}}$ for any $x_j\in \A_j$, $j=1,\dots,N-1$ are regular. Then
$Q^{(2)},\dots, Q^{(N)}$ are also regular. Denote by $\{{P_{j}}(t)\}_{1\le j\le
N}$ and $P^{(N)}(t)$ the transition matrices for $\{{Q_{j}}(t)\}_{1\le j\le N}$
and $Q^{(N)}(t)$.

Let $\mu_N$ be a probability measure on $\A_N$, and for $t\ge0$, let
$(x_1(t),\dots,x_N(t))$ be a $\A^{(N)}$-valued random variable with
\begin{multline}
\Prob\bigl\{(x_1(t),\dots,x_N(t))=(x_1,\dots,x_N)\bigr\}\\=\sum_{Y_N\in
\A^{(N)}} \mu_N(y_N)\Lambda(y_N,Y_{N-1}) P^{(N)}\bigl(t;Y_N,X_N\bigr).
\end{multline}
Then for any sequence of time moments
\begin{multline}0\le t^N_{0}\le t^N_{1}\le\dots\le t^N_{k_N}=t^{N-1}_{0}\le
t^{N-1}_{1}\le\dots\le t^{N-1}_{k_{N-1}}=t^{N-2}_{0}\le\dots\\ \ldots\le
t^2_{k_2}=t^1_{0}\le t^1_{1}\le \dots\le t^1_{k_1}
\end{multline}
the joint distribution of $\{x_m(t^m_{k})\}$ ordered as the time moments
coincides with the stochastic evolution of $\mu_N$ under transition matrices
\begin{multline}
{P_{N}}({t^N_{0}}),{P_{N}}({t^N_{1}-t^N_{0}}),\dots,{P_{N}}({t^N_{k_N}-t^N_{k_{N}-1}}),
\Lambda^N_{N-1},
\\{P_{N-1}}({t^{N-1}_{1}-t^{N-1}_{0}}),\dots,{P_{N-1}}({t^{N-1}_{k_{N-1}}-t^{N-1}_{k_{N-1}-1}}),
\Lambda^{N-1}_{N-2},\dots\\\dots,
{P_{1}}({t^1_{1}-t^1_{0}}),\dots,{P_{1}}({t^1_{k_1}-t^1_{k_{1}-1}}).
\end{multline}
\end{proposition}

\begin{proof} It is a straightforward computation to see that the construction
of the bivariate Markov chain from the previous section applied to
$Q=Q^{(N-1)}$, $Q^*=Q_N$, and $\Lambda$ given by \eqref{eq6.19} (the needed
commutativity is proved in Proposition \ref{6.5}), yields exactly $Q^{(N)}$. We
apply Corollary \ref{6.4}, and induction on $N$ concludes the proof. .
\end{proof}

\begin{corollary}\label{6.7}
In the assumptions of Proposition \ref{6.6}, $(P^{(N)}(t))_{t\ge 0}$ is central
in the sense of Definition \ref{6.1}, and the induced semigroup on $\A_N$ is
exactly $(P_N(t))_{t\ge 0}$. Furthermore, compatibility relations of
Proposition \ref{6.2} also hold.
\end{corollary}

\begin{proof} The first two statements follow from Proposition \ref{6.6} with
$$
k_N=1,\qquad k_{N-1}=k_{N-2}=\dots=k_1=0.
$$
The third statement is \eqref{eq6.7} with $Q=Q^{(N-1)}$, $Q^*=Q_N$,
$Q^{(2)}=Q^{(N)}$, and $\Lambda$ given by \eqref{eq6.19}.
\end{proof}

\section{Stochastic dynamics on paths. Gelfand-Tsetlin graph}\label{Stochastic
dynamics2}

\subsection{Central measures on paths and the boundary}\label{Central measures}

 Let us
return to our concrete setup, cf. Section \ref{Specialization}. We have
$E_N=\gt_N$, the space of signatures of length $N$, and $E^{(N)}$ of
\eqref{eq6.1} is the set of Gelfand-Tsetlin schemes of length $N$; we denote it
by $\gt^{(N)}$.

Due to \eqref{eq2.0}, the notion of centrality for $\mu^{(N)}\in\mathcal
M_p(\gt^{(N)})$ means the following, cf. \eqref{eq6.2}: For any
$\underline\la=(\la^{(1)}\prec\la^{(2)}\prec\dots\prec\la^{(N)})\in\gt^{(N)}$,
$\mu^{(N)}(\underline\la)$ depends only on $\la^{(N)}$. For branching graphs,
the notion of central measures was introduced in \cite{VK81}, see also
\cite{Ker03}.

In Subsection \ref{Overview} we explained that central measures on the space
$E^{(\infty)}=:\gt^{(\infty)}$ of infinite Gelfand-Tsetlin schemes are in
bijection, thanks to Theorem \ref{2.1}, with $\mathcal M_p(\Omega)$. Let us
make this bijection more explicit.

Given a signature $\la\in\gt_N$, denote by $\la^+$ and $\la^-$ its positive and
negative parts. These are two partitions (or Young diagrams) with
$\ell(\la^+)+\ell(\la^-)\le N$, where $\ell(\,\cdot\,)$ is the number of
nonzero rows of a Young diagram. In other words,
$$
\la=(\la_1^+,\dots,\la_k^+,0,\dots,0,-\la_l^-,\dots,-\la_1^-),\qquad
k=\ell(\la^+),\quad l=\ell(\la^-).
$$

Given a Young diagram $\nu$, denote by $d(\nu)$ the number of diagonal boxes in
$\nu$. Introduce {\it Frobenius coordinates\/} of $\nu$ via
$$
p_i(\nu)=\nu_i-i,\quad q_i(\nu)=\nu\,'_i-i,\qquad i=1\dots,d(\nu),
$$
where $\nu\,'$ stands for the transposed diagram. We also set
$$
p_i(\nu)=q_i(\nu)=0, \qquad i> d(\nu).
$$

An element $\underline\la=(\la^{(1)}\prec\la^{(2)}\prec\dots)
\in\gt^{(\infty)}$, which can be viewed as an infinite increasing path in the
Gelfand-Tsetlin graph $\gt$, is called {\it regular\/} if there exist limits
$$
\alpha_i^\pm=\lim_{N\to\infty}\frac{p_i(\la^{(N)})}{N},\quad
\beta_i^\pm=\lim_{N\to\infty}\frac{q_i(\la^{(N)})}{N},\quad i=1,2,\dots,\quad
\delta^\pm=\lim_{N\to\infty} \frac{|\la^\pm|}{N}.
$$
The corresponding point $\omega=(\alpha^\pm,\beta^\pm,\delta^\pm)\in\Omega$ is
called the {\it end\/} of this path.

\begin{theorem}[\cite{Ols03}]\label{7.1}
Any central measure on $\gt^{(\infty)}$ is supported by the Borel set of
regular paths. Pushforward of such measures under the map that takes a regular
path to its end, establishes an isomorphism between the space of central
measures on $\gt^{(\infty)}$ and $\mathcal M_p(\Omega)$.
\end{theorem}

We refer the reader to Section 10 of \cite{Ols03} for details.

\subsection{Matrices of transition rates on $\gt^{(N)}$}

With $E_N=\gt_N$, $E^{(N)}=\gt^{(N)}$ and $Q_N=\D^{(N)}$, let us write out the
specialization of the matrix $Q^{(N)}$ from Subsection \ref{Construction}. We
will use the notation $\mathbf D^{(N)}$ for the resulting matrix of transition
rates on $\gt^{(N)}$. As for the parameters, we will use \eqref{eq4.1} and
\eqref{eq4.2} as before.

To any $\underline\la\in\gt^{(N)}$ we associate an array $\{l_i^j\mid 1\le i\le
j,\; 1\le j\le N \}$ using $l_i^j=\la^{(j)}_i+j-i$. In these coordinates, the
interlacing conditions $\la^{(j)}\prec\la^{(j+1)}$ take the form
$$
l^{j+1}_i> l_i^j\ge l_{i+1}^{j+1}
$$
for all meaningful values of $i$ and $j$.\footnote{One could make the
interlacing condition more symmetric (both inequalities being strict) by
considering the coordinates $\wt l_i^j=\la_i^j+(j+1)/2-i$ instead. This would imply
however that $\wt l^j_i\in \Z+1/2$ for odd $j$ while $\wt l_i^j\in \Z$ for even
$j$.}

Similarly, assign $\gt^{(N)}\ni\underline\nu
\longleftrightarrow\{n_i^j=\nu^{(j)}_i+j-i\}_{1\le i\le j, 1\le j\le N}$.
Gathering all the definitions together, we obtain that the off-diagonal entries
of $\mathbf D^{(N)}$ have the form
$$
\mathbf D^{(N)}(\underline\la,\underline\nu)=\begin{cases} (l_i^k-z-k+1)(l_i^k-z'-k+1),\\
(l_i^k+w)(l_i^k+w'),
\end{cases}
$$
where for the first line we must have $i$, $k$ and $l$, $1\le i\le k\le l\le
N$, such that
$$
l_{i+k-j}^j=l_{i}^k+k-j,\quad n_{i+k-j}^j=l_{i+k-j}^j+1 \quad\text{for
all}\quad k\le j\le l,
$$
and all other coordinates of $\underline \la$ and $\underline\nu$ are equal,
while for the second line we must have $i,k,l$ with $1\le i\le k\le l\le N$
such that
$$
l_i^j=l_i^k,\quad n_i^j=l_i^j-1\quad\text{for all}\quad k\le j\le l,
$$
and all other coordinates of $\underline \la$ and $\underline\nu$ are equal.

The Markov chain generated by $\mathbf D^{(N)}$ can be described as follows:
\smallskip

\noindent (1)\quad Each coordinate $l_i^k$ tries to jump to the right by 1 with
rate $(l_i^k-z-k+1)(l_i^k-z'-k+1)$ and to the left by 1 with rate
$(l_i^k+w)(l_i^k+w')$, independently of other coordinates.
\smallskip

\noindent (2)\quad If the $l^k_i$-clock of the right jump rings but
$l_i^k=l^{k-1}_{i-1}$, the jump is blocked. If its left clock rings but
$l_i^k=l^{k-1}_{i}+1$, the jump is also blocked. (If any of the two jumps were
allowed then the resulting set of coordinates would not have corresponded to an
element of $\gt^{(N)}$ as the interlacing conditions would have been violated.)
\smallskip

\noindent (3)\quad If the right $l^k_i$-clock rings and there is no blocking,
we find the greatest number $l\ge k$ such that $l_{i}^j=l_i^k+k-j$ for
$j=k,k+1,\dots,l$, and move all the coordinates $\{l_{i}^j\}_{j=k}^l$ to the
right by one. Given the change $l^k_i\mapsto l^k_i+1$, this is the minimal
modification of the set of coordinates that preserves interlacing.
\smallskip

\noindent (4)\quad If the left $l^k_i$-clock rings and there is no blocking, we
find the greatest number $l\ge k$ such that $l_{i+j-k}^j=l_i^k$ for
$j=k,k+1,\dots,l$, and move all the coordinates $\{l_{i+j-k}^j\}_{j=k}^l$ to
the left by one. Again, given the change $l^k_i\mapsto l^k_i-1$, this is the
minimal modification of the set of coordinates that preserves interlacing.

\smallskip

Certain Markov chain on interlacing arrays with a similar block-push mechanism
have been studied in \cite{BF08+}, see also \cite{BK10}. In those examples the
jump rates are constant though.

\subsection{Regularity}

 In order to claim the benefits of
Proposition \ref{6.6} and Corollary \ref{6.7}, we need to verify the regularity
of the fiber matrices of transition rates \eqref{eq6.23a}. In our concrete
realization, they take the following form.

For any $N\ge 2$ and any $\kappa\in E_{N-1}=\gt_{N-1}$, the fiber
$E_{\kappa}=:\gt_\kappa\subset \gt_N$ takes the form
$$
\gt_\kappa=\{\la\in\gt_N\mid \kappa\prec\la\}.
$$
Using the coordinates $\{l_i=N+\la_i-i\}_{i=1}^{N}$ for $\la\in\gt_N$ and
$\{n_i=N+\nu_i-i\}_{i=1}^N$ for $\nu\in\gt_N$, the off-diagonal part of the
matrix of transition rates $\D_\kappa:=Q_\kappa$ on the fiber $\gt_\kappa$ has
the form
$$
\D_\kappa(\la,\nu)=\begin{cases} (l_i-z-N+1)(l_i-z'-N+1),\\
(l_i+w)(l_i+w'),
\end{cases}
$$
where for the first line we must have $i$, $1\le i\le N$, such that
$$
n_{i}=l_{i}+1, \qquad n_j=l_j\quad \text{for}\quad j\ne i,
$$
and for the second line we must have
$$
n_{i}=l_{i}-1, \qquad n_j=l_j\quad \text{for}\quad j\ne i.
$$

\begin{proposition}\label{7.2}
For any $N\ge 2$ and any $\kappa\in\gt_{N-1}$, the
matrix of transition rates $\D_\kappa$ on $\gt_\kappa$ is regular.
\end{proposition}

\begin{proof} The interlacing condition in the definition of $\gt_\kappa$
implies that $\D_\kappa$ is the matrix of transition rates for $N$ independent
birth and death processes conditioned to stay within $N$ non-overlapping
intervals inside $\Z$; one interval per process. The results of Section 3.2
show that any such birth and death process is regular as such a process either
lives on a finite set or it is a one-sided birth and death process of the type
considered in the proof of Theorem \ref{3.4}.
\end{proof}

\begin{corollary}\label{7.3}
For any $N\ge 1$, the matrix $\mathbf D^{(N)}$ of transition rates on
$\gt^{(N)}$ is regular, and the corresponding semigroup $(P^{(N)}(t))_{t\ge 0}$
is central. The induced Markov semigroup on $\gt_N$ coincides with that of\/
Section \ref{Semigroups}.
\end{corollary}

\begin{proof} Follows from Proposition \ref{6.6} and Corollary \ref{6.7}.
\end{proof}

\subsection{Exclusion process}

 Observe that the projection of the Markov chain
generated by $\mathbf D^{(N)}$ to the coordinate $l_1^1$ is a bilateral birth
and death process. Furthermore, the jumps of $l_1^2$ are only influenced by
$l_1^1$, the jumps of $l_1^3$ are only influenced by $l_1^1$ and $l_1^2$, and
so on. On the other side, the jumps of $l_k^k$ are only influenced by $\{l_1^1,
l_2^2,\dots l_{k-1}^{k-1}\}$ for any $k\ge 2$.

Hence, the projection of the Markov chain defined by $\mathbf D^{(N)}$ to the
coordinates $(l_N^N\le l_{N-1}^{N-1}\le\dots\le l_1^1<l_1^2<\dots<l_1^N)$ is
also a  Markov chain\footnote{Once again, all the inequalities would be strict
if we considered coordinates $\wt l_i^j=\la_i^j+(j+1)/2-i$.}. The fibers of
this projection are finite, hence, according to Proposition \ref{3.31}, our
Markov chain on $\gt^{(N)}$ collapses to the smaller one, whose matrix of
transition rates is also regular.

Let us project even further to $(l_1^1<l_1^2<\dots<l_1^N)$. Killing extra
coordinates one by one and using the results of Section \ref{Case N=1} to verify the
regularity for the fiber chains, we see that the collapsibility of Proposition
\ref{3.31} holds. Let us give an independent description of the resulting
Markov chain on $\{l_1^j\}_{j\ge 1}$.

Set
\begin{gather*}
\Y_N=\{y_1<y_2<\dots<y_N\mid y_j\in\Z, \, 1\le j\le N\},\\
\Y_\infty=\{y_1<y_2<\dots\mid y_j\in\Z, \, j\ge 1\}.
\end{gather*}
Define the matrix $\mathbf D^{(N)}_{top}$ of transition rates on $\Y_N$ by
$$
\mathbf D^{(N)}(Y',Y'')=\begin{cases} (y_k-z-k+1)(y_k-z'-k+1),\\
(y_k+w)(y_k+w'),
\end{cases}
$$
where for the first line we must have $k$ and $l$, $1\le k\le l\le N$, such
that
$$
y_j'=y_k'+k-j,\quad y_j''=y_j'+1 \quad\text{for all}\quad k\le j\le l,
$$
and all other coordinates of $Y'$ and $Y''$ are equal, while for the second
line we must have
$$
y_k''=y_k'-1,\qquad y_m''=y_m', \text{ for }m\ne k.
$$

In other words, each coordinate $y_k$ tries to jump to the right by 1 with rate
$(y_k-z-j+1)(y_k-z'-j+1)$, and it tries to jump to the left by 1 with rate
$(y_k+w)(y_k+w')$, independently of other coordinates. If the left $y_k$ clock
rings but $y_k=y_{k-1}+1$ then the jump is blocked. If the right $y_k$-clock
rings we find the greatest number $l\ge k$ such that $y_j=y_k+k-j$ for
$j=k,k+1,\dots,l$, and move all the coordinates $\{y_k,\dots,y_l\}$ to the
right by one. One could think of $y_k$ ``pushing'' $y_{k+1},\dots,y_l$.
Alternatively, if one forgets about the labeling one could think of $y_k$
jumping to the first available site on its right.

Clearly, these Markov chains are compatible with projections $\Y_{N+1}\to\Y_N$
that remove the last coordinate. Thus, we obtain a Markov semigroup on
$\varprojlim \Y_N=\Y_\infty$.

This semigroup is a sort of an exclusion process --- it is a one-dimensional
interacting particle system with each site occupied by no more than one
particle (exclusion constraint). A similar system, but with constant jump
rates, was considered in \cite{BF08} and called PushASEP. A system with
one-sided jumps and blocking mechanism as above is usually referred to as {\it
Totally Asymmetric Simple Exclusion Process\/} (TASEP), while a system with
one-sided jumps and pushing mechanism as above is sometimes called {\it long
range\/} TASEP. See \cite{Spi70}, \cite{Lig99} for more information on
exclusion processes.

\begin{proposition}\label{7.4}
The exclusion process defined above has a unique
invariant probability measure. With probability 1 with respect to this measure
there exists a limit $r=\lim_{N\to\infty} y_N/N$, which is a random
variable with values $\ge 1$. Under certain additional restrictions on
parameters $(z,z',w,w')$, see below, the function
$$
\sigma(s)=s(s-1)\frac d{ds}\Prob\{r\le s\}-a_1^2s+\tfrac12(a_3a_4+a_1^2)
$$
is the unique solution of the\/ {\rm(2}nd order nonlinear\/{\rm)} differential
equation
\begin{multline}
-\sigma'\bigl(s(s-1)\sigma'' \bigr)=\bigl(2((s-\tfrac
12)\sigma'-\sigma)\sigma' -a_1a_2a_3a_4\bigr)^2\\
-(\sigma'+a_1^2)(\sigma'+a_2^2)(\sigma'+a_3)^2(\sigma'+a_4^2)
\end{multline}
with boundary condition
$$
\sigma(s)=-a_1^2 s+\tfrac 12(a_3a_4+a_1^2)+\frac{\sin\pi z\sin\pi
z'}{\pi^2}\,s^{-2a_1}+o(s^{-2a_1}), \qquad s\to+\infty,
$$
where the constants $a_1,a_2,a_3,a_4$ are given by
$$
a_1=a_2=\frac{z+z'+w+w'}2,\quad a_3=\frac{z-z'+w-w'}2,\quad
a_4=\frac{z-z'-w+w'}2\,.
$$
\end{proposition}

{\it Remarks\/} 1. The quantity $\lim_{N\to\infty}y_N/N$ can be viewed as the
asymptotic density of the system of particles $(y_j)$ at infinity. Proposition
\ref{7.4} claims that for the invariant measure, this quantity is well-defined
and random.

2. The restrictions on parameters come from Theorem 7.1 of \cite{BD02}. They
can be relaxed, see Remark 7.2 in \cite{BD02} and the end of \S3 in
\cite{Lis09+}.

3. The differential equation above is the so-called $\sigma$-form of the
Painlev\'e VI equation first appeared in \cite{JM81}.
\medskip

\begin{proof}[Proof of Proposition \ref{7.4}]
The invariant measure is simply the projection to
$$
y_1=\la^{(1)},\quad y_2=\lambda_1^{(2)}+1,\quad y_3=\lambda^{(3)}_1+2,\quad
\dots
$$
of the central measure on $\gt^{(\infty)}$ corresponding to the spectral
$zw$-measure. The uniqueness follows from the uniqueness of invariant measure
on countable sets $\Y_N$, cf. Theorem 1.6 of \cite{And91} (a similar argument
was used in the proof of Theorem \ref{5.2}). The existence of
$\lim_{N\to\infty}y_N/N$ follows from Theorem \ref{7.1}. Finally, the
characterization of the distribution of this limit in terms of the Painlev\'e
VI equation was proved in Theorem 7.1 of \cite{BD02}, see \cite{Lis09+} for
another proof.
\end{proof}

\section{Appendix}\label{Appendix}

\subsection{Truncated Gelfand-Tsetlin graph}\label{Jacobi}

Fix two  numbers $k,l=0,1,2,\dots$ not equal to 0 simultaneously. Denote by
$\gt_N(k,l)$ the subset of $\gt_N$ formed by the signatures $\la$ subjected to
the restrictions
$$
k\ge\la_1\ge\dots\ge\la_N\ge-l.
$$
Obviously, this subset is finite and nonempty, and if $\la\in\gt_N(k,l)$ and
$\nu\prec\la$ then $\nu\in\gt_{N-1}(k,l)$. Thus, the union of the sets
$\gt_N(k,l)$ for $N=1,2,\dots$ forms a subgraph of the Gelfand-Tsetlin graph.
Let us denote this truncated graph as $\gt(k,l)$. The definition of the links
$\La^{N+1}_N(\la,\nu)$ in the truncated graph remains the same, the only
difference is that we assume $\la$ and $\nu$ to be vertices of $\gt(k,l)$.

The boundary of $\gt(k,l)$ is the subset $\Omega(k,l)\subset\Omega$ determined
by the restrictions
$$
\al^\pm\equiv0, \quad \ga^\pm=0, \quad \be^+_i=0 \quad (i>k), \quad \be^-_j=0
\quad (j>l),
$$
so that only $k+l$ parameters are nontrivial:
$$
\be^+_1\ge\dots\ge\be^+_k\ge0, \quad \be^-_1\ge\dots\ge\be^-_l\ge0,
$$
where, as before, $\be^+_1+\be^-_1\le1$. The definition of the links
$\La^\infty_N(\om,\la)$ remains the same, only $\om$ is assumed to belong to
$\Om(k,l)$. Viewing $\Om(k,l)$ as a subset of $\R^{k+l}$ one sees that it is a
closed simplex of full dimension.

The boundary is Feller. Since $\Om(k,l)$ is compact, this simply means that the
links $\La^\infty(\om,\la)$ are continuous in $\om\in\Om(k,l)$.

Fix parameters
\begin{equation}\label{eqA.1}
z=k, \quad z'=k+a, \quad w=l, \quad w'=l+b, \quad \text{ where $a,b>-1$}.
\end{equation}
Setting
$$
u=z, \quad u'=z', \quad v=w, \quad v'=w'
$$
in \eqref{3.4} we obtain a truncated birth and death process on
$\{-l,\dots,k\}\subset\Z$; the rates of jumps $k\to k+1$ and $-l\to-l-1$ being
equal to $0$. Note that the nonnegativity of the jump rates on
$\{-l,\dots,k\}\subset\Z$ is ensured by the inequalities $a,b>-1$.

More generally, for any $N\ge 1$ the same expressions as before correctly
determine a matrix of transition rates on $\gt_N(k,l)$. Due to finiteness of
the state space, the existence of the corresponding Markov semigroup $(P_N(t))_{t\ge 0}$
becomes obvious.

The semigroups $(P_N(t))_{t\ge 0}$ with varying $N=1,2,\dots$ are consistent with the
links and thus determine a Feller Markov semigroup $(P(t))_{t\ge 0}$ on the boundary
$\Om(k,l)$.

The expression for $M_{z,z',w,w'\mid N}(\la)$ given in \S5.1 vanishes unless
$\la$ belongs to the subset $\gt_N(k,l)$, and it is strictly positive on this
subset. Thus, the same definition gives us a probability measure on
$\gt_N(k,l)$. The invariance property and the compatibility with the links
$\La^{N+1}_N$ remain valid. The limit measure lives on the boundary $\Om(k,l)$
and it is invariant with respect to the semigroup $(P(t))_{t\ge 0}$.

The whole picture sketched above is consistent with the automorphism of $\gt$
described in Remark \ref{2.0}. More precisely, the shift of all coordinates of
signatures by $1$ amounts to the transformation $k\to k+1$, $l\to l-1$ of the
main parameters (the other two parameters $a$ and $b$ do not change). The use
of this shift automorphism allows one to reduce the case of general parameters
$(k,l)$ to the special case $(n,0)$ with $n=k+l$, which simplifies some
formulas and computations.

Hence, let us now assume that $l=0$. Then the coordinates $\be^-_j$ disappear
and we are left with $n$ coordinates $y_i:=\be_i^+$ subjected to
$$
1\ge y_1\ge\dots\ge y_n\ge0.
$$

The signatures $\la\in\gt_N(n,0)$ may be identified with Young diagrams
contained in the rectangular shape $N\times n$; that is, $\la$ has at most $N$
rows and $n$ columns. Under this identification, one has a simple expression
for the link:
\begin{equation}\label{eqA.2}
\La^\infty_N(\om,\la)
=\di_N\la\cdot\prod_{i=1}^N(1-y_i)^Ns_{\la'}\left(\frac{y_1}{1-y_1},\dots,
\frac{y_n}{1-y_n}\right)
\end{equation}
where $s_{\la'}$ is the Schur polynomial indexed by the transposed diagram
$\la'$.

Further, the invariant measure on the boundary takes the form
\begin{equation}\label{eqA.3}
\const\cdot\prod_{1\le i<j\le n}(y_i-y_j)^2 \cdot\prod_{i=1}^n(1-y_i)^ay_i^b\,
dy_i
\end{equation}
with an appropriate constant prefactor that turns the measure into a
probability distribution. The random $n$-tuple $(y_1,\dots,y_n)\subset[0,1]$
with this distribution is known under the name of the {\it $n$-particle Jacobi
orthogonal polynomial ensemble\/}.

The fact that the integral of \eqref{eqA.2} against the distribution
\eqref{eqA.3} reproduces the measure $M_{n,n+a,0,b\mid N}$ on $\gt_N(n,0)$ can
be verified directly; this is a version of the Selberg integral.

For more detail, see \cite{Ker03} and \cite{BO05a}.

Since the boundary has finite dimension, there is a possibility to describe the
Markov process defined by $(P(t))_{t\ge 0}$ more directly; this is done in the theorem
below.

Consider the ordinary differential operator associated with the Jacobi
orthogonal polynomials with weight $(1-y)^ay^b$,
$$
D^{(a,b)}=y(1-y)\frac{d^2}{dy^2}+[b+1-(a+b+2)y]\frac{d}{dy}.
$$
More generally, abbreviate
$$
V_n=V_n(y_1,\dots,y_n)=\prod_{1\le i<j\le n}(y_i-y_j)
$$
and consider the partial differential operator in variables $y_1,\dots,y_n$
given by
\begin{gather*}
D^{(a,b)}_n :=\frac1{V_n}\circ\left(\sum_{i=1}^n
\left(y_i(1-y_i)\frac{\partial^2}{\partial
y_i^2}+[b+1-(a+b+2)y_i]\frac{\partial}{\partial y_i}\right)\right)\circ
V_n+(\cdots)\\
=\sum_{i=1}^n \left(y_i(1-y_i)\frac{\partial^2}{\partial
y_i^2}+\left[b+1-(a+b+2)y_i+\sum_{j:\, j\ne
i}\frac{2y_i(1-y_i)}{y_i-y_j}\right]\frac{\partial}{\partial y_i}\right),
\end{gather*}
where $(\cdots)$ stands for the constant annihilating the constant term arising
from the conjugation by the Vandermonde determinant:
$$
(\cdots)=\sum_{m=0}^{n-1}m(m+a+b+1).
$$

Although the coefficients in front of the first order derivatives have
singularities on the hyperplanes $y_i=y_j$, the operator is well defined on
smooth {\it symmetric\/} functions in variables $y_1,\dots,y_n$, and it
preserves this space. It also preserves the space of symmetric polynomials.

\begin{theorem}\label{A.1}
Let $z=n$, $z'=n+a$, $w=0$, $w'=b$, where $a,b>-1$. Let
$y_1=\be_1^+, \dots, y_n=\be_n^+$ be the coordinates on the boundary $\Om(n,0)$
of the truncated graph\/ $\gt(n,0)$, and recall that\/ $1\ge y_1\ge\dots\ge
y_n\ge0$.

The Markov process on the simplex\/ $\Om(n,0)$ determined by the Markov
semigroup $(P(t))_{t\ge 0}$ is a diffusion whose infinitesimal generator is the
differential operator $D^{(a,b)}_n$ with an appropriate domain containing the
space of all symmetric polynomials in variables $y_1,\dots,y_n$.
\end{theorem}

As shown in \cite{Gor09}, the same diffusion process also arises in a scaling
limit transition from some {\it discrete time\/} Markov chains on the sets
$\gt_N(n,0)$, as $N\to\infty$.

The above description of the infinitesimal generator remains true in the
general case \eqref{eqA.1} of the degenerate series parameters. The only change
concerns the correspondence between the $\be^\pm$-coordinates  and the
$y$-coordinates; now it takes the form
$$
(y_1,\dots,y_n)=(1-\be^-_l,\dots,1-\be^-_1, \be^+_1,\dots,\be^+_k), \qquad
n:=k+l.
$$

\subsection{The formal generator}\label{Formal generator}

A natural question is how to extend the explicit description of the
infinitesimal generator obtained in Theorem \ref{A.1} to the case of general
(admissible) values of parameters $(z,z',w,w')$. There are some indications
that a direct generalization is impossible, in the sense that the generator
cannot be expressed as a second order differential operator in the natural
coordinates $(\al^\pm,\be^\pm,\de^\pm)$ on $\Om$. \footnote{A possible
explanation is that these coordinate functions are not in the domain of the
generator.} Instead of this, we present below an explicit expression for the
generator in a different system of coordinates.

Introduce the functions $\varphi_n=\varphi_n(\om)$ on $\Om$, $n\in\Z$, as the
coefficients in the Laurent expansion of $\Phi_\om(u)$, see \eqref{eq2.2} and
\eqref{eq2.3},
$$
\Phi_\om(u)=\sum_{n\in\Z}\varphi_n(\om)u^n.
$$
The functions $\varphi_n$ are continuous and nonnegative on $\Om$, and they
satisfy the relation
$$
\sum_{n\in\Z}\varphi_n(\om)\equiv1, \qquad \om\in\Om,
$$
which is an immediate consequence of the fact that $\Phi_\om(u)$ takes value
$1$ at $u=1$ for all $\om\in\Om$.

Let us extend this definition by setting, for any $N=1,2,\dots$ and any
signature $\nu\in\gt_N$,
$$
\varphi_\nu(\om)=\det\bigl[\varphi_{\nu_i-i+j}(\om)\bigr]_{i,j=1}^N, \qquad
\om\in\Om.
$$
For instance, for $N=2$ we have $\nu=(\nu_1,\nu_2)$ and
$$
\varphi_{(\nu_1,\nu_2)}=\vmatrix \varphi_{\nu_1} & \varphi_{\nu_1+1}\\
\varphi_{\nu_2-1} & \varphi_{\nu_2}\endvmatrix =
\varphi_{\nu_1}\varphi_{\nu_2}-\varphi_{\nu_1+1}\varphi_{\nu_2-1}.
$$

The functions $\varphi_\nu$ appear in the expansion
$$
\prod_{i=1}^N\Phi_\om(u_i)=\sum_{\nu\in\gt_N}\varphi_\nu(\om)s_\nu(u_1,\dots,u_N).
$$
Their fundamental role is explained by \eqref{eq2.1}.

By definition, the functions $\varphi_\nu$ are contained in the algebra of
functions generated by the functions $\varphi_n$, $n\in\Z$. Conversely, any
monomial in $\varphi_n$'s of degree $N$ can be expanded into a series on the
functions $\varphi_\nu$, $\nu\in\gt_N$. Namely, for arbitrary integers
$n_1,\dots,n_N$, one has
$$
\varphi_{n_1}\dots\varphi_{n_N}=\sum_{\nu\in\gt_N}K(\nu\mid
n_1,\dots,n_N)\varphi_\nu,
$$
where the numbers $K(\dots)$ are defined as the ``rational'' analogs of the
Kostka numbers, that is, these are the coefficients in the Laurent expansion of
the rational Schur functions,
$$
s_\nu(u_1,\dots,u_N)=\sum_{n_1,\dots,n_N\in\Z}K(\nu\mid n_1,\dots,n_N)
u_1^{n_1}\dots u_N^{n_N}.
$$
For instance, for $N=2$ and $\nu=(\nu_1,\nu_2)$ we have
$$
K(\nu_1,\nu_2\mid n_1,n_2)=\begin{cases} 1, & \text{if $n_1+n_2=\nu_1+\nu_2$
and $|n_1-n_2|\le\nu_1-\nu_2$,}\\ 0, & \textrm{otherwise},
\end{cases}
$$
which implies
$$
\varphi_{n_1}\varphi_{n_2}=\sum_{p=0}^\infty\varphi_{(n_1+p,n_2-p)}.
$$

\begin{definition}\label{B.1}
Fix an arbitrary quadruple $(z,z',w,w')$ of complex
parameters and introduce the following formal differential operator in
countably many variables $\{\varphi_n:n\in\Z\}$
$$
\DD=\sum_{n\in\Z}A_{nn}\frac{\pd^2}{\pd\varphi_n^2}+2\sum_{\substack{n_1,n_2\in\Z\\
n_1>n_2}} A_{n_1 n_2}\frac{\pd^2}{\pd\varphi_{n_1}\pd\varphi_{n_2}}
+\sum_{n\in\Z}B_n\frac{\pd}{\pd\varphi_n}, \
$$
where, for any indices $n_1\ge n_2$,
$$
\gathered A_{n_1
n_2}=\sum_{p=0}^\infty(n_1-n_2+2p+1)(\varphi_{n_1+p+1}\varphi_{n_2-p}
+\varphi_{n_1+p}\varphi_{n_2-p-1})\\
-(n_1-n_2)\varphi_{n_1}\varphi_{n_2}
-2\sum_{p=1}^\infty(n_1-n_2+2p)\varphi_{n_1+p}\varphi_{n_2-p}
\endgathered
$$
and, for any $n\in\Z$,
$$
\gathered B_n=(n+w+1)(n+w'+1)\varphi_{n+1}+(n-z-1)(n-z'-1)\varphi_{n-1}\\
-\bigl((n-z)(n-z')+(n+w)(n+w')\bigr)\varphi_n.
\endgathered
$$
\end{definition}

\medskip

Note that only coefficients $B_n$ depend on the parameters $(z,z',w,w')$.

Assume now that the quadruple $(z,z',w,w')$ satisfies the condition
\eqref{4.1}. As above, let $P_N(t)$ ($N=1,2,\dots$) and $P(t)$ be the
corresponding Markov semigroups and let $A_N$ and $A$ stand for their
infinitesimal generators. These are densely defined operators in the Banach
spaces $C_0(\gt_N)$ and $C_0(\Om)$, respectively. We know that all finitely
supported functions on $\gt_N$ belong to the domain of $A_N$. This implies, cf.
\eqref{eq2.1}, that all the functions $\varphi_\nu$ lie in the domain of $A$.

\begin{theorem}\label{B.2}
Let $A$ be the infinitesimal generator of the Markov semigroup $(P(t))_{t\ge 0}$
with
parameters $(z,z',w,w')$ satisfying condition \eqref{eq4.1}, and $\DD$ be the formal
differential operator introduced in Definition \ref{B.1}. Regard $\DD$ as an
operator from the space $\C[\varphi_n: n\in\Z]$ of polynomials in countably
many variables $\varphi_n$, $n\in\Z$, to the larger space $\C[[\varphi_n:
n\in\Z]]$ of formal series in the same variables.

Then for any $N=1,2,\dots$ and any $\nu\in\gt_N$ one has
$$
\DD\varphi_\nu=A\varphi_\nu.
$$

Moreover, $\DD$ is the only formal second order differential operator in
variables $\varphi_n$, $n\in\Z$,  with such a property.
\end{theorem}

Note some properties of $\DD$:

1. Formal application of  $\DD$ to the infinite series $\sum_{n\in\Z}\varphi_n$
gives $0$. This agrees with the fact that the sum of this series on $\Om$
equals 1 and the fact that $A1=0$.

2. For any fixed integer $m$, $\DD$ is invariant under the change of variables
$\varphi_n\to \varphi_{n+m}$ ($n\in\Z$) combined with the shift of parameters
$$
z\to z+m, \quad z'\to z'+m, \quad w\to w-m, \quad w'\to w-m,
$$
cf. Remark 3.7 in \cite{BO05a}.

3. Set $z=k$ and $w=l$, where $(k,l)$ is a couple of nonnegative integers not
equal to $(0,0)$. Then $\DD$ respects the relations
$$
\dots=\varphi_{-l-2}=\varphi_{-l-1}=0=\varphi_{k+1}=\varphi_{k+2}=\dots
$$
and thus can be reduced to an operator in the polynomial algebra
$\C[\varphi_{-l},\dots,\varphi_k]$. More precisely, this means the following
assertion: When $\DD$ is applied to a monomial containing at least one variable
$\varphi_n$ with index $n$ outside $[-l,k]$ then all monomials entering the
resulting series with nonzero coefficients have the same property. Indeed, this
follows from the structure of the coefficients $A_{n_1n_2}$ and $B_n$.

Moreover, the resulting operator in the algebra
$\C[\varphi_{-l},\dots,\varphi_k]$ can be further reduced modulo the
relation
$$
\sum_{-l\le n\le k}\varphi_n=1,
$$
and then it coincides with the differential operator $D_{k+l}^{(a,b)}$ from Theorem
\ref{A.1}, where $a=z'-k$, $b=w'-l$. Here we use the fact that quotient algebra
$$
\C[\varphi_{-l},\dots,\varphi_k]/(\varphi_{-l}+\dots+\varphi_k=1)
$$
can be identified with the algebra of polynomial functions on the simplex
$\Om(k,l)$.


\begin{thebibliography}{Macd95}

\bibitem[And91]{And91}
W. J. Anderson, {\it Continuous time Markov chains: An applications-oriented
approach\/}. Springer, 1991.


\bibitem[Ask87]{Ask87}
R. Askey, {\it An integral of Ramanujan and orthogonal polynomials\/}. J.
Indian Math. Soc. {\bf51} (1987), 27--36.

\bibitem[Bor10+]{Bor10+}
A. Borodin, {\it Schur dynamics of the Schur processes\/}. arXiv:1001.3442

\bibitem[BD02]{BD02}
A. Borodin and P. Deift, {\it Fredholm determinants, Jimbo-Miwa-Ueno
$\tau$-functions, and representation theory\/}. Comm. Pure Appl. Math. {\bf55}
(2002), 1160--1230;  arXiv:math/0111007.

\bibitem[BF08]{BF08}
A. Borodin and P. L. Ferrari, {\it Large time asymptotics of growth models on
space-like paths I: PushASEP\/}. Electron. J. Probab. 13 (2008), 1380--1418;
arXiv:0707.2813

\bibitem[BF08+]{BF08+}
A. Borodin and P. L. Ferrari, {\it Anisotropic growth of random surfaces in 2+1
dimensions\/}. To appear; arXiv:0804.3035.

\bibitem[BG09]{BG09}
A. Borodin and V. Gorin, {\it Shuffling algorithm for boxed plane
partitions\/}. Adv. Math. {\bf220} (2009), 1739--1770; arXiv:0804.3071.

\bibitem[BGR09+]{BGR09+}
A. Borodin, V. Gorin, and E. M. Rains, {\it q-Distributions on boxed plane
partitions\/}. Selecta Math., to appear; arXiv:0905.0679.

\bibitem[BK10]{BK10}
A. Borodin and J. Kuan, {\it Random surface growth with a wall and Plancherel
measures for $O(\infty)$\/}, Commun. Pure Appl. Math. {\bf63} (2010), 831--894;
arXiv:0904.2607.


\bibitem[BO05a]{BO05a}
A. Borodin and G. Olshanski, {\it Harmonic analysis on the infinite-dimensional
unitary group and determinantal point processes\/}. Ann. Math. {\bf161} (2005),
1--104; arXiv:math/0109194.

\bibitem[BO05b]{BO05b}
A. Borodin and G. Olshanski, {\it Representation theory and random point
processes\/}. In: European Congress of Mathematics, Eur. Math. Soc., Zurich,
2005, pp. 73--94; arXiv:math/0409333.

\bibitem[BO05c]{BO05c}
A. Borodin and G. Olshanski, {\it Random partitions and the Gamma kernel\/}.
Adv. Math. {\bf194} (2005),141--202; arXiv:math-ph/0305043.

\bibitem[BO06a]{BO06a}
A.~Borodin and G.~Olshanski, {\it Markov processes on partitions\/}. Probab.
Theory  Rel. Fields {\bf135} (2006), 84--152; arXiv:math-ph/0409075.


\bibitem[BO06b]{BO06b}
A. Borodin and G. Olshanski, {\it Stochastic dynamics related to Plancherel
measure on partitions\/}. In: Representation Theory, Dynamical Systems, and
Asymptotic Combinatorics (V.~Kaimanovich and A.~Lodkin, eds). Amer. Math. Soc.
Translations, Series 2: Advances in the Mathematical Sciences, vol. {\bf217},
2006, pp. 9--21; arXiv:math-ph/0402064.


\bibitem[BO09]{BO09}
A.~Borodin and G.~Olshanski,  {\it Infinite-dimensional diffusions as limits of
random walks on partitions\/}. Probab. Theory Rel. Fields {\bf144} (2009),
281--318; arXiv:0706.1034.

\bibitem[DF90]{DF90}
P. Diaconis and J. A. Fill, {\it Strong stationary times via a new form of
duality\/}. Ann. Probab. {\bf18} (1990), 1483--1522.

\bibitem[EK86]{EK86}
S. N. Ethier and T. G. Kurtz, {\it Markov processes --- Characterization and
convergence\/}. Wiley--Interscience,  New York 1986.

\bibitem[Fel40]{Fel40}
W. Feller, {\it On the integro-differential equations of purely discontinuous
Markoff processes\/} Trans. Amer. Math. Soc., {\bf48} (1940), 488--815 and
Errata,  {\bf58} (1945) p. 474.

\bibitem[Fel57]{Fel57}
W. Feller, {\it On boundaries and lateral conditions for the Kolmogorov
differential equations\/}. Ann. Math. {\bf65} (1957), 527--570.

\bibitem[Fel59]{Fel59}
W. Feller, {\it The birth and death processes as diffusion processes\/}. J.
Math. Pures Appl. {\bf38} (1959), 301--345.


\bibitem[Gor09]{Gor09}
V. Gorin, {\it Noncolliding Jacobi processes as limits of Markov chains on the
Gelfand-Tsetlin graph\/}. J. Math. Sciences (New York) {\bf158} (2009), no. 6,
819--837 (translated from Zapiski Nauchnykh Seminarov POMI, Vol. 360 (2008),
pp. 91–123); arXiv:0812.3146.

\bibitem[Ito06]{Ito06}
K. It\^o, {\it Essentials of stochastic processes\/}. Translated from the 1957
Japanese original. Translations of Mathematical Monographs, 231. American
Mathematical Society, Providence, RI, 2006.

\bibitem[JM81]{JM81}
M. Jimbo and T. Miwa, {\it Monodromy preserving deformation of linear ordinary
differential equations with rational coefficients. II\/}.  Phys. D  {\bf2}
(1981), no. 3, 407--448.

\bibitem[KM59]{KM59}
S. Karlin and J. McGregor, {\it Coincidence probabilities\/}. Pacific J. Math.
{\bf9} (1959), 1141--1164.

\bibitem[Kat80]{Kat80}
T. Kato, {\it Perturbation theory of linear operators\/}, 2nd ed.
Springer-Verlag, New York, 1980.

\bibitem[KT09]{KT09}
M. Katori and H. Tanemura, {\it Zeros of Airy function and relaxation
process\/}. J. Stat. Phys. 136 (2009) 1177--1204; arXiv:0906.3666.

\bibitem[KT10]{KT10}
M. Katori and H. Tanemura, {\it Non-equilibrium dynamics of Dyson's model with
an infinite number of particles\/}. Commun. Math. Phys. {\bf293} (2010),
469--497; arXiv:0812.4108.

\bibitem[Kel83]{Kel83}
F. P. Kelly, {\it Invariant measures and the Q-matrix\/}. In:  Probability,
statistics and analysis,  London Math. Soc. Lecture Note Ser., 79, Cambridge
Univ. Press, Cambridge-New York, 1983, pp.  143--160s.

\bibitem[Ker03]{Ker03}
S. V. Kerov, {\it Asymptotic representation theory of the symmetric group and
its applications in analysis\/}.  Amer. Math. Soc., Providence, RI, 2003.


\bibitem[KLS10]{KLS10}
R. Koekoek, P. A. Lesky, and R. F. Swarttouw, {\it Hypergeometric orthogonal
polynomials and Their q-analogues\/}. Springer, 2010.

\bibitem[Kon05]{Kon05}
W. K\"onig, {\it Orthogonal polynomial ensembles in probability theory\/}.
Probab. Surveys {\bf2} (2005), 385--447.

\bibitem[Les97]{Les97}
P. A. Lesky, {\it Unendliche und endliche Orthogonalsysteme von continuous
Hahnpolynomen\/}. Results Math. {\bf31} (1997), 127--135.


\bibitem[Les98]{Les98}
P. A. Lesky, {\it Eine Charakterisierung der kontinuierlichen und diskreten
klassischen Orthogonalpolynome\/}. Preprint 98–12, Mathematisches Institut A,
Universit\"at Stuttgart.

\bibitem[Lig99]{Lig99}
T. Liggett, {\it Stochastic interacting systems: contact, voter and exclusion
processes\/}. Grundlehren der Mathematischen Wissenschaften 324.
Springer-Verlag, Berlin, 1999.

\bibitem[Lis09+]{Lis09+}
O. Lisovyy, {\it Dyson's constant for the hypergeometric kernel\/}.
arXiv:0910.1914.


\bibitem[Macd95]{Macd95}
I. G. Macdonald, {\it Symmetric functions and Hall polynomials\/}. 2nd edition.
Oxford University Press, 1995.

\bibitem[Mack57]{Mack57}
G. W. Mackey, {\it Borel structures on groups and their duals\/}. Trans. Amer.
Math. Soc. {\bf85} (1957), 134--165.

\bibitem[Mey66]{Mey66}
P.-A. Meyer, {\it Probability and potentials\/}. Blaisdell, 1966.

\bibitem[Nor10]{Nor10}
E. Nordenstam, {\it On the shuffling algorithm for domino tilings\/}. Electron.
J. Probab. {bf15} (2010), no. 3, 75--95; arXiv:0802.2592.


\bibitem[OO97]{OO97}
A. Okounkov and G. Olshanski, {\it Shifted Schur functions\/}. Algebra i Analiz
{\bf9} (1997), no. 2,  73--146 (Russian); English translation: St.~Petersburg
Math. J. {\bf 9} (1998), no.~2, 239--300; arXiv:q-alg/9605042.

\bibitem[OO98]{OO98}
A. Okounkov and G. Olshanski, {\it Asymptotics of Jack polynomials as the
number of variables goes to infinity\/}. Intern. Math. Res. Notices {\bf1998}
(1998), no. 13, 641--682; arXiv:q-alg/9709011.

\bibitem[Ols03]{Ols03}
G. Olshanski, {\it The problem of harmonic analysis on the infinite-dimensional
unitary group\/}. J. Funct. Anal. {\bf205} (2003), 464--524;
arXiv:math/0109193.

\bibitem[Ols10]{Ols10}
G. Olshanski, {\it Anisotropic Young diagrams and infinite-dimensional
diffusion processes with the Jack parameter\/}.  Intern. Math. Res. Notices
{\bf2010} (2010), no. 6, 1102--1166; arXiv:0902.3395.

\bibitem[Ols10+]{Ols10+}
G. Olshanski, {\it Laguerre and Meixner symmetric functions, and
infinite-dimensional diffusion processes\/}. arXiv:1009.????.

\bibitem[Osa09+]{Osa09+}
H. Osada, {\it Interacting Brownian motions in infinite dimensions with
logarithmic interaction potentials\/}. arXiv:0902.3561.

\bibitem[Pru63]{Pru63}
W. E. Pruitt, {\it Bilateral birth and death processes\/}. Trans. Amer. Math.
Soc. {\bf107} (1963), 508--525.

\bibitem[Spi70]{Spi70}
F. Spitzer, {\it Interaction of Markov processes\/}. Adv. Math.  {\bf5} (1970),
246--290.

\bibitem[Spo87]{Spo87}
H. Spohn, {\it Interacting Brownian particles: a study of Dyson's model\/}. In:
Hydrodynamic Behavior and Interacting Particle Systems, Papanicolaou, G. (ed),
IMA Volumes in Mathematics and its Applications, {\bf9}, Berlin:
Springer-Verlag, 1987, pp. 151--179.

\bibitem[VK81]{VK81}
A. M. Vershik and S. V. Kerov, {\it Asymptotic theory of characters of the
symmetric group\/}. Funct. Anal. Appl. {\bf15} (1981), 246--255.

\bibitem[VK82]{VK82}
A. M. Vershik and S. V. Kerov, {\it Characters and factor-representations of
the infinite unitary group\/}. Dokl. Akad. Nauk SSSR {\bf267} (1982), no. 2,
272--276 (Russian); English translation: Soviet Math. Dokl. {\bf26} (1982), no.
3, 570--574 (1983).

\bibitem[Voi76]{Voi76}
D. Voiculescu, {\it Repr\'esentations factorielles de type\/ {\rm II}$_1$ de
$U(\infty)$\/}. J. Math. Pures Appl. {\bf55} (1976), 1--20.

\bibitem[War07]{War07}
J. Warren, {\it Dyson's Brownian motions, intertwining and interlacing\/}.
Electron. J. Probab. {\bf12} (2007), 573--590; arXiv:math/0509720.

\bibitem[Wey39]{Wey39}
H. Weyl, {\it The classical groups. Their invariants and representations\/}.
Princeton Univ. Press, 1939; 1997 (fifth edition).

\bibitem[Yan90]{Yan90}
Xiangqun Yang, {\it The construction theory of denumerable Markov processes\/}.
Hunan Science and Technology Publ. House, 1990 (Wiley series in probability and
mathematical statistics : Probability and mathematical statistics).

\bibitem[Zhe70]{Zhe70}
D. P. Zhelobenko, {\it Compact Lie groups and their representations\/}, Nauka,
Moscow, 1970 (Russian); English translation: Transl. Math. Monographs 40, Amer.
Math. Soc., Providence, RI, 1973.




\end{thebibliography}
\end{document}